\documentclass[psamsfonts]{amsart}
\pdfoutput=1

\usepackage{amssymb,amsfonts}
\usepackage{amsmath}
\usepackage[all,arc]{xy}
\usepackage{enumerate}
\usepackage{mathrsfs}
\usepackage{graphicx}
\usepackage{tikz}
\usepackage{tikz-cd}
\usepackage{eufrak}
\usepackage[utf8]{inputenc}
\usepackage{mathtools}
\usepackage{stmaryrd}
\usepackage{enumitem}
\usepackage[T1]{fontenc} 
\usepackage{marginnote}

\usepackage{hyperref}
\hypersetup{hidelinks}

\newtheorem{thm}{Theorem}[section]
\newtheorem{cor}[thm]{Corollary}
\newtheorem{prop}[thm]{Proposition}
\newtheorem{lem}[thm]{Lemma}

\theoremstyle{definition}
\newtheorem{defn}[thm]{Definition}

\newtheorem{exmp}[thm]{Example}

\theoremstyle{remark}

\makeatletter
\let\c@equation\c@thm
\makeatother
\numberwithin{equation}{section}

\title[Quantum shuffle algebras and homology of type-\textit{B} Artin groups]{Fox-Neuwirth cells, quantum shuffle algebras,\\ and the homology of type-\textit{B} Artin groups}

\author{Anh Trong Nam Hoang}
\address{School of Mathematics, University of Minnesota, Minneapolis, MN 55455, USA}
\email{hoang278@umn.edu}

\date{\today}   

\begin{document}

\begin{abstract}
In this paper, we will develop a family of braid representations of Artin groups of type \textit{B} from braided vector spaces, and identify the homology of these groups with these coefficients with the cohomology of a specific bimodule over a quantum shuffle algebra. As an application, we give a complete characterization of the homology of type-\textit{B} Artin groups with coefficients in one-dimensional braid representations over a field of characteristic $0$. We will also discuss two different approaches to this computation: the first method extends a computation of the homology of braid groups due to Ellenberg--Tran--Westerland by means of induced representation, while the second method involves constructing a cellular stratification for configuration spaces of the punctured complex plane.
\end{abstract}

\maketitle




\tableofcontents


\section{Introduction}

Artin groups and Coxeter groups, particularly the classical Artin's braid groups, have been a rich object of study from a (co)homological standpoint. The cohomology of braid groups with trivial coefficients has been established in the 1970s by several contributors \cite{arn70,fuk70,coh1,coh2,vai78}, while the cohomology of pure braid groups was calculated in \cite{arn69}. For specific local coefficients, the (co)homology of braid groups in the (reduced) Burau representation has been determined by several authors \cite{dcps01,che17}, while Callegaro computed the homology of braid groups with coefficients in $R[q,q^{-1}]$, the ring of Laurent polynomials over a ring $R$, when $R = \mathbb{Z}$ or $R =k$ is a generic field \cite{cal06}. The case with coefficients in braid representations constructed from braided vector spaces was studied in \cite{etw17}.

For other Artin and Coxeter groups, few computations of their (co)homologies are known. The integral cohomology of Artin groups of types \textit{B} and \textit{D} was computed by Gorjunov \cite{gor78}, and those of Artin groups associated with exceptional Coxeter groups were determined by Salvetti \cite{sal94}. In \cite{dcpss}, De Concini, Procesi, Salvetti, and Stumbo calculated the cohomology of all finite-type Artin groups with coefficients in $\mathbb{Q}[q,q^{-1}]$, while the homology of type-\textit{B} Artin groups with coefficients in $\mathbb{Q}[q^{\pm 1}, t^{\pm 1}]$ was studied in \cite{cms08a,cms08b}. A general approach to the computation of the cohomology of Coxeter groups was detailed in \cite{sal02}. More recently, Boyd determined the second and third integral homology of an arbitrary finitely generated Coxeter group \cite{boy20}, while the second mod $2$ homology of an arbitrary Artin group was computed by Akita and Liu \cite{al18}. For a review of the (co)homologies of braid groups and groups connected with them, as well as other results on generalized braid groups, see \cite{frenkel,ver98,ver06,cm14,mar17}.

Recall that a \textit{braided vector space} over a field $k$ is a finite dimensional vector space $V$ equipped with an automorphism $\sigma: V \otimes V \to V \otimes V$ that satisfies the braid equation $(\sigma \otimes \mathrm{id}) \circ (\mathrm{id} \otimes \sigma) \circ (\sigma \otimes \mathrm{id}) = (\mathrm{id} \otimes \sigma) \circ (\sigma \otimes \mathrm{id}) \circ (\mathrm{id} \otimes \sigma)$ on $V^{\otimes 3}$. There is a natural action of the braid group $A_n$ \footnote{In this paper, we reserve the use of the notation $B_n$ for the $n^\mathrm{th}$ Artin group of type $B$. The $n^\mathrm{th}$ braid group will be denoted by $A_n$, due to its relation to the Artin groups of type $A$.} on $V^{\otimes n}$. Furthermore, we may define a braided, graded Hopf algebra named the \textit{quantum shuffle algebra} $\mathfrak{A}(V)$ whose underlying coalgebra is the cofree coalgebra on $V$ and multiplication is given by a shuffle product involving the braiding $\sigma$ (see Section~\ref{ssec:bvs_qsa}). Shuffle algebras have previously been used to express the cellular homology of configuration spaces and consequently the homology of braid groups with certain twisted coefficients \cite{fuk70,vai78,mar96,cal06,ks20}. However, very little is known about the quantum shuffle algebra and its cohomology. Some notable characterizations of this algebra can be found in \cite{dkkt97,ros98,leb13}, while an effort to bound its cohomology was detailed in \cite{etw17}.

The approach of this paper is inspired by the work of Ellenberg, Tran, and Westerland on the homology of braid groups with coefficients in monoidal braid representations \cite{etw17}. In particular, they identified the homology of the braid group $A_n$ with coefficients in $V^{\otimes n}$ with the cohomology of the quantum shuffle algebra $\mathfrak{A} = \mathfrak{A}(V^*_\epsilon)$; here $V^*_\epsilon$ is the dual vector space $V^*$ with braiding dual to that of $V$ and twisted by a sign.

\begin{thm}[\cite{etw17}]\label{thm:etw}
There is an isomorphism
\[\displaystyle H_q(A_n; V^{\otimes n}) \cong \mathrm{Ext}^{n-q,n}_{\mathfrak{A}} (k,k)\]
where the first index in the bigrading on $\mathrm{Ext}$ is the homological degree, and the second the internal degree. Furthermore, the natural multiplication on the braid homology is carried to the Yoneda product on $\mathrm{Ext}$; that is,
\[\displaystyle \bigoplus_{n=0}^{\infty} H_*(A_n; V^{\otimes n}) \cong \bigoplus_n \mathrm{Ext}^{n-*,n}_{\mathfrak{A}} (k,k)\]
is an isomorphism of bigraded rings.
\end{thm}

An application of this theorem with a suitable choice of $V$ proves an upper bound for the number of finite extensions of $\mathbb{F}_q(t)$, the field of rational functions over a finite field, with a given Galois group and conditions on the local monodromy at the ramified places of $\mathbb{F}_q(t)$ for sufficiently large $q$. As a consequence, the upper bound in Malle's conjecture \cite{mal02,mal04} over $\mathbb{F}_q(t)$ holds for all choices of the Galois group $G$ and all sufficiently large $q$ \cite{etw17}. This result is partially responsible for the motivation of this work.

In this paper, we will prove an analogue of Theorem~\ref{thm:etw} for the Artin groups of type \textit{B}, the subgroups of braid groups consisting of braids whose last strand is always pure. As an application, we will give a complete characterization of the homology of these Artin groups with coefficients in one-dimensional braid representations over a field of characteristic $0$.

\textit{Outline of the argument.} In Section~\ref{ssec:lbvs}, we will introduce the concept of a \textit{left-braided vector space} over $k$ to be a pair of finite dimensional vector spaces $(V,W)$ where $(V,\sigma)$ is a braided vector space, with an additional braiding $\tau: V \otimes W \to V \otimes W$ that along with $\sigma$ satisfies an additional braid equation $(\sigma \otimes \mathrm{id}) \circ (\mathrm{id} \otimes \tau) \circ (\sigma \otimes \mathrm{id}) \circ (\mathrm{id} \otimes \tau)  = (\mathrm{id} \otimes \tau) \circ (\sigma \otimes \mathrm{id}) \circ (\mathrm{id} \otimes \tau) \circ (\sigma \otimes \mathrm{id})$ on $V^{\otimes 2} \otimes W$. The main purpose of introducing this object is to provide a representation of the type-\textit{B} Artin group $B_n$ on $V^{\otimes n} \otimes W$.

A left-braided vector space $(V,W,\sigma,\tau)$ is called \textit{separable} if, roughly speaking, there exists a notion of ``square root'' for the braiding $\tau$.
Given a separable left-braided vector space $(V,W)$, we may define a bimodule $\mathfrak{M}(V,W)$ over the quantum shuffle algebra $\mathfrak{A}(V)$ by
\[ \mathfrak{M}(V,W) = \displaystyle \bigoplus_{q \ge 1} \bigoplus_{0 \le j \le q-1} V^{\otimes j} \otimes W \otimes V^{\otimes q-j-1} \]
with multiplication resembling the quantum shuffle product (see Definition~\ref{defn:M} for the specific formulae). The goal of the paper is to express the homology of the Artin groups of type \textit{B} with coefficients in the given braid representations as the cohomology of the bimodule $\mathfrak{M}=\mathfrak{M}(V^*_\epsilon,W^*)$ defined over the quantum shuffle algebra $\mathfrak{A} = \mathfrak{A}(V^*_\epsilon)$ as described above.

\begin{thm}\label{thm:main}
Given a separable left-braided vector space $(V,W)$, there is an isomorphism 
\[H_q (B_n; V^{\otimes n} \otimes W) \cong \mathrm{Ext}^{n-q, n+1}_{\mathfrak{A}^e} (\mathfrak{M},k)\]
where $\mathfrak{A}^e = \mathfrak{A} \otimes \mathfrak{A}^{op}$ is the enveloping algebra of $\mathfrak{A}$.
\end{thm}

Bigrading indices on Ext follow the notation in Theorem~\ref{thm:etw}. This result is proved in Section~\ref{ssec:main} by showing that the Fox-Neuwirth cellular chain complex of $\mathrm{Conf}_{n+1}(\mathbb{C})$ with coefficients in the local system associated with the $A_{n+1}$-representation induced by the $B_n$-representation $V^{\otimes n}\otimes W$ is isomorphic to the internal degree $n+1$ part of a chain complex $F_*(\mathfrak{M},\mathfrak{I})$ that computes the Tor modules of $\mathfrak{M}$ (here $\mathfrak{I}$ is the augmentation ideal of $\mathfrak{A}$). As an application of Theorem~\ref{thm:main}, we will compute the homology of the Artin group $B_n$ with coefficients in $V^{\otimes n} \otimes W$, where $V = W = k$ are fields of characteristic $0$ and the braidings are given by multiplications by arbitrary units.

Along the way, we will also discuss two different approaches to study the homology of type-\textit{B} Artin groups with coefficients in braid representations. The first approach utilizes the identification of the homology of braid groups in Theorem~\ref{thm:etw} by taking coefficients in the induced representation of any given type-\textit{B} braid representation (see Section~\ref{ssec:ind_rep} and Corollary~\ref{cor:etw_rewrite}). The second method involves constructing a cellular stratification for configuration spaces of the punctured complex plane, based on the Fox-Neuwirth/Fuks stratification of $\mathrm{Conf}_n(\mathbb{C})$ developed in \cite{fn62,fuk70} (see Section~\ref{sec:fnf}). A similar description may be found by utilizing the Salvetti complex, which provides a tool to compute the homology of spherical Artin groups with any coefficients \cite{sal87,sal94}; for the purpose of this paper, our stratification has the advantage of a strong geometric intuition and compatibility with the specific coefficient system of interest. These methods are shown to be equivalent and hence offer a flexible approach to the computation of the homology of Artin groups of type \textit{B}.

For instance, in a recent preprint \cite{h23_res}, we generalized and combined these techniques to study the homology of configuration spaces of the plane with multiple punctures and prove vanishing theorems for the homology of bicolor configuration spaces with coefficients arising from braided vector spaces and with certain one-dimensional twisted coefficients. As an arithmetic application, we produced an upper bound on character sums of the resultant over pairs of monic squarefree polynomials of given degrees. Constructed in a comparable generality with that of Ellenberg--Tran--Westerland \cite{etw17}, this new setup also provides a means to study variations of Malle's conjecture for function fields as well as other problems relying on counting rational points on covers of configuration spaces of a general curve, e.g., very recent work of Ellenberg and Landesman on the distribution of Selmer groups of quadratic twist families of abelian varieties over function fields \cite{ellenberg-landesman}.


\section{Shuffle algebras and representations of Artin groups}\label{sec:alg}

In this section, we will give a brief review of algebra from braid representations, mainly braided vector spaces and quantum shuffle algebras (their treatments in Section~\ref{ssec:bvs_qsa} are summarized from \cite{etw17}). We will also develop a family of braid representations for the Artin groups of type \textit{B} and study their induced representations of braid groups. Let $k$ be a field; unless otherwise noted, all tensor products will be over $k$.


\subsection{Braided vector spaces and quantum shuffle algebras}\label{ssec:bvs_qsa}
Recall that the $n^{\mathrm{th}}$ \textit{braid group} $A_n$ (equivalently the $n - 1^\mathrm{st}$ \textit{Artin group of type A}) is presented as
\[A_n = \langle \sigma_1, ..., \sigma_{n-1} : \sigma_i \sigma_j = \sigma_j \sigma_i \text{ if } |i-j| > 1; \sigma_i \sigma_{i+1} \sigma_i = \sigma_{i+1} \sigma_i \sigma_{i+1} \rangle.\]

The groupoid $\mathscr{A} = \sqcup_{n \ge 0} [*/{A_n}]$ of all braid groups has the structure of a braided monoidal category. The family of strictly monoidal functors $\Phi : \mathscr{A} \to \mathrm{FinVect}_k$ forms the category of \textit{monoidal braid representations}, where $\mathrm{FinVect}_k$ denotes the category of finite dimensional $k$-vector spaces.

\begin{defn}
A \textit{braided vector space} $V$ over $k$ is a finite dimensional $k$-vector space equipped with an invertible \textit{braiding} $\sigma : V \otimes V \to V \otimes V$ such that it satisfies the braid equation on $V^{\otimes 3}$:
\[(\sigma \otimes \mathrm{id}) \circ (\mathrm{id} \otimes \sigma) \circ (\sigma \otimes \mathrm{id}) = (\mathrm{id} \otimes \sigma) \circ (\sigma \otimes \mathrm{id}) \circ (\mathrm{id} \otimes \sigma).\]
\end{defn}

Braided vector spaces form a category where morphisms $(V_1, \sigma_1) \to (V_2,\sigma_2)$ are $k$-linear maps $f:V_1 \to V_2$ that satisfy $(f \otimes f) \circ \sigma_1 = \sigma_2 \circ (f \otimes f)$. There is a natural action of $A_n$ on $V^{\otimes n}$ defined by $\sigma_i \mapsto \mathrm{id}^{\otimes i-1} \otimes \sigma \otimes \mathrm{id}^{\otimes n-i-1}$.

\begin{prop}[\cite{etw17}]
There is a pair of inverse equivalences between the categories of monoidal braid representations and braided vector spaces that send $\Phi$ to $\Phi(1)$ and $V$ to the braid representation on $V^{\otimes n}$ discussed above.
\end{prop}

Recall that an $(n,m)$-\textit{shuffle} $\gamma : \{1,...,n\} \sqcup \{1,...,m\} \to \{1,...,n+m\}$ is a bijection that preserves orders on both $\{1,...,n\}$ and $\{1,...,m\}$. Alternatively, an $(n,m)$-shuffle is a permutation in $S_{n+m}$ that preserves order on the first $n$ and the last $m$ elements. Given any shuffle $\gamma$, there is a choice of a lift $\tilde{\gamma} \in A_{n+m}$ given by the braid that shuffles the endpoints according to $\gamma$ by moving the right $m$ strands in front of the left $n$ strands (see Figure~\ref{fig:shuf_lift}).

\begin{figure}[t]
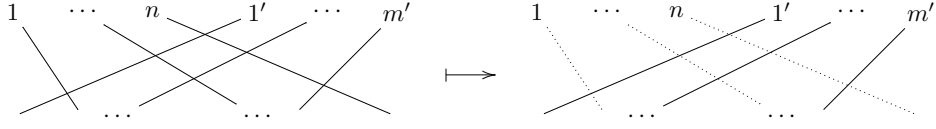

    \centering
    \[ 
\resizebox{\textwidth}{!}{
	\xy
		{\ar@{-} (0,0)*+{}; (35,15)*+{1'} };
		{\ar@{-} (15,0)*+{\cdots}; (45,15)*+{\cdots} }; 
		{\ar@{-} (40,0)*+{}; (55,15)*+{m'} }; 
		{\ar@{-} (10,0)*+{}; (0,15)*+{1} }; 
		{\ar@{-} (35,0)*+{\cdots}; (10,15)*+{\cdots} }; 
		{\ar@{-} (55,0)*+{}; (20,15)*+{n} }; 
		{\ar@{|->} (62,6)*{}; (69,6)*{}};
		{\ar@{-} (75,0)*+{}; (110,15)*+{1'} };
		{\ar@{-} (90,0)*+{\cdots}; (120,15)*+{\cdots} }; 
		{\ar@{-} (115,0)*+{}; (130,15)*+{m'} }; 
		{\ar@{..} (85,0)*+{}; (75,15)*+{1} }; 
		{\ar@{..} (110,0)*+{\cdots}; (85,15)*+{\cdots} }; 
		{\ar@{..} (130,0)*+{}; (95,15)*+{n} }; 
	\endxy
}
\]
    \caption{Lifting an $(n,m)$-shuffle to a braid.}
    \label{fig:shuf_lift}
\end{figure}

Let $(V,\sigma)$ be a braided vector space. We will write elements of $V^{\otimes n}$ using bar complex notation, i.e. $[a_1|...|a_n]$.

\begin{defn}
The \textit{quantum shuffle algebra} $\mathfrak{A}(V)$ is a braided, graded bialgebra: its underlying coalgebra is the tensor coalgebra
\[ T^{co}(V) = \displaystyle \bigoplus_{n \ge 0} V^{\otimes n}\]
equipped with a multiplication given by the quantum shuffle product:
\[\displaystyle [a_1 | ... | a_n] \star [b_1 | ... | b_m] = \sum_{\gamma} \tilde{\gamma} [a_1 | ... | a_n | b_1 | ... | b_m]\]
where the sum is over all $(n,m)$-shuffles $\gamma$.
\end{defn}

The quantum shuffle algebra has the structure of a Hopf algebra in a braided monoidal category. For more specific properties of this algebra, see \cite{mil58,ros98,leb13,ksv14,etw17}.


\subsection{Left-braided vector spaces}\label{ssec:lbvs}

The purpose of this subsection is to develop a representation for the Artin groups of type \textit{B}, using an analogue of the braided vector spaces. Recall that the $n^\mathrm{th}$ \textit{Artin group of type B} is presented by
\[B_n = \left\langle \sigma_1, ..., \sigma_{n-1},\tau_n : \begin{array}{l}
     \sigma_i \sigma_j = \sigma_j \sigma_i \text{ if } |i-j| > 1 \\
     \sigma_i \tau_n = \tau_n \sigma_i\text{ and } \sigma_i \sigma_{i+1} \sigma_i = \sigma_{i+1} \sigma_i \sigma_{i+1} \text{ if } i \ne n-1 \\
     \sigma_{n-1} \tau_n \sigma_{n-1} \tau_n = \tau_n \sigma_{n-1} \tau_n \sigma_{n-1}
\end{array}
\right\rangle.\]
There is a natural inclusion of $B_n$ into the braid group $A_{n+1}$ that sends $\sigma_i$ to the corresponding $\sigma_i$ in $A_{n+1}$ for all $1\le i \le n-1$ and sends $\tau_n$ to $\sigma^2_n$. This identification presents an interesting geometric interpretation of the Artin group $B_n$ that proves to be very useful for the approach of this paper.

\begin{prop}[\cite{cri99}]
$B_n$ is isomorphic to the finite index subgroup of the $(n+1)$-strand braid group $A_{n+1}$ consisting of braids whose the last strand is pure, i.e. the $n+1^\mathrm{st}$ endpoint is connected to itself in the braid.
\end{prop}

Given a representation of $A_{n+1}$, we may obtain a representation of the subgroup $B_n$ by means of the restricted representation. Let $(V,\sigma)$ be a braided vector space. Recall that we have an $A_{n+1}$-representation on $V^{\otimes n+1}$, which we can restrict to a $B_n$-representation on the same vector space. The action of the generator $\sigma_i$ of $B_n$ for all $1 \le i \le n-1$ is the same as that of the corresponding $\sigma_i$ of $A_{n+1}$, while the last generator $\tau_n$ acts on $V^{\otimes n+1}$ by squaring the action of $\sigma_n$, i.e. $\tau_n \mapsto \mathrm{id}^{\otimes n-1} \otimes \sigma^2$. If we restrict this action on the last two tensor factors, it is clearly given by a mapping $\tau := \sigma^2 : V \otimes V \to V \otimes V$ that preserves the order of the factors; furthermore, in this restricted representation, there is no well-defined action of the group $B_n$ that applies the braiding $\sigma$ individually to these two factors. In other words, we lose information about the action of the generator $\sigma_n$ of $A_{n+1}$ on $V^{\otimes n+1}$, which can only be recovered partially as the ``square root'' of the action of $\tau_n \in B_n$ on the same space. We will generalize this restricted representation into a family of $B_n$-representations based on the observations above.

\begin{defn}
A \textit{left-braided vector space} $(V, W)$ over $k$ is a pair of finite dimensional $k$-vector spaces $V$ and $W$, where $V$ is a braided vector space with a braiding $\sigma$, further equipped with another isomorphism $\tau : V \otimes W \to V \otimes W$ such that it satisfies an additional braid equation on $V^{\otimes 2} \otimes W$:
\[(\sigma \otimes \mathrm{id}) \circ (\mathrm{id} \otimes \tau) \circ (\sigma \otimes \mathrm{id}) \circ (\mathrm{id} \otimes \tau)  = (\mathrm{id} \otimes \tau) \circ (\sigma \otimes \mathrm{id}) \circ (\mathrm{id} \otimes \tau) \circ (\sigma \otimes \mathrm{id}).\]
\end{defn}

We may define a morphism between left-braided vector spaces $(V_1,W_1,\sigma_1,\tau_1)$ and $(V_2,W_2,\sigma_2,\tau_2)$ to be a pair of $k$-linear maps $f_V :V_1 \to V_2$ and $f_W : W_1\to W_2$ where $f_V$ is a morphism of braided vector spaces $(V_1,\sigma_1) \to (V_2, \sigma_2)$ and $f_W$ satisfies the relation: $(f_V \otimes f_W) \circ \tau_1 = \tau_2 \circ (f_V \otimes f_W)$ on $V_1 \otimes W_1$. The collection of left-braided vector spaces then forms a category. Similar to the case of braided vector spaces, we may define an action of $B_n$ on $V^{\otimes n} \otimes W$ by $\sigma_i \mapsto \mathrm{id}^{\otimes i-1} \otimes \sigma \otimes \mathrm{id}^{n-i}$ for all $1 \le i \le n-1$ and $\tau_n \mapsto \mathrm{id}^{\otimes n-1} \otimes \tau$. From this identification, the following is straight-forward:

\begin{prop}\label{prop:Bn_rep}
Given a left-braided vector space $(V,W,\sigma,\tau)$, $V^{\otimes n} \otimes W$ provides a representation for the Artin group $B_n$.
\end{prop}

\begin{exmp}\label{exmp:V=W=k}
Given $V=W=k$, we can define a left-braided vector space $(V,W)$ with braidings $\sigma$ and $\tau$ given by multiplications by $q$ and $p$ respectively, for some $p, q \in k^\times$. The braid action of $B_n$ on the representation $V^{\otimes n} \otimes W \cong k$ is therefore given by $\sigma_i \mapsto q$ for all $1 \le i \le n-1$ and $\tau_n \mapsto p$.
\end{exmp}

\begin{exmp}\label{exmp:V=bvs}
If $(V,\sigma)$ is a braided vector space, then $(V,V,\sigma,\sigma^2)$ forms a left-braided vector space. In this case, there is an obvious choice for the ``square root'' of the braiding $\tau = \sigma^2$; generally, this is not the case. We will discuss this matter in Section~\ref{ssec:ind_rep}.

The $B_n$-representation constructed from this left-braided vector space per Proposition~\ref{prop:Bn_rep} is precisely the restricted representation to $B_n$ of the previously described $A_{n+1}$-representation on $V^{\otimes n+1}$.
\end{exmp}

Recall that there is a bijection between the category of monoidal functors $\Phi: \mathscr{A} \to \mathrm{FinVect}_k$ and the category of braided vector spaces. There is a similar functorial description for the category of left-braided vector spaces. Define the category $\mathscr{B}$ of type-\textit{B} Artin groups to be the wide subcategory of the groupoid $\mathscr{A}$ of braid groups with morphisms given only by their type-\textit{B} subgroups: the objects $n$ of $\mathscr{B}$ are indexed by positive integers, and the morphisms in $\mathscr{B}$ are automorphisms of $n$ given by the group $B_{n-1}$, i.e. $\mathrm{Hom}_\mathscr{B} (n,n) = B_{n-1}$. There is a tensor product $\mathscr{A} \times \mathscr{B} \to \mathscr{B}$ induced by the homomorphism $A_n \times B_m \to B_{n+m}$ that places braids side-by-side. It is easy to see that this tensor product must agree with the tensor product in the groupoid $\mathscr{A}$ via the inclusion map $\mathscr{B} \hookrightarrow \mathscr{A}$, i.e. the diagram
\[ \begin{tikzcd}
\mathscr{A} \times \mathscr{B} \arrow[hookrightarrow]{r} \arrow[swap]{d}{\otimes} & \mathscr{A} \times \mathscr{A} \arrow{d}{\otimes} \\%
\mathscr{B} \arrow[hookrightarrow]{r}& \mathscr{A}
\end{tikzcd} \]
commutes (up to natural isomorphism), i.e. $\mathscr{B}$ is a left tensor ideal in $\mathscr{A}$.

Given a left-braided vector space $(V,W,\sigma,\tau)$, while a monoidal functor $\Phi : \mathscr{A} \to \mathrm{FinVect}_k$ is enough to capture all data of the braided vector space $(V,\sigma)$, we need an additional functor $\Psi:\mathscr{B} \to \mathrm{FinVect}_k$ to capture the information of the vector space $W$ and the braiding $\tau$. In addition, these functors must be compatible with the tensor product $\mathscr{A} \times \mathscr{B} \to \mathscr{B}$ discussed above. These observations lead to the following identification of the category of left-braided vector spaces.

\begin{prop}
There is an equivalence of categories between the category of left-braided vector spaces and the category $\mathscr{F}$ of pairs of functors $\Phi:\mathscr{A} \to \mathrm{FinVect}_k$ and $\Psi:\mathscr{B} \to \mathrm{FinVect}_k$ that satisfy the following conditions:
\begin{enumerate}
    \item $\Phi$ is a monoidal functor; and
    \item The diagram
    \[ \begin{tikzcd}
    \mathscr{A} \times \mathscr{B} \arrow{r}{\Phi \times \Psi} \arrow[swap]{d}{\otimes} & \mathrm{FinVect}_k \times \mathrm{FinVect}_k \arrow{d}{\otimes} \\%
    \mathscr{B} \arrow{r}{\Psi} & \mathrm{FinVect}_k
    \end{tikzcd} \]
    commutes (up to natural isomorphism).
\end{enumerate}
\end{prop}

\begin{proof}
It is not hard to see that the collection of such pairs of functors $(\Phi,\Psi)$ forms a well-defined category when equipped with pairs of natural transformations as morphisms. Given a functor pair $(\Phi,\Psi) \in \mathscr{F}$, we obtain a left-braided vector space by setting $V = \Phi(1)$ and  $W = \Psi(1)$. The braiding morphism $\sigma$ is the image under $\Phi$ of the positive generator of $\mathrm{Hom}_\mathscr{A}(2,2) = A_2 \cong \mathbb{Z}$, while $\tau$ is obtained by applying $\Psi$ to the positive generator of $\mathrm{Hom}_\mathscr{B} (2,2) = B_1 \cong \mathbb{Z}$.
Condition (1) in the proposition enforces that $(V,\sigma)$ is a braided vector space by Proposition 2.2; meanwhile, condition (2) maintains that the tensor product in the representation is compatible with the tensor product in the categories $\mathscr{A}$ and $\mathscr{B}$. Conversely, given an arbitrary left-braided vector space $(V,W,\sigma,\tau)$, we have constructed above an $A_n$-representation on $V^{\otimes n}$ and a $B_m$-representation on $V^{\otimes m} \otimes W$. It is straightforward to verify that these identifications are inverses (up to natural isomorphism) and hence form a pair of inverse equivalences between the category of left-braided vector spaces and the category $\mathscr{F}$ of functor pairs $(\Phi,\Psi)$ that satisfy conditions (1-2).
\end{proof}


\subsection{Induced representation of braid groups}\label{ssec:ind_rep}

Recall that there is a natural inclusion of $B_n$ into the braid group $A_{n+1}$ that identifies elements of $B_n$ with braids of $n+1$ strands where the last strand is pure. Consider the left cosets of $B_n$ in $A_{n+1}$.

\begin{prop}
The collection of left cosets of $B_n$ in $A_{n+1}$ has the form
\[A_{n+1}/B_n = \{\alpha_{i,n+1} B_n : \alpha_{i,n+1} = \sigma_i ... \sigma_{n-1} \sigma_n \text{ for } 1 \le i \le n \text{ or }\mathrm{id} \text{ for } i = n+1\}.\]
\end{prop}

\begin{proof}
We claim that the left cosets $a B_n$ are indexed by the image of the $n+1^\mathrm{st}$ endpoint under the braid $a \in A_{n+1}$, and therefore $\left| A_{n+1}:B_n \right| = n+1$. This statement, particularly the second fact, was proved by Crisp in slightly different language \cite{cri99}. For any $a_1, a_2 \in A_{n+1}$, $a_1 B_n = a_2 B_n$ as cosets iff $a_1^{-1} a_2 \in B_n$. Let $\underline{a}$ denote the underlying permutation of a braid $a$, then this is equivalent to $\underline{a_1}^{-1} \underline{a_2} (n+1) = n+1$, or $\underline{a_1}(n+1) = \underline{a_2}(n+1)$. So we have a simple characterization of the cosets of $B_n$ in $A_{n+1}$: two braid elements of $A_{n+1}$ are in the same coset of $B_n$ if and only if their underlying permutations map $n+1$ to the same number. Since there are $n+1$ choices for the image, the index of $B_n$ in $A_{n+1}$ is $n+1$. Furthermore, we may explicitly choose representatives for the cosets of $B_n$ to be $\alpha_{i,n+1} = \sigma_i ... \sigma_{n-1} \sigma_n$ for $1 \le i \le n$ and $\alpha_{n+1,n+1} = \mathrm{id}$. Alternatively, for $1 \le i \le n+1$, the representative element $\alpha_{i,n+1}$ is the lift to $A_{n+1}$ (as described in Section~\ref{ssec:bvs_qsa}) of the $(n,1)$-shuffle that sends $n+1$ to $i$.
\end{proof}

The second index of a representative element $\alpha_{i,n+1}$ records the number of strands in the braid; when this datum is unambiguous it is omitted from the notation.

Given a representation of any subgroup, we may define a representation of the parent group by means of the \textit{induced representation}. Let $L$ be a representation of $B_n$. The braid representation of $B_n$ on $L$ induces a representation on
\[\mathrm{Ind}^{A_{n+1}}_{B_n} (L) = k[A_{n+1}] \displaystyle \otimes_{k[B_n]} L \]
of the braid group $A_{n+1}$. We may give a more detailed description of this induced representation based on the cosets of the subgroup $B_n$ in $A_{n+1}$ described above. Since the collection $\{\alpha_1, ..., \alpha_{n+1} \}$ gives a full set of representatives in $A_{n+1}$ for the left cosets of $B_n$, as vector spaces, the induced representation can be identified as
\[\mathrm{Ind}^{A_{n+1}}_{B_n} (L) \cong \displaystyle \bigoplus_{i=1}^{n+1} \alpha_i L.\]
Here each $\alpha_i L$ is an isomorphic copy of the vector space $L$ whose elements are written as $\alpha_i \ell$ where $\ell \in L$. We may give a concrete description of the action of the braid group on this induced representation.

\begin{prop}\label{prop:act_ind_rep}
The action of the braid group $A_{n+1}$ on the induced representation $\mathrm{Ind}^{A_{n+1}}_{B_n} (L)$ is given by
\[a \sum^{n+1}_{i=1} \alpha_i \ell_i = \sum^{n+1}_{i=1} \alpha_{\underline{a}(i)} \big[ (\alpha_{\underline{a}(i)}^{-1} a \alpha_i) (\ell_i) \big]\]
where $(\alpha_{\underline{a}(i)}^{-1} a \alpha_i) (\ell_i)$ is obtained by applying the action of $\alpha_{\underline{a}(i)}^{-1} a \alpha_i \in B_n$ on $\ell_i \in L$ for every $1 \le i \le n+1$.
\end{prop}

\begin{proof}
Since the collection $\{ \alpha_i \}^{n+1}_{i=1}$ forms a full set of representatives, for each $a \in A_{n+1}$ and each $\alpha_i$, there exist $b_i \in B_n$ and $\alpha_j$ such that $a \alpha_i = \alpha_j b_i$. The action of $a$ on an element in the induced representation is defined by
\[ a \sum^{n+1}_{i=1} \alpha_i \ell_i = \sum^{n+1}_{i=1} \alpha_j [b_i (\ell_i)] \]
where the action of $b_i$ on $\ell_i$ is defined by the $B_n$-representation $L$. In this case, we can make specific choices for the elements $\alpha_j$ and $b_i$. Since the underlying permutation of $b_i = \alpha_j^{-1} a \alpha_i \in B_n$ preserves $n+1$, it follows that $\underline{\alpha_j} (n+1) = \underline{a} \big[\underline{\alpha_i}(n+1)\big] = \underline{a} (i)$, thus $j = \underline{a}(i)$. Hence the choices of $\alpha_j = \alpha_{\underline{a}(i)}$ and $b_i = \alpha_{\underline{a}(i)}^{-1} a \alpha_i$ for each summand in the element of the representation give the desired action of the braid group $A_{n+1}$ on the induced representation $\mathrm{Ind}^{A_{n+1}}_{B_n} (L)$.
\end{proof}

\begin{cor}
The action of the generators of $A_{n+1}$ on $\mathrm{Ind}^{A_{n+1}}_{B_n}(L)$ can be expressed in terms of the action of the generators of $B_n$ in the following way:
\[ \sigma_m(\alpha_i \ell) = \begin{cases}
\alpha_i [\sigma_m (\ell)] \hspace{1.4in} \hfill 1 \le m \le i-2 \\
\alpha_{i-1} \ell \hfill m = i-1 \\
\alpha_{i+1} [(\alpha_i \tau_n \alpha_i^{-1})(\ell)] \hfill m = i \\
\alpha_i [\sigma_{m-1} (\ell)] \hfill i+1 \le m \le n+1.
\end{cases}
\]
\end{cor}

Strictly speaking, the element of $B_n$ presented in the formula when $m = i$ is not entirely in terms of the generators of $B_n$, since $\alpha_i$ contains the braid element $\sigma_n \in A_{n+1}\backslash{B_n}$. However, it can be rewritten as $\alpha_i \tau_n \alpha_i^{-1} = \sigma_i ... \sigma_n \sigma_n^2 \sigma_n^{-1} ... \sigma_i^{-1} = \sigma_i ... \sigma_{n-1} \sigma_n^2 \sigma_{n-1}^{-1} ... \sigma_i^{-1} = \sigma_i ... \sigma_{n-1} \tau_n \sigma_{n-1}^{-1} ... \sigma_i^{-1} \in B_n$.

\begin{proof}
We will prove this statement case-by-case.

\textit{Case 1 ($1 \le m \le i-2$):} This case follows directly from the commutativity of non-adjacent braid generators.

\textit{Case 2 ($m = i-1$):} Observe that as braid elements, $\alpha_{i-1} = \sigma_{i-1} \alpha_i$. Hence $\sigma_{i-1}(\alpha_i\ell) = \alpha_{i-1} [(\alpha_{i-1}^{-1} \sigma_{i-1} \alpha_i)(\ell)] = \alpha_{i-1} \ell$.

\textit{Case 3 ($m = i$):} It is equivalent to prove that $\alpha_{i+1}^{-1} \sigma_i \alpha_i = \alpha_i \sigma_n^2 \alpha_i^{-1}$ as braid elements for all $i$. We will prove this case by induction on $i$ with the starting index $i = n$. The base case is simple: $\alpha_{n+1}^{-1} \sigma_n \alpha_n = (\mathrm{id})\sigma_n \sigma_n = \sigma_n^2$.

Suppose the hypothesis holds for all $k \ge i$, i.e. $\alpha_{k+1}^{-1} \sigma_k \alpha_k = \alpha_k \sigma_n^2 \alpha_k^{-1}$. To show that it holds for $i-1$, first we apply the statement of case 2 to the right side: $\alpha_{i-1} \sigma_n^2 \alpha_{i-1}^{-1} = (\sigma_{i-1} \alpha_i) \sigma_n^2 (\alpha_i^{-1} \sigma_{i-1}^{-1})$. By the induction hypothesis, it follows that
\[ \begin{array}{l}
\sigma_{i-1} (\alpha_i \sigma_n^2 \alpha_i^{-1}) \sigma_{i-1}^{-1} = \sigma_{i-1} (\alpha_{i+1}^{-1} \sigma_i \alpha_i) \sigma_{i-1}^{-1} = \alpha_{i+1}^{-1} \sigma_{i-1} \sigma_i (\sigma_i \alpha_{i+1}) \sigma_{i-1}^{-1}\\[5pt]
= \alpha_i^{-1} (\sigma_i \sigma_{i-1} \sigma_i) \sigma_i \sigma_{i-1}^{-1} \alpha_{i+1} = \alpha_i^{-1} \sigma_{i-1} (\sigma_i \sigma_{i-1} \sigma_i) \sigma_{i-1}^{-1} \alpha_{i+1}\\[5pt]
= \alpha_i^{-1} \sigma_{i-1} \sigma_{i-1} \sigma_i \sigma_{i-1} \sigma_{i-1}^{-1} \alpha_{i+1} = \alpha_i^{-1} \sigma_{i-1} (\sigma_{i-1} \sigma_i \alpha_{i+1}) = \alpha_i^{-1} \sigma_{i-1} \alpha_{i-1}.
\end{array}\]

\textit{Case 4 ($i+1 \le m \le n+1)$:} We will use another induction argument on $i$ with the starting index $i = n$. The base case $i = n$ is vacuously true. For $i = n-1$, we only need to check when $m = n$; this case follows directly from the braid relation in $A_{n+1}$: $\sigma_n (\alpha_{n-1}\ell) = \alpha_{n-1}[(\alpha_{n-1}^{-1} \sigma_n \alpha_{n-1})(\ell)] = \alpha_{n-1} [(\sigma_n^{-1} \sigma_{n-1}^{-1} \sigma_n \sigma_{n-1} \sigma_n)(\ell)] = \alpha_{n-1} [(\sigma_n^{-1} \sigma_{n-1}^{-1} \sigma_{n-1} \sigma_n \sigma_{n-1})(\ell)] = \alpha_{n-1}[\sigma_{n-1}(\ell)].$

Suppose the statement holds true for all $k \ge i$, i.e. $\sigma_m (\alpha_k \ell) = \alpha_k [\sigma_{m-1}(\ell)]$ for all $k+1 \le m \le n+1$. We will make use of the fact that for all $k$, $\alpha_{k-1}\ell = \sigma_{k-1} (\alpha_k \ell)$ as proved in case 2. For $m \ge i+1$, $\sigma_m (\alpha_{i-1} \ell) = \sigma_m \sigma_{i-1} (\alpha_i \ell) = \sigma_{i-1} \sigma_m (\alpha_i \ell) = \sigma_{i-1} \alpha_i [\sigma_{m-1} (\ell)] = \alpha_{i-1} [\sigma_{m-1} (\ell)].$ For $m=i$, $\sigma_i (\alpha_{i-1} \ell) = \sigma_i \sigma_{i-1} \sigma_i (\alpha_{i+1} \ell) = \sigma_{i-1} \sigma_i \sigma_{i-1} (\alpha_{i+1} \ell) = \sigma_{i-1}\sigma_i \alpha_{i+1}[\sigma_{i-1}(\ell)] = \alpha_{i-1}[\sigma_{i-1}(\ell)].$

We have verified the action of the generators of $A_{n+1}$ on $\mathrm{Ind}^{A_{n+1}}_{B_n} (L)$ for all cases, hence the proof is finished. 
\end{proof}

This corollary highlights a benefit of our choice of the coset representatives $\{\alpha_i\}^{n+1}_{i=1}$. Since the composition of braid actions works well with the braid multiplication, the fact that we understand the action of the generators of $A_{n+1}$ in terms of the action of the generators of $B_n$ implies that it is possible to decompose a general braid action on $\mathrm{Ind}^{A_{n+1}}_{B_n} (L)$ into a series of actions of the generators of the subgroup $B_n$ on the original representation $L$.

Consider when $L = V^{\otimes n} \otimes W$, the representation of $B_n$ formed from the left-braided vector space $(V,W,\sigma,\tau)$ as discussed in the previous section. Recall that the induced representation $\mathrm{Ind}^{A_{n+1}}_{B_n} (V^{\otimes n}\otimes W)$ can be written as the direct sum $\bigoplus_{i=1}^{n+1} \alpha_{i,n+1} (V^{\otimes n}\otimes W)$ where each $\alpha_{i,n+1} (V^{\otimes n}\otimes W)$ is isomorphic to $V^{\otimes n}\otimes W$. For all $1 \le i \le n+1$, since the underlying permutation of $\alpha_{i,n+1}$ sends $n+1$ to $i$, it is natural to identify $\alpha_{i,n+1} (V^{\otimes n}\otimes W)$ with $V^{\otimes i-1} \otimes W \otimes V^{\otimes n-i+1}$ (as vector spaces), where $W$ is the $i^\mathrm{th}$ tensor factor, via an isomorphism $\xi_{i,n+1} : \alpha_{i,n+1} (V^{\otimes n}\otimes W) \xrightarrow{\cong} V^{\otimes i-1} \otimes W \otimes V^{\otimes n-i+1}$. This yields the following identification:

\begin{prop}\label{prop:ind_rep_lbvs}
There is an isomorphism of vector spaces
\[ \mathrm{Ind}^{A_{n+1}}_{B_n} (V^{\otimes n}\otimes W) \cong \displaystyle \bigoplus_{i=1}^{n+1} V^{\otimes i-1} \otimes W \otimes V^{\otimes n-i+1}.\]
Moreover, given a choice of isomorphism $\xi_{i,n+1} : \alpha_{i,n+1} (V^{\otimes n}\otimes W) \to V^{\otimes i-1} \otimes W \otimes V^{\otimes n-i+1}$ for all $1 \le i \le n+1$, there is an $A_{n+1}$-action on $\bigoplus_{i=1}^{n+1} V^{\otimes i-1} \otimes W \otimes V^{\otimes n-i+1}$ defined by $a \mapsto \xi_{\underline{a}(i)} a \xi_i^{-1}$, such that the above is an isomorphism of $A_{n+1}$-representations.
\end{prop}

\begin{proof}
The isomorphism of $A_{n+1}$-representations follows immediately from the definition of the action of $A_{n+1}$ on $\bigoplus_{i=1}^{n+1} V^{\otimes i-1} \otimes W \otimes V^{\otimes n-i+1}$.
\end{proof}

By convention, it is always assumed that $\xi_{n+1,n+1}$ is the identity map. The second index of the map $\xi_{i,n+1}$ again denotes the total degree of the domain and is omitted from the notation if there is no ambiguity. Observe that on the right hand side, we can apply braids that ``move'' the factor $W$, an operation that is forbidden in the $B_n$-representation on $V^{\otimes n} \otimes W$. This allows for a more intuitive framework to study the action of $A_{n+1}$ on the induced representation, analogous to its action in the monoidal braid representation on $V^{\otimes n+1}$.

As in the quantum shuffle algebra, we will also denote elements of $V^{\otimes i-1} \otimes W \otimes V^{\otimes n-i+1}$ by the bar complex notation, i.e. $[v_1|...|v_{i-1}| w| v_{i+1}| ...| v_{n+1}]$. In a specific case of composition of braid actions, if we apply the braid element $\sigma_n$ to an element in $\alpha_n(V^{\otimes n}\otimes W)$, we obtain
\[\sigma_n (\alpha_n [v_1|...|v_n|w]) = (\sigma_n^2) [v_1|...|v_n|w] = (\mathrm{id}^{\otimes n-1} \otimes \tau) [v_1|...|v_n|w].\]
Hence on the vector space $V^{\otimes n} \otimes W$, $\sigma_n$ acts as the ``left square root'' of the braid action of $\tau_n$. Formally, we may define maps $s_1 : V^{\otimes n} \otimes W \to V^{\otimes n-1} \otimes W\otimes V$ by
\[ [v_1|...|v_n|w] \mapsto \xi_n(\sigma_n[v_1|...|v_n|w]) \]
and $s_2 : V^{\otimes n-1} \otimes W\otimes V \to V^{\otimes n} \otimes W$ by
\[ [v_1|...|v_{n-1}|w|v_{n+1}] \mapsto \sigma_n (\xi_n^{-1}[v_1|...|v_{n-1}|w|v_{n+1}]) \]
that can be treated as ``left square roots'' of the braiding map $\mathrm{id}^{\otimes n-1} \otimes \tau$ in the $A_{n+1}$-representation on $\mathrm{Ind}^{A_{n+1}}_{B_n} (V^{\otimes n}\otimes W)$.

Observe that in general, the $A_{n+1}$-action on $\bigoplus_{i=1}^{n+1} V^{\otimes i-1} \otimes W \otimes V^{\otimes n-i+1}$, while slightly more intuitive than the induced representation, does not have an explicit formula that can be used in computations. This issue is more tractable when the representation on $\bigoplus_{i=1}^{n+1} V^{\otimes i-1} \otimes W \otimes V^{\otimes n-i+1}$ behaves analogously to that on $V^{\otimes n+1}$, in the sense that the $A_{q+1}$-action on a tensor subfactor $V^{\otimes i-1} \otimes W \otimes V^{q-i+1}$ of a summand $V^{\otimes i'-1} \otimes W \otimes V^{n-i'+1}$ agrees with the $A_{n+1}$-action on the entire summand for all $1 \le q \le n$ and $1 \le i \le q+1$. That is, we desire the following property: for any $a \in A_{q+1}$, the diagram

\begin{equation}\label{diag:separable}
\begin{tikzcd}
V^{\otimes p} \otimes (V^{\otimes i-1} \otimes W \otimes V^{q-i+1}) \otimes V^{\otimes n-p-q} \arrow[swap]{d}{\mathrm{id}^{\otimes p} \otimes \xi^{-1}_{i,q+1} \otimes \mathrm{id}^{\otimes n-p-q}} \arrow{r}{\xi^{-1}_{p+i,n+1}} &[0.7cm] \alpha_{p+i,n+1}(V^{\otimes n} \otimes W) \arrow{ddd}{a'} \\%
V^{\otimes p} \otimes \alpha_{i,q+1} (V^{\otimes q} \otimes W) \otimes V^{\otimes n-p-q} \arrow[swap]{d}{\mathrm{id}^{\otimes p} \otimes a \otimes \mathrm{id}^{\otimes n-p-q}} \\
V^{\otimes p} \otimes \alpha_{\underline{a}(i),q+1} (V^{\otimes q} \otimes W) \otimes V^{\otimes n-p-q} \arrow[swap]{d}{\mathrm{id}^{\otimes p} \otimes \xi_{\underline{a}(i),q+1} \otimes \mathrm{id}^{\otimes n-p-q}}\\
V^{\otimes p} \otimes (V^{\otimes \underline{a}(i)-1} \otimes W \otimes V^{q-\underline{a}(i)+1}) \otimes V^{\otimes n-p-q} & \arrow[swap]{l}{\xi_{\underline{a'}(p+i),n+1}} \alpha_{\underline{a'}(p+i),n+1}(V^{\otimes n} \otimes W)
\end{tikzcd}
\end{equation}
commutes, where the braid $a'$ is the natural inclusion of $a$ into the copy $A_{q+1} \le A_{n+1}$ consisting of braids that are only nontrivial on the $q+1$ strands starting with the $p+1^{\mathrm{st}}$. The following proposition gives criteria to detect this property.

\begin{prop}\label{prop:comm_diag_cond}
Let $\varphi_{i,n} : V^{\otimes n-1} \otimes W \to V^{\otimes i-1} \otimes W \otimes V^{\otimes n-i}$ be defined by $\varphi_{i,n}:= \xi_{i,n}\alpha_{i,n}$ for all $n \ge 1$ and $1 \le i \le n$, and in particular denote $\varphi := \varphi_{1,2} = \xi_{1,2}\alpha_{1,2}$. Then Diagram~\ref{diag:separable} always commutes if and only if 
\[\text{(1) } \varphi_{i,n} = (\mathrm{id}^{\otimes i-1} \otimes \varphi \otimes \mathrm{id}^{\otimes n-i-1}) \circ (\mathrm{id}^{\otimes i} \otimes \varphi \otimes \mathrm{id}^{\otimes n-i-2}) \circ \cdots \circ (\mathrm{id}^{\otimes n-2} \otimes \varphi)\]
and the following identities hold:
\begin{enumerate}\addtocounter{enumi}{1}
    \item $(\mathrm{id} \otimes \sigma) \circ (\varphi \otimes \mathrm{id}) \circ (\mathrm{id} \otimes \varphi) = (\varphi \otimes \mathrm{id}) \circ (\mathrm{id} \otimes \varphi) \circ (\sigma \otimes \mathrm{id})$;
    \item $(\tau \otimes \mathrm{id}) \circ (\mathrm{id} \otimes \varphi) = (\mathrm{id} \otimes \varphi) \circ (\sigma \otimes \mathrm{id}) \circ (\mathrm{id} \otimes \tau) \circ (\sigma^{-1} \otimes \mathrm{id})$.
\end{enumerate}
\end{prop}

\begin{proof}
Since every braid action is decomposable into those of the generators, it suffices to study the commutativity of Diagram~\ref{diag:separable} for all braid generators. By a straightforward yet arduous reduction, we can show that Diagram~\ref{diag:separable} commutes for all braid generators if and only if the following identities hold:
\begin{enumerate}[label=(\alph*)]
    \item $(\mathrm{id}^{\otimes p} \otimes \varphi_{i,q} \otimes \mathrm{id}^{\otimes n-p-q}) \varphi_{p+q,n} = \varphi_{p+i,n}$;
    
    \vspace{.05in}
    \item $\sigma_m \varphi_{i,n} = \begin{cases}
    \varphi_{i,n} \sigma_m \qquad \text{if } m \le i-2\\
    \varphi_{i,n} \sigma_{m-1} \hfill \text{if } m \ge i+1
    \end{cases}$
    for $\sigma_m = \mathrm{id}^{\otimes m-1} \otimes \sigma \otimes \mathrm{id}^{\otimes n-m-1}$;
    
    \item $\tau_{i-1} \varphi_{i,n} = \varphi_{i,n} (\sigma_{i-1} ... \sigma_{n-2} \tau_{n-1} \sigma_{n-2}^{-1} ... \sigma_{i-1}^{-1})$, for $\tau_{j} = \mathrm{id}^{\otimes j-1} \otimes \tau \otimes \mathrm{id}^{\otimes n-j-1}$.
\end{enumerate}
We will show that these conditions are equivalent to equations (1-3).

First, we will prove that formula (1) is equivalent to condition (a). It is easy to see that the formula (1) for $\varphi_{i,n}$ satisfies condition (a). Conversely, given condition (a) and a choice of $\varphi = \varphi_{1,2}$, we will prove formula (1) by induction on $n$. The base case $n = 3$ can be checked directly:
\[\varphi_{2,3} = (\mathrm{id} \otimes \varphi_{1,2})\varphi_{3,3} = \mathrm{id} \otimes \varphi\]
and
\[\varphi_{1,3} = (\varphi_{1,2} \otimes \mathrm{id})\varphi_{2,3} = (\varphi \otimes \mathrm{id}) \circ (\mathrm{id} \otimes \varphi).\]
Suppose the formula holds for $n$. Consider $\varphi_{i,n+1}$. For all $1 \le i \le n$, we have
\[ \begin{array}{l}
    \varphi_{i+1,n+1} = (\mathrm{id} \otimes \varphi_{i,n}) \varphi_{n+1,n+1} \\[5pt]
     = \mathrm{id} \otimes \left[ (\mathrm{id}^{\otimes i-1} \otimes \varphi \otimes \mathrm{id}^{\otimes n-i-1}) \circ (\mathrm{id}^{\otimes i} \otimes \varphi \otimes \mathrm{id}^{\otimes n-i-2}) \circ \cdots \circ (\mathrm{id}^{\otimes n-2} \otimes \varphi) \right]\\[5pt]
     = (\mathrm{id}^{\otimes i} \otimes \varphi \otimes \mathrm{id}^{\otimes n-i-1}) \circ (\mathrm{id}^{\otimes i+1} \otimes \varphi \otimes \mathrm{id}^{\otimes n-i-2}) \circ \cdots \circ (\mathrm{id}^{\otimes n-1} \otimes \varphi).
\end{array} \]
Finally, apply the above formula for $\varphi_{2,n+1}$ to $\varphi_{1,n+1} = (\varphi_{1,2} \otimes \mathrm{id}^{\otimes n-1})\varphi_{2,n+1}$ to complete the remaining case. With this identification, observe that condition (b) gives $\sigma_2 \varphi_{1,3} = \varphi_{1,3}\sigma_1$, which is identity (2), and condition (c) gives $\tau_1 \varphi_{2,3} = \varphi_{2,3}(\sigma_1 \tau_2 \sigma_1^{-1})$, which is identity (3). So the forward direction of the statement holds.

Conversely, assume formulae (1-3). Let $\varphi_i := \mathrm{id}^{\otimes i-1} \otimes \varphi \otimes \mathrm{id}^{\otimes n-i-1}$, then we may rewrite these formulae as
\begin{enumerate}
    \item $\varphi_{i,n} = \varphi_i \varphi_{i+1} ... \varphi_{n-1}$;
    \item $\sigma_{i+1}\varphi_i \varphi_{i+1} = \varphi_i \varphi_{i+1} \sigma_i$;
    \item $\tau_i \varphi_{i+1} = \varphi_{i+1} \sigma_i \tau_{i+1} \sigma_i^{-1}$.
\end{enumerate}
In addition to these identities, note that $\sigma_m \varphi_i = \varphi_i \sigma_m$ whenever $|m-i| \ge 2$. For condition (b), observe that if $m \le i-2$, $\sigma_m \varphi_{i,n} = \sigma_m \varphi_i ... \varphi_{n-1} = \varphi_i ... \varphi_{n-1} \sigma_m = \varphi_{i,n} \sigma_m$, since $\sigma_m$ commutes with each $\varphi_j$. If $m \ge i+1$, by applying (1) and (2), we have $\sigma_m \varphi_{i,n} = \varphi_i ... \varphi_{m-2} \sigma_m \varphi_{m-1} \varphi_m ... \varphi_{n-1} = \varphi_i ... \varphi_{m-1} \varphi_m \sigma_{m-1} \varphi_{m+1} ... \varphi_{n-1} = \varphi_i ... \varphi_{n-1} \sigma_{m-1}= \varphi_{i,n} \sigma_{m-1}$, so condition (b) is satisfied. We will prove condition (c) by backward induction on $i$ for a fixed $n$. The base case $i=n-1$ is precisely formula (3). Suppose that condition (c) holds for $i$. We then have
\[\begin{array}{r l}
\tau_{i-2}\varphi_{i-1,n} & = \tau_{i-2} \varphi_{i-1} \varphi_{i,n} = \varphi_{i-1} \sigma_{i-2} \tau_{i-1} \sigma_{i-2}^{-1} \varphi_{i,n} = \varphi_{i-1} \sigma_{i-2} \tau_{i-1} \varphi_{i,n} \sigma_{i-2}^{-1} \\[5pt]
& = \varphi_{i-1} \sigma_{i-2} \varphi_{i,n} (\sigma_{i-1} ... \sigma_{n-2} \tau_{n-1} \sigma_{n-2}^{-1} ... \sigma_{i-1}^{-1}) \sigma_{i-2}^{-1}\\[5pt]
& = \varphi_{i-1} \varphi_{i,n} (\sigma_{i-2} \sigma_{i-1} ... \sigma_{n-2} \tau_{i-1} \sigma_{n-2}^{-1} ... \sigma_{i-1}^{-1} \sigma_{i-2}^{-1})\\[5pt]
& = \varphi_{i-1,n} (\sigma_{i-2} ... \sigma_{n-2} \tau_{i-1} \sigma_{n-2}^{-1} ... \sigma_{i-2}^{-1}). 
\end{array}\]
This concludes our induction.
\end{proof}

It follows that our desired property for the action of $A_{n+1}$ on $\bigoplus_{i=1}^{n+1} V^{\otimes i-1} \otimes W \otimes V^{\otimes n-i+1}$ can be detected by the existence of an isomorphism $\varphi: V \otimes W \to W \otimes V$ satisfying identities (2) and (3) in the proposition above. In principle, the choices of maps $\varphi$ and $\xi_{1,2}: \alpha_{1,2}(V\otimes W) \to W \otimes V$ are equivalent via the relation $\varphi = \xi_{1,2}\alpha_{1,2}$; in practice however, given an explicit left-braided vector space, it is often more convenient to construct a map $\varphi$, due to the fact that $\alpha_{1,2}(V \otimes W)$ is abstract. Observe that these identities directly involve the braiding maps $\sigma$ and $\tau$ of the left-braided vector space $(V,W)$, while $\varphi : V \otimes W \to W \otimes V$ plays the role of the ``left square root'' of $\tau$. These criteria therefore are strictly internal to the structure of left-braided vector spaces, as encapsulated in the following definition:

\begin{defn}\label{defn:sep_lbvs}
    A left-braided vector space $(V,W,\sigma,\tau)$ is \textit{separable} if there exists an isomorphism $\varphi : V \otimes W \to W \otimes V$ (called the \textit{separated braiding}) that satisfies the following braid equations on $V^{\otimes 2} \otimes W$:
    \begin{enumerate}
    \item $(\mathrm{id} \otimes \sigma) \circ (\varphi \otimes \mathrm{id}) \circ (\mathrm{id} \otimes \varphi) = (\varphi \otimes \mathrm{id}) \circ (\mathrm{id} \otimes \varphi) \circ (\sigma \otimes \mathrm{id})$;
    \item $(\tau \otimes \mathrm{id}) \circ (\mathrm{id} \otimes \varphi) = (\mathrm{id} \otimes \varphi) \circ (\sigma \otimes \mathrm{id}) \circ (\mathrm{id} \otimes \tau) \circ (\sigma^{-1} \otimes \mathrm{id})$.
    \end{enumerate}
    A separable left-braided vector space $(V,W,\sigma,\tau)$ with the choice of separated braiding $\varphi$ is denoted by $(V,W,\sigma,\tau,\varphi)$.
\end{defn}

\begin{exmp}
When $V=W=k$, the left-braided vector space in Example~\ref{exmp:V=W=k} is separable with the separated braiding $\varphi$ given simply by permutation of tensor factors.

Generally, this choice of $\varphi$ does not satisfy condition (2). One such example is $(V,V,\sigma,\sigma^2)$ for a given braided vector space $(V,\sigma)$ (see Example~\ref{exmp:V=bvs}). In this case, however, the left-braided vector space is separable with the obvious choice of separated braiding $\varphi := \sigma$.
\end{exmp}

Separability of a left-braided vector space $(V,W)$ is integrally connected to the existence of a braid structure on the direct sum $V \oplus W$.

\begin{prop}\label{prop:sep_lbvs_iff_V+W=bvs}
Let $V$ and $W$ be finite dimensional $k$-vector spaces, and let $X = V \oplus W$. Suppose there is an automorphism $\sigma_X$ of $(V \oplus W)^{\otimes 2} \cong V^{\otimes 2} \oplus (V \otimes W) \oplus (W \otimes V) \oplus W^{\otimes 2}$ defined summand-wise by isomorphisms $\sigma_V: V^{\otimes 2} \to V^{\otimes 2}$, $\varphi: V \otimes W \to W \otimes V$, $\psi: W \otimes V \to V \otimes W$, and $\sigma_W: W^{\otimes 2} \to W^{\otimes 2}$.
\begin{enumerate}
    \item If $X$ is a braided vector space, then $(V,W,\sigma_V,\tau,\varphi)$ is a separable left-braided vector space where $\tau = \psi\varphi$;
    
    \item A weak version of the converse holds: if the assumption on $\sigma_X$ is relaxed by setting $\sigma_W = 0$ (in particular, $\sigma_X$ is no longer an isomorphism) and $(V,W,\sigma_V,\tau,\varphi)$ is a separable left-braided vector space, then $X$ is a lax braided vector space, i.e. the map $\sigma_X$ satisfies the braid equation on $X^{\otimes 3}$:
    \[(\sigma_X \otimes \mathrm{id}) \circ (\mathrm{id} \otimes \sigma_X) \circ (\sigma_X \otimes \mathrm{id}) = (\mathrm{id} \otimes \sigma_X) \circ (\sigma_X \otimes \mathrm{id}) \circ (\mathrm{id} \otimes \sigma_X).\]
\end{enumerate}
\end{prop}

\begin{proof}
This proof rests on the following key observation: $X$ is a (lax) braided vector space if and only if the braid equation holds on each of the eight summands of $(V \oplus W)^{\otimes 3}$, e.g.,
\begin{enumerate}[label=(\alph*)]
    \item $(\sigma_V \otimes \mathrm{id}) \circ (\mathrm{id} \otimes \sigma_V) \circ (\sigma_V \otimes \mathrm{id}) = (\mathrm{id} \otimes \sigma_V) \circ (\sigma_V \otimes \mathrm{id}) \circ (\mathrm{id} \otimes \sigma_V)$;
    \item $(\varphi \otimes \mathrm{id}) \circ (\mathrm{id} \otimes \varphi) \circ (\sigma_V \otimes \mathrm{id}) = (\mathrm{id} \otimes \sigma_V) \circ (\varphi \otimes \mathrm{id}) \circ (\mathrm{id} \otimes \varphi)$;
    \item $(\psi \otimes \mathrm{id}) \circ (\mathrm{id} \otimes \sigma_V) \circ (\varphi \otimes \mathrm{id}) = (\mathrm{id} \otimes \varphi) \circ (\sigma_V \otimes \mathrm{id}) \circ (\mathrm{id} \otimes \psi)$;
    \item $(\sigma_V \otimes \mathrm{id}) \circ (\mathrm{id} \otimes \psi) \circ (\psi \otimes \mathrm{id}) = (\mathrm{id} \otimes \psi) \circ (\psi \otimes \mathrm{id}) \circ (\mathrm{id} \otimes \sigma_V)$;
\end{enumerate}
and four others involving $\sigma_W$. When $\sigma_W = 0$, the latter four equations are automatically satisfied, so it suffices to show that conditions (a-d) are the necessary and sufficient conditions for the braid structure and separability of $(V,W,\sigma_V,\tau,\varphi)$. The arguments for both directions are very similar; here we will only show the proof of part (1) of the statement. For the rest of this proof, we will use the notation $\sigma_1 := \sigma_V \otimes \mathrm{id}$ and $\sigma_2 := \mathrm{id} \otimes \sigma_V$, and analogous notations for $\tau$, $\varphi$, and $\psi$.

Assume conditions (a-d). First we show that $(V,W,\sigma_V,\tau)$ is a left-braided vector space where $\tau = \psi\varphi$. Condition (a) implies that $(V,\sigma_V)$ is a braided vector space. For the additional braid equation, we have
\[ \begin{array}{r l}
     \tau_2 \sigma_1 \tau_2 \sigma_1 & = \psi_2 (\varphi_2 \sigma_1 \psi_2) \varphi_2 \sigma_1 = (\psi_2 \psi_1 \sigma_2)(\varphi_1 \varphi_2 \sigma_1) \\
     & = \sigma_1 \psi_2 (\psi_1 \sigma_2 \varphi_1) \varphi_2 = \sigma_1 (\psi_2 \varphi_2) \sigma_1 (\psi_2 \varphi_2) = \sigma_1 \tau_2 \sigma_1 \tau_2.
\end{array} \]
So indeed $(V,W,\sigma_V,\tau)$ forms a left-braided vector space. Condition (b) is the same as (1) in Definition~\ref{defn:sep_lbvs}, while for condition (2) we have:
\[ \tau_1 \varphi_2 = \psi_1 (\varphi_1 \varphi_2 \sigma_1) \sigma_1^{-1} = (\psi_1 \sigma_2 \varphi_1) \varphi_2 \sigma_1^{-1} = \varphi_2 \sigma_1 \psi_2 \varphi_2 \sigma_1^{-1} = \varphi_2 \sigma_1 \tau_2 \sigma_1^{-1}.\]
Therefore $(V,W,\sigma_V,\tau)$ is separable with the separated braiding $\varphi$.
\end{proof}

Roughly speaking, Proposition~\ref{prop:sep_lbvs_iff_V+W=bvs} states that separability of a left-braided vector space $(V,W)$ is equivalent to half of the data of a braid structure on $V \oplus W$. This fact again exhibits an asymmetry in the structure of left-braided vector spaces, which stems from the nature of the Artin groups of type \textit{B}. Finally, when $(V,W)$ is a separable left-braided vector space, Propositions~\ref{prop:comm_diag_cond} and ~\ref{prop:sep_lbvs_iff_V+W=bvs} together imply that there is a simple description of the action of $A_{n+1}$ in the braid representation on $\bigoplus_{i=1}^{n+1} V^{\otimes i-1} \otimes W \otimes V^{\otimes n-i+1}$ defined in Proposition~\ref{prop:ind_rep_lbvs}:

\begin{cor}\label{cor:ind_rep_sep_lbvs}
Let $(V,W,\sigma,\tau,\varphi)$ be a separable left-braided vector space. Then there is an isomorphism of $A_{n+1}$-representations
\[ \mathrm{Ind}^{A_{n+1}}_{B_n} (V^{\otimes n}\otimes W) \cong \displaystyle \bigoplus_{i=1}^{n+1} V^{\otimes i-1} \otimes W \otimes V^{\otimes n-i+1}\]
where the action of $A_{n+1}$ on $\bigoplus_{i=1}^{n+1} V^{\otimes i-1} \otimes W \otimes V^{\otimes n-i+1}$ is defined by
\[ \sigma_m[v_1|...|w_i|...|v_{n+1}] = \begin{cases}
(\mathrm{id}^{\otimes m-1} \otimes \sigma \otimes \mathrm{id}^{\otimes n-m})[v_1|...|w_i|...|v_{n+1}] \hspace{.3in} \hfill m \ne i-1, i \\
(\mathrm{id}^{\otimes i-2} \otimes \varphi \otimes \mathrm{id}^{\otimes n-i+1})[v_1|...|w_i|...|v_{n+1}] \hfill m = i-1 \\
(\mathrm{id}^{\otimes i-1} \otimes \tau\varphi^{-1} \otimes \mathrm{id}^{\otimes n-i})[v_1|...|w_i|...|v_{n+1}] \hfill m = i.
\end{cases}
\]
\end{cor}


\section{Homology of type-\textit{B} Artin groups}\label{sec:main}

Let $(V,W)$ be a separable left-braided vector space. In this section, we will prove the main result of this paper, identifying
\[H_* (B_n; V^{\otimes n} \otimes W) \cong \mathrm{Ext}^{n-*, n+1}_{\mathfrak{A}^e} (\mathfrak{M},k)\]
where $\mathfrak{A}$ is a quantum shuffle algebra and $\mathfrak{M}$ is an $\mathfrak{A}$-bimodule that will be defined in Section~\ref{ssec:alg_req}. We will first recall a homological algebra statement from \cite{etw17} about the homology of braid groups with twisted coefficients, then develop suitable algebraic structures to apply this result in the proof of our main theorem.


\subsection{Fox-Neuwirth cellular stratification of Conf$_n (\mathbb{C})$}\label{ssec:fnf}

We recall the Fox-Neuwirth/Fuks stratification of Conf$_n(\mathbb{C})$ by Euclidean spaces, which provides a CW-complex structure for the 1-point compactification of Conf$_n (\mathbb{C})$. This construction was first demonstrated in \cite{fn62,fuk70} and further studied in \cite{vas92,gs12,etw17}; the treatment detailed here is taken primarily from \cite{etw17}.

A \textit{composition} $\lambda$ of $n$ is an ordered partition $\lambda = (\lambda_1, ..., \lambda_k)$ of $n$ where $\sum \lambda_i = n$. The number of parts $k$ is called the \textit{length} of $\lambda$, denoted by $l(\lambda)$. Recall that the $n^\mathrm{th}$ symmetric product Sym$_n(\mathbb{R})$ has a stratification given by these partitions Sym$_n (\mathbb{R}) = \coprod_{\lambda \vdash n} \text{Sym}_\lambda (\mathbb{R})$. Elements of Sym$_\lambda (\mathbb{R})$ are unordered subsets of $l(\lambda)$ distinct points $x_1, ..., x_{l(\lambda)}$, where the multiplicity of $x_i$ is $\lambda_i$. We further assume that $x_1 < ... < x_{l(\lambda)}$ where the ordering is that of the real line. Define a map $\pi : \text{Conf}_n(\mathbb{C}) \to \text{Sym}_n(\mathbb{R})$ by projecting each coordinate onto the real line, i.e. $\pi (z_1, ..., z_n) = (\text{Re}(z_1), ..., \text{Re}(z_n))$. The preimage of Sym$_\lambda(\mathbb{R})$ under $\pi$, Conf$_\lambda (\mathbb{C})$ is homeomorphic to $\text{Sym}_\lambda (\mathbb{R}) \times \prod^{l(\lambda)}_{i=1} \text{Conf}_{\lambda_i} (\mathbb{R})$, where each configuration factor Conf$_{\lambda_i} (\mathbb{R})$ keeps track of the imaginary parts of points with the same real part (see Figure~\ref{fig:fnf}). Since $\mathrm{Sym}_k(\mathbb{R}) \cong \mathrm{Conf}_k(\mathbb{R}) \cong \mathbb{R}^k$, that implies Conf$_\lambda(\mathbb{C}) \cong \mathbb{R}^{n+l(\lambda)}$.

\begin{figure}[t]
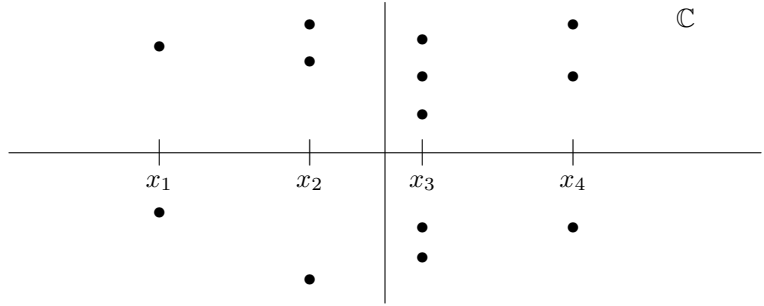

\begin{centering}
\[ \vcenter{	
	\xy
		(-30,14)*{\bullet}; (-30,-8)*{\bullet}; (-30,-4)*{x_1}; (-30,0)*{|};
		(-10,12)*{\bullet}; (-10,-17)*{\bullet}; (-10,17)*{\bullet}; (-10,-4)*{x_2}; (-10,0)*+{|};
		(5,10)*{\bullet}; (5,-10)*{\bullet}; (5,15)*{\bullet}; (5,-14)*{\bullet}; (5,5)*{\bullet}; (5,-4)*{x_3}; (5,0)*+{|};
		(25,10)*{\bullet}; (25,-10)*{\bullet}; (25,17)*{\bullet}; (25,-4)*{x_4}; (25,0)*+{|};
		{\ar@{-} (-50,0)*{}; (50,0)*{}};
		{\ar@{-} (0,-20)*{}; (0,20)*{}};
		(40,18)*{ \mathbb{C}};
	\endxy
	}
\]
	\caption{A configuration in Conf$_{(2,3,5,3)} (\mathbb{C}) \subset$ Conf$_{13} (\mathbb{C})$. The configuration is mapped by $\pi$ to $(x_1, x_1, x_2, x_2, x_2, x_3, x_3, x_3, x_3, x_3, x_4, x_4, x_4) \in \mathrm{Sym}_{(2,3,5,3)} (\mathbb{R}) \subset \mathrm{Sym}_{13} (\mathbb{R})$. Geometrically, the points in the configuration are arranged into columns based on their (ordered) real coordinates, while the configuration factors $\mathrm{Conf}_{\lambda_i}(\mathbb{R})$ record the imaginary parts of the points in respective columns.}
	\label{fig:fnf}
\end{centering}
\end{figure}

The collection of the spaces Conf$_\lambda(\mathbb{C})$ forms a cellular decomposition for the 1-point compactification of Conf$_n(\mathbb{C})$. Each configuration space Conf$_\lambda (\mathbb{C})$ gives a cell of dimension $n+l(\lambda)$. Loosely speaking, a cell can be constructed by first picking a composition $\lambda$ of $n$, placing the points in the configuration into columns according to the partition, and finally letting the columns to move horizontally (the symmetric part of the product) and the points on each column to move vertically without colliding (the configuration part of the product). For a configuration in such a cell, we may define a \textit{total order} of points in that configuration by labelling the lowest point on the leftmost column with $1$, increasing the indices as we move up, and continuing the process for all subsequent columns on the right. The index of a point in this total order is called the \textit{overall position} of that point in the configuration.

The boundary of a cell is obtained in two ways. The first type of boundary occurs by moving a point in a configuration to approach either the point at infinity or another point on the same vertical line, i.e. with the same real part. In this case, the boundary is the point at infinity, since the number of points in the configuration decreases to $n-1$ and hence the configuration is no longer an element in Conf$_n (\mathbb{C})$. The second type of boundary occurs by horizontally joining two neighboring vertical columns of the configuration without colliding the points. The boundary cell of Conf$_\lambda (\mathbb{C})$ obtained by joining the $i^\mathrm{th}$ and $i+1^\mathrm{st}$ columns has the form Conf$_\rho(\mathbb{C})$, where $\rho^i = (\lambda_1, ..., \lambda_{i-1}, \lambda_i + \lambda_{i+1}, ..., \lambda_k)$ is the \textit{coarsening} of $\lambda$ obtained by summing $\lambda_i$ and $\lambda_{i+1}$ $(1 \le i < l(\lambda)).$

\begin{prop}[\cite{fn62,fuk70}]
The space $\mathrm{Conf}_n (\mathbb{C}) \cup \{ \infty \}$ has a CW-complex decomposition where the positive dimension cells are given by $\mathrm{Conf}_\lambda(\mathbb{C})$ (of dimension $n+l(\lambda)$) with indices $\lambda$ coming from compositions of $n$. The boundary of $\mathrm{Conf}_\lambda (\mathbb{C})$ is the union of $\mathrm{Conf}_\rho(\mathbb{C})$ where $\lambda$ is a refinement of $\rho$. \end{prop}

Fox-Neuwirth and Fuks provided an explicit cellular chain complex for Conf$_n(\mathbb{C}) \cup \{ \infty \}$ based on this decomposition:

\begin{defn}[The Fox-Neuwirth/Fuks complex]
For integers $i$ and $j$, let $c_{i,j} = \sum_\gamma (-1)^{|\gamma|}$ be the sum of the signs of all $(i,j)$-shuffles $\gamma : \{1,..., i\} \sqcup \{1,...,j\} \to \{1,..., i+j\}$. Let $C(n)_*$ denote the chain complex which in degree $q$ is generated over $\mathbb{Z}$ by the set of ordered partitions $\lambda = (\lambda_1, ..., \lambda_{q-n})$ of $n$ with $q-n$ parts. The differential $d: C(n)_q \to C(n)_{q-1}$ is given by the formula
\[
d(\lambda_1, ..., \lambda_{q-n}) = \sum^{q-n-1}_{i=1} (-1)^{i-1} c_{\lambda_i, \lambda_{i+1}} (\lambda_1, ..., \lambda_{i-1}, \lambda_i + \lambda_{i+1},..., \lambda_{q-n}).
\]
\end{defn}

The constant $c_{\lambda_i, \lambda_{i+1}}$ results from the formation of the boundary cells of Conf$_\lambda (\mathbb{C})$ by shuffling the points in the $i^\mathrm{th}$ and $i+1^\mathrm{st}$ columns into a single vertical line. The signs in the formula arise from the induced orientations on the boundary strata, as detailed in \cite{gs12}. There is a simple formula to compute this constant using the quantum binomial coefficient (see, e.g., Prop. 1.7.1 in \cite{sta12}):
\[c_{p,q} = \displaystyle {p + q \choose q}_{-1}. \]
The complex $C(n)_*$ is isomorphic to the cellular chain complex of Conf$_n (\mathbb{C}) \cup \{ \infty \}$, relative to the point at infinity. For more about this complex, consult \cite{fuk70,vai78,vas92}.

Let $T$ be a representation of $A_n$, and $\mathcal{T}$ the associated local system over $\mathrm{Conf}_n(\mathbb{C})$. There is an explicit formula for the differential of the chain complex $C(n)_* \otimes T$ that incorporates the braid action on $T$ given by:
\[ d[(\lambda_1, ..., \lambda_{q-n}) \otimes t] = \sum^{q-n-1}_{i=1} (-1)^{i-1} \Big[ (\lambda_1, ..., \lambda_{i-1}, \lambda_i + \lambda_{i+1},..., \lambda_{q-n}) \otimes \sum_{\gamma} (-1)^{|\gamma|} \tilde{\gamma}(t) \Big] \]
where $\gamma$ is drawn from the $(\lambda_i, \lambda_{i+1})$-shuffles, and $\tilde{\gamma}$ is its lift (as described in Section~\ref{ssec:bvs_qsa}) to the copy $A_{\lambda_i + \lambda_{i+1}} \le A_n$ consisting of braids that are only nontrivial on the $\lambda_i + \lambda_{i+1}$ strands starting with the $\lambda_1 + ... + \lambda_{i-1} +1^\mathrm{st}$ \cite{etw17}.

\begin{thm}[\cite{etw17}]\label{thm:etw_fnf}
There is an isomorphism
\[H_*(\mathrm{Conf}_n(\mathbb{C}) \cup \{ \infty \},\{ \infty \}; \mathcal{T}) \cong H_*(C(n)_* \otimes T).\]
\end{thm}

Let $\mathcal{T}^*$ denote the dual local system of $\mathcal{T}$, which is associated with the dual representation $T^*$ of $T$. By applying the universal coefficient theorem and Poincar\'e duality to the dual over $k$ of the left hand side with coefficients in $\mathcal{T}^*$, we have
\[\begin{split}
    H_* (\text{Conf}_{n}(\mathbb{C}) \cup \{ \infty \}, \{\infty\} ; \mathcal{T}^*)^* & \cong H^* (\text{Conf}_{n}(\mathbb{C}) \cup \{ \infty \}, \{\infty\} ; \mathcal{T})\\
    & \cong H^*_c (\text{Conf}_{n}(\mathbb{C});\mathcal{T})\\
    & \cong H_{2n-*} (\text{Conf}_{n}(\mathbb{C});\mathcal{T}).
\end{split}\]
Since $\pi_1 (\mathrm{Conf}_n(\mathbb{C})) = A_n$, we may rewrite the isomorphism in the previous theorem as
\[H_*(A_n; T) \cong H_{2n-*}(C(n)_* \otimes T^*)^*. \]

Recall that given any $B_n$-representation, the induced representation gives a representation for the parent group $A_{n+1}$. Combining this fact with the previous result, we obtain the following corollary:

\begin{cor}\label{cor:etw_rewrite}
For any $B_n$-representation $L$, there is an isomorphism
\[ H_*(B_n; L) \cong H_{2n+2-*} \Big( C(n+1)_* \otimes \mathrm{Ind}^{A_{n+1}}_{B_n} (L^*) \Big)^*. \]
\end{cor}

\begin{proof}
By Shapiro's Lemma (see, e.g., Lemma 6.3.2 of \cite{wei94}), the homology of $B_n$ with local coefficients in $L$ can be identified with the homology of the parent group $A_{n+1}$ with coefficients in the induced representation, i.e.
\[ H_*(B_n; L) \cong H_* \Big( A_{n+1}; \mathrm{Ind}^{A_{n+1}}_{B_n} (L) \Big). \]
Applying the rewritten form of Theorem~\ref{thm:etw_fnf} to the right side completes the proof of the corollary.
\end{proof}

This homological algebra result provides a convenient machinery to compute the homology of $B_n$ (or similarly any other subgroup of the braid groups) with twisted coefficients, as long as we understand the action of $A_{n+1}$ on the induced representation. However, it does not offer a geometrically intuitive explanation for how the cellular homology on the right side is related to the homology of the Artin groups of type \textit{B}. We will resolve this issue in Section~\ref{sec:fnf}, by explicitly constructing a cellular chain complex that computes the homology of $B_n$ with coefficients in $L$ and relating it to the complex $C(n+1)_* \otimes \mathrm{Ind}^{A_{n+1}}_{B_n} (L^*)$.


\subsection{Homological algebra prerequisites}\label{ssec:alg_req}

The purpose of this subsection is to develop homological algebra objects necessary to express the homology of the complex $C(n+1)_* \otimes \mathrm{Ind}^{A_{n+1}}_{B_n} (L^*)$ elaborated in the previous section when $L = V^{\otimes n}\otimes W$, the $B_n$-representation derived from a separable left-braided vector space $(V,W,\sigma,\tau,\varphi)$.

\begin{defn}
Given an associative $k$-algebra $A$ and an $A$-bimodule $M$, the \textit{chain complex} $F_*(M,A)$ is defined at degree $q \ge 1$ by
\[ F_q(M,A) = \displaystyle \bigoplus_{i=1}^{q} A^{\otimes i-1}\otimes M\otimes A^{\otimes q-i}\]
with face maps for $1\le m \le q-1$:
\[d_m (a_1 \otimes ... \otimes a_{i-1} \otimes \mu_i \otimes a_{i+1} \otimes ... \otimes a_q) = \begin{cases}
a_1 \otimes ... \otimes a_m a_{m+1} \otimes ... \otimes a_q \hspace{.2in} m \ne i-1, i\\
a_1 \otimes ... \otimes a_{i-1} \mu_i \otimes ... \otimes a_q \hfill m = i-1\\
a_1 \otimes ... \otimes \mu_i a_{i+1} \otimes ... \otimes a_q \hfill m = i\\
\end{cases}\]
and differential $d = \sum_{m=1}^{q-1} (-1)^{m-1} d_m$.
\end{defn}

The graded group structure and the differential of $F_*(M,A)$ are similar to that of the extended two-sided bar complex $B^e_*(A,A,A)$ of the associative algebra $A$ (see, e.g., \cite{gin05,etw17}). In fact, the chain complex $F_*(M,A)$ can be constructed by introducing modifications to $B^e_*(A,A,A)$ in the following way: to construct $F_q(M,A)$, we start with $B^e_{q-2}(A,A,A) \cong A^{\otimes q}$ and subsequently replace exactly one copy of $A$ with a copy of $M$ for every copy of $A$ in the tensor product. This yields $q$ tensor products of the form $A^{\otimes i-1}\otimes M\otimes A^{\otimes q-i}$ for $1 \le i \le q$, and their direct sum forms the $k$-module $F_q(M,A)$. Any multiplication with this new factor $M$ is replaced by either the left or the right multiplication of this $A$-bimodule. It is easy to see that $F_*(M,A)$ forms a well-defined chain complex in a similar manner as the bar complex $B^e_*(A,A,A)$.

\begin{prop}
$F_*(M,A)$ is an exact chain complex.
\end{prop}

\begin{proof}
$F_*(M,A)$ has a simplicial structure $X_*$ where $X_q = F_{q-2}(M,A)$. The face maps are given in the definition of the chain complex, while the degeneracy maps are defined by
\[\begin{array}{c}
    s_m(a_1 \otimes ... \otimes a_{i-1} \otimes \mu_i \otimes a_{i+1} \otimes ... \otimes a_q)= \\[5pt]
    a_1 \otimes ... \otimes a_m \otimes 1 \otimes a_{m+1} \otimes ... \otimes a_{i-1} \otimes \mu_i \otimes a_{i+1} \otimes ... \otimes a_q
\end{array}\]
for all $0 \le m \le q$. The extra degeneracy $s_0$ guarantees that the simplicial object is contractible, hence the chain complex $F_*(M,A)$ is exact.
\end{proof}

It follows that $F_*(M,A)$, upon omitting $M$ in degree $1$, gives a free resolution of $M$ as an $A$-bimodule.

Let $I := A/k$ be the \textit{augmentation ideal} of $A$, consisting of elements of positive degree. We then have a decomposition $A \cong I \oplus k$, hence given any $a \in A$ we may write $a = \overline{a} + \hat{a}$ where $\overline{a}$ and $\hat{a}$ are the projections of $a$ onto the factors $I$ and $k$, respectively. Let $F_*(M,I)$ denote the chain complex obtained by replacing all copies of $A$ in $F_*(M,A)$ with $I$ (one may think of $F_*(M,I)$ as a reduced form of the chain complex $F_*(M,A)$).

Given an associative $k$-algebra $A$, let $A^{op}$ denote the opposite algebra, and $A^e := A \otimes A^{op}$ be the enveloping algebra of $A$. There is a canonical isomorphism $(A^e)^{op} \cong A^e$, thus an $A$-bimodule can be regarded as a left (or equivalently, right) $A^e$-module \cite{gin05}. We compute the homology of the chain complex $F_*(M,I)$ in the following theorem.

\begin{thm}
$H_*(F_{*}(M,I)) \cong \mathrm{Tor}^{A^e}_{*-1}(M,k)$.
\end{thm}

\begin{proof}
Recall that the Hochschild chain complex $CH_*(A,M)$ is defined degree-wise by $CH_q(A,M) = M \otimes A^{\otimes q}$, with face maps
\[d_m (\mu \otimes a_1 \otimes ... \otimes a_q) = \begin{cases}
\mu a_1 \otimes a_2 \otimes ... \otimes a_q \quad \hspace{1.6in}\hfill m = 0\\
\mu \otimes a_1 \otimes ... \otimes a_m a_{m+1} \otimes ... \otimes a_q \hfill 1 \le m \le q-1 \\
a_q \mu \otimes a_1 \otimes ... \otimes a_{q-1} \hfill m = q\\
\end{cases}\]
and differential $d = \sum_{m=0}^{q} (-1)^m d_m$. When the $A$-bimodule $M$ is treated as a left $A^e$-module, there is an identification of the $k$-module $CH_q(A,M)$ as $M \otimes A^{\otimes q} \cong M \otimes_{A^e} A^{\otimes q+2}$, which induces an isomorphism of chain complexes $CH_*(A,M) \cong M \otimes_{A^e} B_*(A,A,A)$ where $B_*(A,A,A)$ is the two-sided bar complex \cite{gin05}.

Construct a chain map $f_*: F_*(M,I) \to CH_*(A,M)$ by mapping
\[a_1 \otimes ... \otimes a_{i-1} \otimes \mu_i \otimes a_{i+1} \otimes ... \otimes a_q \mapsto (-1)^{q(i-1)} \mu_i \otimes a_{i+1} \otimes ... \otimes a_q \otimes 1 \otimes a_1 \otimes ... \otimes a_{i-1}.\]
Note that for all $m \ne i$, $\mathrm{deg}(a_m) > 0$ since each $a_m$ is in the augmentation ideal $I$ of $A$. It follows that there is precisely one occurrence of the unit in the image of an arbitrary element of $F_*(M,I)$, and its position is determined by the indices $q$ and $i$. We deduce that this map is injective. The isomorphic image of $F_*(M,I)$ via this map forms a subcomplex of the Hochschild chain complex which in degree $q$ has the form $\bigoplus^q_{i=1} M \otimes I^{\otimes q-i} \otimes k \otimes I^{\otimes i-1}$.
Via the identification described above, we have the following isomorphisms of $k$-modules
\[ F_q(M,I) \cong \displaystyle\bigoplus^q_{i=1} M \otimes I^{\otimes q-i} \otimes k \otimes I^{\otimes i-1} \cong \displaystyle\bigoplus^q_{i=1} M \underset{A^e}{\otimes} (A \otimes I^{\otimes q-i} \otimes k \otimes I^{\otimes i-1} \otimes A)\]
and thus an isomorphism of chain complexes $F_*(M,I) \cong M \otimes_{A^e} Z_*$, where $Z_*$ is the subcomplex of the two-sided bar complex $B_*(A,A,A)$ defined degree-wise by
\[ Z_q = \displaystyle \bigoplus_{i=1}^q A \otimes I^{\otimes q-i} \otimes k \otimes I^{\otimes i-1} \otimes A \]
for all $q \ge 1$. Observe that $Z_*$ is free as a graded right $A^e$-module, so to prove the statement of the theorem, it suffices to show that $Z_*$ is a resolution of $k$.

We will show that $Z_*$ can be written as the direct sum of two subcomplexes of the two-sided bar complex $B_*(A,A,A)$ and compute their homologies. Define $X_*$ and $Y_*$ to be the subcomplexes of $B_*(A,A,A)$ defined degree-wise respectively by
\[X_q = k \otimes k \otimes I^{\otimes q-1} \otimes A\]
and
\[Y_q = (I \otimes k \otimes I^{\otimes q-1} \otimes A) \oplus \left( \displaystyle \bigoplus^{q-1}_{i=1} A \otimes I^{\otimes q-i} \otimes k \otimes I^{\otimes i-1} \otimes A \right).\]
Observe that for every $q \ge 1$, $Z_q \cong X_q \oplus Y_q$, so as chain complexes (more specifically, subcomplexes of the bar complex) indeed we have $Z_* \cong X_* \oplus Y_*$.

Define collections of maps $g_* : X_*\to X_{*+1}$ by
\[g_q (1 \otimes 1 \otimes a_2 \otimes ... \otimes a_q \otimes a_{q+1}) = (-1)^{q-1} 1 \otimes 1 \otimes a_2 \otimes ... \otimes a_q \otimes \overline{a_{q+1}} \otimes 1\]
and $h_*: Y_* \to Y_{*+1}$ by
\[h_q (a_0 \otimes a_1 \otimes ... \otimes a_q \otimes a_{q+1}) = 1 \otimes \overline{a_0} \otimes a_1 \otimes ... \otimes a_q \otimes a_{q+1}\]
where $\overline{a}$ denotes the positive-degree part of $a$.
We claim that $g_*$ satisfies the null homotopy condition for all degrees except for degree 1, i.e. $\mathrm{id} = d_{q+1} g_q + g_{q-1} d_q$ for all $q \ge 2$. We compute:
\[\begin{array}{l}
     d_{q+1} g_q(1 \otimes 1 \otimes a_2 \otimes ... \otimes a_q \otimes a_{q+1}) \\[5pt]
     = d_{q+1}\left((-1)^{q-1} 1 \otimes 1 \otimes a_2 \otimes ... \otimes a_q \otimes \overline{a_{q+1}} \otimes 1\right) \\[5pt]
     = (-1)^{q-1} \Bigg[ \displaystyle\sum_{m=2}^{q-1} \big((-1)^m 1 \otimes 1 \otimes ... \otimes a_m a_{m+1} \otimes ... \otimes \overline{a_{q+1}} \otimes 1 \big)\\[5pt]
     \quad + (-1)^q 1 \otimes 1 \otimes ... \otimes a_q \overline{a_{q+1}} \otimes 1 + (-1)^{q+1} 1 \otimes 1 \otimes ... \otimes a_q \otimes \overline{a_{q+1}} \Bigg]
\end{array}\]
and
\[\begin{array}{l}
     g_{q-1} d_q(1 \otimes 1 \otimes a_2 \otimes ... \otimes a_q \otimes a_{q+1}) \\[5pt]
     = g_{q-1}\Bigg[ \displaystyle\sum_{m=2}^{q} (-1)^m 1 \otimes 1 \otimes ... \otimes a_m a_{m+1} \otimes ... \otimes a_{q+1} \otimes 1 \Bigg]\\[15pt]
     = (-1)^{q} \Bigg[ \displaystyle\sum_{m=2}^{q-1} \big((-1)^m 1 \otimes 1 \otimes ... \otimes a_m a_{m+1} \otimes ... \otimes \overline{a_{q+1}} \otimes 1 \big) \\[5pt]
     \quad + (-1)^q 1 \otimes 1 \otimes ... \otimes \overline{a_q a_{q+1}} \otimes 1 \Bigg].
\end{array}\]
Observe that $\overline{a_q a_{q+1}} = a_q a_{q+1}$ since $\mathrm{deg}(a_q) > 0$. We can then simplify the sum of the above formulae to get
\[\begin{array}{l}
      (d_{q+1} g_q + g_{q-1} d_q)(1 \otimes 1 \otimes a_2 \otimes ... \otimes a_q \otimes a_{q+1}) \\[5pt]
      = 1 \otimes 1 \otimes ... \otimes a_q a_{q+1} \otimes 1 - 1 \otimes 1 \otimes ... \otimes a_q \overline{a_{q+1}} \otimes 1 + 1 \otimes 1 \otimes ... \otimes a_q \otimes \overline{a_{q+1}} \\[5pt]
      = 1 \otimes 1 \otimes ... \otimes a_q \widehat{a_{q+1}} \otimes 1 + 1 \otimes 1 \otimes ... \otimes a_q \otimes \overline{a_{q+1}} \\[5pt]
      =  1 \otimes 1 \otimes ... \otimes a_q \otimes \widehat{a_{q+1}} + 1 \otimes 1 \otimes ... \otimes a_q \otimes \overline{a_{q+1}} = 1 \otimes 1 \otimes ... \otimes a_q \otimes a_{q+1}.
\end{array}\]
The existence of the (almost) null homotopy $g_*$ implies that $H_q(X_*)$ vanishes for all $q \ge 2$. The computation of $H_1(X_*)$ is straighforward:
\[H_1 (X_*) = \dfrac{\mathrm{ker} (d_1)}{\mathrm{im} (d_2)} = \dfrac{k \otimes k \otimes A}{k \otimes k \otimes I} \cong \dfrac{A}{I} \cong k.\]
Hence the chain complex $X_*$ is in fact a resolution of $k$. A similar computation shows that $h_*$ is a null homotopy and thus the complex $Y_*$ is exact. It follows from the direct sum decomposition above that $Z_*$ is a resolution of $k$, concluding our proof.
\end{proof}

If $M$ is graded, for an element $f = a_1 \otimes ... \otimes \mu_i \otimes ... \otimes a_q$ with $a_m$ homogeneous elements of $A$ of degree deg$(a_m)$, we may define the degree of $f$ to be $\mathrm{deg}(f) := \mathrm{deg}(\mu_i) + \sum_{m\ne i} \mathrm{deg}(a_m)$. The differential in $F_*(M,I)$ strictly preserves the degree of elements, hence we may define the split subcomplex generated by homogeneous elements of $F_*(M,I)$ of degree precisely $n$, denoted by $F_{*,n}(M,I)$.

\begin{cor}\label{cor:hom_Fn}
$H_*(F_{*,n}(M,I)) \cong \mathrm{Tor}^{A^e}_{*-1,n}(M,k)$.
\end{cor}

\begin{proof}
Since the differentials on both the complex $F_*(M,I)$ and the Hochschild chain complex strictly preserve the degree of elements, this corollary follows directly from the previous theorem.
\end{proof}

To state the main result of the paper, we need to define a compatible $\mathfrak{A}$-bimodule where $\mathfrak{A}$ is a quantum shuffle algebra.

\begin{defn}\label{defn:M}
Let $(V,W,\sigma,\tau,\varphi)$ be a separable left-braided vector space. The \textit{graded} $\mathfrak{A}(V)$\textit{-bimodule} $\mathfrak{M} = \mathfrak{M}(V,W)$ is defined by
\[ \mathfrak{M} = \displaystyle \bigoplus_{q \ge 1} \bigoplus_{0 \le j \le q-1} V^{\otimes j} \otimes W \otimes V^{\otimes q-j-1}.\]
Multiplication on both sides is given by the quantum shuffle product in the following sense: for the left multiplication we have
\[[a_1 | ... | a_p] \star [b_1 | ... | b_j |w|b_{j+2}|...| b_q] = \sum_{\gamma} \tilde{\gamma} [a_1 | ... | a_p | b_1 | ... | b_j | w | b_{j+2} | ... | b_q]\]
and for the right multiplication
\[\displaystyle [a_1 | ... |a_j|w|a_{j+2}|... | a_p] \star [b_1 | ... | b_q] = \sum_{\gamma}  \tilde{\gamma} [a_1 | ...|a_j|w|a_{j+2}|... | a_p | b_1 | ... | b_q]\]
where $\gamma$ draws from all $(p,q)$-shuffles, and $\tilde{\gamma}$ is the lift of $\gamma$ to $A_{p+q}$.
\end{defn}

The action of the braid $\tilde{\gamma}$ on an element in $V^{\otimes j} \otimes W \otimes V^{\otimes q-j-1}$ is described in Corollary~\ref{cor:ind_rep_sep_lbvs}. The formulae of both the left and the right multiplications of $\mathfrak{M}$ resemble that of the quantum shuffle product in $\mathfrak{A}(V)$. In particular, the definition of $\mathfrak{M}$ satisfies the associative requirement of multiplication for an $\mathfrak{A}(V)$-bimodule.


\subsection{Homology of type-\textit{B} Artin groups}\label{ssec:main}

Write $\epsilon$ for the braided $k$-module $\epsilon = k$ with braiding on $\epsilon^{\otimes 2} \cong k$ given by multiplication by $-1$. For a general braided $k$-module $(V,\sigma)$, write $V_\epsilon = V \otimes \epsilon$ with braiding twisted by the sign on $\epsilon$.

Let $(V,W,\sigma,\tau,\varphi)$ be a separable left-braided vector space, then its twisted dual $(V_\epsilon, W, \sigma_\epsilon, \tau, \varphi_\epsilon)$ also forms a separable left-braided vector space with $\varphi_\epsilon := -\varphi$. Let $\mathfrak{A} := \mathfrak{A}(V_\epsilon)$ be the quantum shuffle algebra generated by the twisted dual $V_\epsilon$, and $\mathfrak{M}:=\mathfrak{M}(V_\epsilon,W)$ be the $\mathfrak{A}$-bimodule defined from the left-braided vector space $(V_\epsilon, W, \sigma_\epsilon, \tau, \varphi_\epsilon)$ as described in Definition~\ref{defn:M}. Let $\mathfrak{I}$ be the augmentation ideal of $\mathfrak{A}$, consisting of elements of positive degree. The following proposition shows the relationship between the algebraic structures we developed in the previous subsection and the cellular chain complex of configuration spaces with twisted coefficients.

\begin{prop}\label{prop:Fn=Cn}
There is an isomorphism of chain complexes
\[F_{*,n+1}(\mathfrak{M},\mathfrak{I}) \cong C(n+1)_{*+n+1} \otimes \mathrm{Ind}^{A_{n+1}}_{B_n}(V^{\otimes n}\otimes W).\]
\end{prop}

\begin{proof}
Recall that the induced representation $\mathrm{Ind}^{A_{n+1}}_{B_n} (V^{\otimes n}\otimes W)$ is isomorphic to the direct sum $\bigoplus_{i=1}^{n+1} V^{\otimes i-1} \otimes W \otimes V^{\otimes n-i+1}$. Observe that $F_{q,n+1}(\mathfrak{M},\mathfrak{I})$ consists of all spaces of the form
\[ V_\epsilon^{\otimes \lambda_1} \otimes ... \otimes V_\epsilon^{\otimes \lambda_{i-1}} \otimes (V_\epsilon^{\otimes j} \otimes W \otimes V_\epsilon^{\otimes \lambda_i - j - 1}) \otimes V_\epsilon^{\otimes \lambda_{i+1}} \otimes ... \otimes V_\epsilon^{\lambda_{q}} \]
where $\sum \lambda_m = n+1$. This is an ordered partition of $n+1$ with $q$ parts labelled by an element of $V_\epsilon^{\iota-1} \otimes W \otimes V_\epsilon^{\otimes n+1-\iota} \cong V^{\iota-1} \otimes W \otimes V^{\otimes n+1-\iota}$, where $\iota = j+1+\sum_{m=1}^{i-1} \lambda_m$ is the overall position of the factor $W$ in the tensor product. Hence there is an isomorphism of $k$-modules between $F_{q,n+1}(\mathfrak{M},\mathfrak{I})$ and $C(n+1)_{q+n+1} \otimes \mathrm{Ind}^{A_{n+1}}_{B_n}(V^{\otimes n}\otimes W)$.

There are two main pieces of data in the boundary of an element $\lambda \otimes t$ in the chain complex $C(n+1)_* \otimes \mathrm{Ind}^{A_{n+1}}_{B_n}(V^{\otimes n}\otimes W)$ (stated in its general form prior to Theorem~\ref{thm:etw_fnf}): the coarsening $\rho^i$ of $\lambda$ and the signed sum over all $(\lambda_i,\lambda_{i+1})$-shuffles of the actions of their lifts on $t$. Both are encapsulated in the differential of $F_{*,n+1}(\mathfrak{M},\mathfrak{I})$: the coarsening $\rho^i$ is encoded in the choice of two multiplied elements, and the sum of the braid actions is contained in the quantum shuffle product of $\mathfrak{A}$ or the multiplication of $\mathfrak{M}$ by $\mathfrak{A}$. Observe that in the differential of $F_{*,n+1} (\mathfrak{M},\mathfrak{I})$, the braids act only on a tensor subfactor of $V^{\otimes \iota-1} \otimes W \otimes V^{\otimes n+1-\iota}$, whereas the corresponding braids act on the isomorphic image $\alpha_\iota(V^{\otimes n} \otimes W)$ of this full factor in the differential of $C(n+1)_{*+n+1} \otimes \mathrm{Ind}^{A_{n+1}}_{B_n}(V^{\otimes n}\otimes W)$. These braid actions match precisely due to the commutativity of Diagram~\ref{diag:separable}, which is equivalent to the separability of the left-braided vector space $(V_\epsilon, W)$ by Proposition~\ref{prop:comm_diag_cond}. The signs coming from $\epsilon$ encode the boundary orientations on cells in the Fox-Neuwirth/Fuks model. Via these identifications, the differentials of $F_{*,n+1}(\mathfrak{M},\mathfrak{I})$ and $C(n+1)_{*+n+1} \otimes \mathrm{Ind}^{A_{n+1}}_{B_n}(V^{\otimes n}\otimes W)$ are precisely the same formula, which shows their isomorphism as chain complexes.
\end{proof}

We are ready to prove the main result of the paper.

\begin{cor}\label{cor:main}
Let $\mathfrak{A} = \mathfrak{A}(V^*_\epsilon)$ and $\mathfrak{M} = \mathfrak{M}(V^*_\epsilon, W^*)$. Then there is an isomorphism
\[H_* (B_n; V^{\otimes n} \otimes W) \cong \mathrm{Ext}^{n-*, n+1}_{\mathfrak{A}^e} (\mathfrak{M},k).\]
\end{cor}

\begin{proof}
We invoke Corollary~\ref{cor:etw_rewrite} for the $B_n$-representation $L = V^{\otimes n} \otimes W$, Proposition~\ref{prop:Fn=Cn}, Corollary~\ref{cor:hom_Fn}, and the duality of the Tor and Ext functors subsequently to get the desired isomorphism
\[\begin{split}
    H_*(B_n; V^{\otimes n} \otimes W) & \cong H_{2n+2-*} \Big( C(n+1)_* \otimes \mathrm{Ind}^{A_{n+1}}_{B_n} ((V^*)^{\otimes n} \otimes W^*) \Big)^*\\
    & \cong H_{n+1-*} (F_{*,n+1}(\mathfrak{M}, \mathfrak{I}))^*\\
    & \cong \mathrm{Tor}^{\mathfrak{A}^e}_{n-*, n+1}(\mathfrak{M}, k)^* \cong \mathrm{Ext}_{\mathfrak{A}^e}^{n-*, n+1} (\mathfrak{M}, k).
\end{split}\]
\end{proof}

This result creates a new avenue to compute the homology of type-{\it B} Artin groups with coefficients arising from braided vector spaces by studying the cohomology of bimodules over quantum shuffle algebras using tools from homological algebra. In the next section, we demonstrate the power of the main theorem by giving a complete characterization of the homology of these groups with coefficients in one-dimensional representations over a field of characteristic $0$. In general, one can reduce the computation of the cohomology of $\mathfrak{M}$ to that of the cohomology of $\mathfrak{A}$ using a series of spectral sequences, which is still fairly difficult to compute. However, in certain specific cases, a bound on the cohomology of quantum shuffle algebras was given by Ellenberg--Tran--Westerland \cite{etw17}. In work in progress, we leverage their result to give a bound on the cohomology of $\mathfrak{M}$ and study variations of Malle's conjecture for function fields.


\subsection{Example: Module over the quantum divided power algebra}\label{ssec:1dim}

For the only concrete computation in this paper, we revisit Example~\ref{exmp:V=W=k} when the base field $k$ has characteristic $0$. Let the left-braided vector space $(V,W,\sigma,\tau)$ be composed of one-dimensional $k$-vector spaces $V=k$ and $W = k\{w\}$, with braidings $\sigma$ on $V\otimes V = k$ and $\tau$ on $V\otimes W \cong k$ given by multiplications by $q$ and $p$ respectively, for some $p, q \in k^\times$. Thus the braid action of $B_n$ on the representation $V^{\otimes n} \otimes W \cong k$ is given by $\sigma_i \mapsto q$ for all $1 \le i \le n-1$ and $\tau_n \mapsto p$. In this case, the quantum shuffle algebra $\mathfrak{A} = \mathfrak{A}(V)$ for the braided vector space $(V,\sigma)$ is generated (as a $k$-module) by the classes $x_n = [1|...|1]_n$, where there are $n$ occurrences of $1$. It has been shown that the algebra $\mathfrak{A}$ is isomorphic as graded rings to the quantum divided power algebra $\Gamma_q[x]$ \cite{etw17}, whose structure has been previously studied by \cite{cal06}.

\begin{defn}
    The \textit{quantum divided power algebra} $\Gamma_q[x]$ associated to $q \in k^\times$ is additively generated by elements $x_n$ in degree $n$, equipped with the product
    \[ x_n \star x_m := {n+m \choose m}_q x_{n+m} \]
    where the quantum binomial coefficient is defined by
    \[ {a \choose b}_q = \dfrac{[a]_q [a-1]_q \cdots [a-b+1]_q}{[b]_q[b-1]_q \cdots [1]_q}; \quad \text{here} \quad [r]_q = \dfrac{1-q^r}{1-q} = 1+ q + \cdots + q^{r-1}. \]
\end{defn}

The isomorphism between $\Gamma_q[x]$ and $\mathfrak{A}$ sends the class $x_n$ to $[1|...|1]_n$ in $\mathfrak{A}$ \cite{etw17}. The following identification of the algebra $\Gamma_q[x]$ is due to Callegaro.

\begin{prop}[\cite{cal06}]\label{prop:q_div_pow_alg}
If $q$ is not a root of unity in $k$, then there is an isomorphism $\Gamma_q[x] \cong k[x_1]$.  If $q$ is a primitive $m^\mathrm{th}$ root of unity, then 
\[\Gamma_q[x] = k[x_1]/x_1^{m} \otimes \Gamma[x_m].\]
\end{prop}

As analyzed in Example~\ref{exmp:V=W=k}, the left-braided vector space $(V,W,\sigma,\tau)$ is separable with the separated braiding $\varphi'$ given simply by permuting tensor factors. Observe that this separated braiding is not unique: multiplying $\varphi'$ by an arbitrary unit gives another separated braiding. For generality, we will choose the separated braiding to be $\varphi := u\varphi'$ for some $u \in k^{\times}$.

Recall that the $\mathfrak{A}$-bimodule $\mathfrak{M}$ is defined by
\[ \mathfrak{M} = \displaystyle \bigoplus_{n \ge 1} \bigoplus_{1 \le i \le n} V^{\otimes i-1} \otimes W \otimes V^{\otimes n-i}.\] In this case, $\mathfrak{M}$ as a $k$-module is generated by the classes $y_{i,n} := [1|...|1|w^{(i)}|1|...|1]_n$ (where $i$ records the position of $w$) for all $n \ge 1$ and $1 \le i \le n$; in particular, denote $y_n := y_{n,n} = [1|...|1|w]_n$. Corollary~\ref{cor:ind_rep_sep_lbvs} specifies the action of $A_n$ on these generators as follows:
\[ \sigma_m (y_{i,n}) = \begin{cases}
q y_{i,n} \hspace{.5in} \hfill m \ne i-1, i \\
u y_{i-1,n} \hfill m = i-1 \\
pu^{-1} y_{i+1,n} \hfill m = i.
\end{cases}\]
For the remainder of this section, we will study the structure of the module $\mathfrak{M}$ as an $\mathfrak{A}$-bimodule and compute its homology.

Consider the left multiplicative structure of $\mathfrak{M}$. Recall that $\{y_n, y_{n-1,n}, ..., y_{1,n} \}$ forms a basis for the degree-$n$ homogeneous subspace $\mathfrak{M}_n$ of $\mathfrak{M}$.

\begin{lem}\label{lem:xy_to_y}
For $n,m \ge 1$,
\[ x_m y_n = \sum^m_{h=0} u^{m-h} q^{(n-1)(m-h)} {n-1+h \choose h }_q y_{n+h,n+m}.\]
\end{lem}

The proof of this lemma involves the concept of $((p,h),q,j)$-\textit{shuffles}, which will be later developed in Section~\ref{ssec:FNF_Cx}, in particular Definition~\ref{defn:pqhj_shuf} and Proposition~\ref{prop:phqj_decomp}. For convenience, we will apply these results and defer their proofs to the discussion in that section.

\begin{figure}[t]
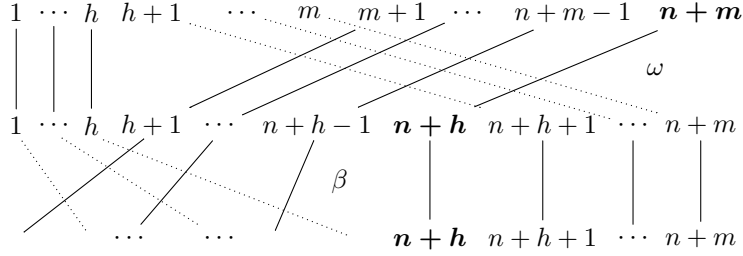

    \centering
    \[ 
	\xy
		{\ar@{-} (0,15)*+{1}; (0,30)*+{1} };
		{\ar@{-} (5,15)*+{\cdots}; (5,30)*+{\cdots} }; 
		{\ar@{-} (10,15)*+{h}; (10,30)*+{h} }; 
		{\ar@{-} (18,15)*+{h+1}; (50,30)*+{m+1} };
		{\ar@{-} (27,15)*+{\cdots}; (60,30)*+{\cdots} };
		{\ar@{-} (40,15)*+{n+h-1}; (74,30)*+{n+m-1} };
		{\ar@{-} (55,15)*+{\boldsymbol{n+h}}; (91,30)*+{\boldsymbol{n+m}} };
		{\ar@{..} (18,30)*+{h+1}; (70,15)*+{n+h+1} };
		{\ar@{..} (30,30)*+{\cdots}; (82,15)*+{\cdots} };
		{\ar@{..} (39,30)*+{m}; (91,15)*+{n+m} };
		{\ar@{-} (55,14)*+{}; (55,0)*+{\boldsymbol{n+h}} };
		{\ar@{..} (0,14)*+{}; (10,0)*+{} };
		{\ar@{..} (5,14)*+{}; (27,0)*+{\cdots} }; 
		{\ar@{..} (10.5,13.5)*+{}; (45,0)*+{} };
		{\ar@{-} (18,14)*+{}; (0,0)*+{} };
		{\ar@{-} (40,14)*+{}; (34,0)*+{} };
		{\ar@{-} (27,14)*+{}; (15,0)*+{\cdots} };
		{\ar@{-} (70,14)*+{}; (70,0)*+{n+h+1} };
		{\ar@{-} (91,14)*+{}; (91,0)*+{n+m} };
		{\ar@{-} (82,14)*+{}; (82,0)*+{\cdots} };
		{(85,22.5)*+{\omega}};
		{(43,7.5)*+{\beta}};
	\endxy
\]
    \caption{Decomposition of an $((m,h),n,n-1)$-shuffle $\gamma$ into a sequence of two permutations: the permutation $\omega$, followed by an $(h,n-1)$-shuffle $\beta$ on the integer interval $\llbracket 1, n+h-1 \rrbracket$. The lift $\tilde{\gamma}$ is the product of lifts of the component permutations to $B_{n+m}$, i.e. $\tilde{\gamma} = \tilde{\beta} \tilde{\omega}$. The factor $w$ (in bold) is sent from $n+m$ to $n+h$.}
    \label{fig:phqj_decomp_1dim}
\end{figure}

\begin{proof}
The left multiplication formula of $\mathfrak{M}$ gives
\[ x_m y_n = [1|...|1]_m \star [1|...|1|w]_n = \sum_\gamma \tilde{\gamma}[1|...|1|w]_{n+m} = \sum_\gamma \tilde{\gamma} (y_{n+m}) \]
where $\gamma$ draws from all $(m,n)$-shuffles. Given a shuffle $\gamma$, the summand of $\mathfrak{M}_{n+m}$ that $\tilde{\gamma} (y_{n+m})$ belongs to is determined by the image of $n+m$ under the shuffle $\gamma$, i.e. $\tilde{\gamma}(y_{n+m}) \in V^{\otimes i-1} \otimes W \otimes V^{\otimes n+m-i}$ where $i = \gamma(n+m)$. Such a shuffle $\gamma$ belongs to a family of $((p,h),q,j)$-shuffles; in particular, $\gamma$ is an $((m,h),n,n-1)$-shuffle, for $h = i-n$. The decomposition of $((p,h),q,j)$-shuffles then states that $\gamma$ can be uniquely decomposed into a sequence of three permutations: a permutation $\omega$ that sends the integer interval $\llbracket h+1, m \rrbracket$ to $\llbracket n+h+1, n+m \rrbracket$ (while preserving its order) and $n+m$ to $n+h$, followed by an $(h,n-1)$-shuffle $\beta$ on $\llbracket 1, h+n-1 \rrbracket$ and an $(m-h, 0)$-shuffle on $\llbracket n+h+1, n+m \rrbracket$ (i.e. the identity). Observe that the lift $\tilde{\omega}$ can be written as the product of $n$ braids, each of which represents the lift of the shuffle that moves $m-h$ points in the original interval $\llbracket h+1, m \rrbracket$ past the next point on the right (see Figure~\ref{fig:phqj_decomp_1dim}). The action of the first $n-1$ such braids each results in the multiplication by $q^{m-h}$, while the last braid element is precisely $\alpha_{n+h}$ which acts by $\varphi_{n+h,n+m} = \varphi_{n+h} ... \varphi_{n+m-1}$ on $y_{n+m}$. It follows that
\[ \begin{split}
    x_m y_n & = \sum_\gamma \tilde{\gamma} (y_{n+m}) = \sum_{h=0}^m \sum_\beta q^{(n-1)(m-h)} \tilde{\beta} \varphi_{n+h,n+m}(y_{n+m}) \\
    & = \sum_{h=0}^m u^{m-h} q^{(n-1)(m-h)} {n-1+h \choose h}_q y_{n+h,n+m}
\end{split} \]
where $\beta$ draws from all $(h, n-1)$-shuffles. The fact that the weighted sum over all $(h,n-1)$-shuffles $\beta$ of $q^{cr(\beta)}$, where $cr(\beta)$ is the number of crossings in $\beta$, computes the quantum binomial coefficient is well established (see, e.g., Prop. 1.7.1 in \cite{sta12}).
\end{proof}

\begin{prop}\label{prop:M_fr_left}
The set $\{x_{n-k} y_k\}^{n}_{k=1}$ forms a basis for $\mathfrak{M}_n$. Consequently, $\mathfrak{M}$ is a free left $\mathfrak{A}$-module with respect to the basis $\mathcal{Y} = \{y_1, y_2, ... \}$.
\end{prop}

\begin{proof}
Observe that for any given $1 \le k \le  n$, the coefficient of $y_{i,n}$ in the expansion of $x_{n-k} y_k$ given by the previous lemma vanishes for all $1 \le i \le k-1$, and that of $y_{k,n}$ is $u^{n-k}q^{(k-1)(n-k)}$. Hence there is a change-of-basis matrix from the basis $\{y_{k,n}\}_{k=1}^{n}$ to $\{ x_{n-k} y_k \}_{k=1}^{n}$
\begin{center}
{\renewcommand{\arraystretch}{1.5}%
\begin{tabular}{ c|c c c c c } 
  & $y_{n}$ & $x_1 y_{n-1}$ & $x_2 y_{n-2}$ & $\cdots$ & $x_{n-1} y_1$ \\
 \hline
 $y_{n}$ & 1 & * & * & $\cdots$ & * \\ 
 $y_{n-1,n}$ & 0 & $uq^{n-2}$ & * & $\cdots$ & * \\
 $y_{n-2,n}$ & 0 & 0 & $u^2 q^{2(n-3)}$ & $\cdots$ & * \\
 $\vdots$ & $\vdots$ & $\vdots$ & $\vdots$ & $\ddots$ & $\vdots$ \\
 $y_{1,n}$ & 0 & 0 & 0 & $\cdots$ & $u^{n-1}$ \\
\end{tabular}}
\end{center}
which is upper-triangular with all non-zero diagonal entries. This matrix is therefore invertible, which implies that $\{ x_{n-k} y_k \}_{k=1}^{n}$ indeed forms a basis for $\mathfrak{M}_n$. It follows that the $\mathfrak{A}$-module $\mathfrak{M}$ is generated (as a $k$-module) by the basis $\bigcup_{n=1}^\infty \{ x_{n-k} y_k \}_{k=1}^{n}$. This basis can be freely generated from the set $\mathcal{Y}$ by the left multiplication by generators of the algebra $\mathfrak{A}$.
\end{proof}

We now move onto the right multiplicative structure of $\mathfrak{M}$.

\begin{lem}\label{lem:yx_to_y}
For $n,m \ge 1$,
\[ y_n x_m = \sum^m_{h=0} \left(\frac{p}{u}\right)^h {n-1+h \choose h }_q y_{n+h,n+m}.\]
\end{lem}

Similar to that of Lemma~\ref{lem:xy_to_y}, the proof of this lemma requires an introduction of the $(p,(q,h),j)$-\textit{shuffles} that will later be studied extensively in Section~\ref{ssec:FNF_Cx}, particularly Definition~\ref{defn:pqhj_shuf} and Proposition~\ref{prop:pqhj_decomp}.

\begin{proof}
The right multiplication formula of $\mathfrak{M}$ gives
\[ y_n x_m = [1|...|1|w]_n \star [1|...|1]_m = \sum_\gamma \tilde{\gamma}[1|...|1|w^{(n)}|1|...|1]_{n+m} = \sum_\gamma \tilde{\gamma} (y_{n,n+m}) \]
where $\gamma$ draws from all $(n,m)$-shuffles. Similar to the case with left multiplication, a shuffle $\gamma$ in this case belongs to a family of $(p,(q,h),j)$-shuffles; in particular, if $\gamma$ sends $n$ to $i$ ($n \le i \le n+m$), it is classified as an $(n,(m,h),n-1)$-shuffle, for $h = i-n$.
By applying the decomposition of $(p,(q,h),j)$-shuffles, we observe that $\gamma$ can be uniquely decomposed into a sequence of two permutations: a permutation $\omega$ that sends the integer interval $\llbracket n+1, n+h \rrbracket$ to $\llbracket n, n+h-1 \rrbracket$ (while preserving its order) and $n$ to $n+h$, followed by an $(n-1,h)$-shuffle $\beta$ on $\llbracket 1, n+h-1 \rrbracket$ (see Figure~\ref{fig:pqhj_decomp_1dim}). The lift $\tilde{\omega}$ can be written as the product $\tilde{\omega} = \sigma_{n+h-1} \sigma_{n+h-2} ... \sigma_n$. Observe that $\sigma_n (y_{n,n+m}) = p u^{-1} y_{n+1,n+m}$, so by successive multiplication, it follows that $\tilde{\omega} (y_{n,n+m}) = \sigma_{n+h-1} \sigma_{n+h-2} ... \sigma_n (y_{n,n+m}) = (pu^{-1})^h y_{n+h,n+m}$. Thus
\[ \begin{array}{r l}
y_n x_m & = \displaystyle \sum_\gamma \tilde{\gamma} (y_{n,n+m}) = \sum_{h=0}^m \sum_\beta \left(\frac{p}{u}\right)^h \tilde{\beta} y_{n+h,n+m}\\[15pt]
& = \displaystyle \sum_{h=0}^m \left(\frac{p}{u}\right)^h {n-1+h \choose h}_q y_{n+h,n+m}
\end{array} \]
where $\beta$ draws from all $(n-1,h)$-shuffles. The proof is complete.
\end{proof}

\begin{figure}[t]
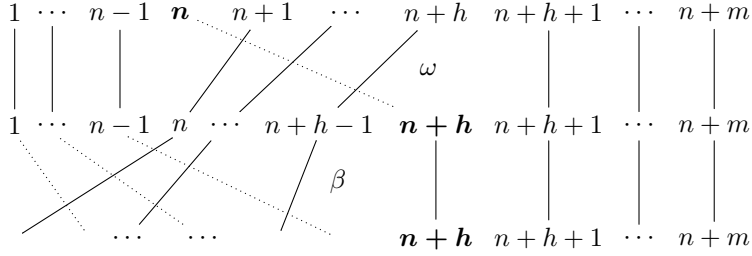

    \centering
    \[ 
	\xy
		{\ar@{-} (0,15)*+{1}; (0,30)*+{1} };
		{\ar@{-} (5,15)*+{\cdots}; (5,30)*+{\cdots} }; 
		{\ar@{-} (14,15)*+{n-1}; (14,30)*+{n-1} }; 
		{\ar@{-} (22,15)*+{n}; (33,30)*+{n+1} };
		{\ar@{..} (22,30)*+{\boldsymbol{n}}; (56,15)*+{\boldsymbol{n+h}} };
		{\ar@{-} (28,15)*+{\cdots}; (44,30)*+{\cdots} };
		{\ar@{-} (40.5,15)*+{n+h-1}; (56,30)*+{n+h} };
		{\ar@{-} (71,15)*+{n+h+1}; (71,30)*+{n+h+1} };
		{\ar@{-} (83,15)*+{\cdots}; (83,30)*+{\cdots} };
		{\ar@{-} (93,15)*+{n+m}; (93,30)*+{n+m} };
		{\ar@{-} (56,14)*+{}; (56,0)*+{\boldsymbol{n+h}} };
		{\ar@{..} (0,14)*+{}; (10,0)*+{} };
		{\ar@{..} (5,14)*+{}; (25,0)*+{\cdots} }; 
		{\ar@{..} (14,14)*+{}; (43,0)*+{} };
		{\ar@{-} (22,14)*+{}; (0,0)*+{} };
		{\ar@{-} (40.5,14)*+{}; (35,0)*+{} };
		{\ar@{-} (27,14)*+{}; (15,0)*+{\cdots} };
		{\ar@{-} (71,14)*+{}; (71,0)*+{n+h+1} };
		{\ar@{-} (83,14)*+{}; (83,0)*+{\cdots} };
		{\ar@{-} (93,14)*+{}; (93,0)*+{n+m} };
		{(55,22.5)*+{\omega}};
		{(43,7.5)*+{\beta}};
	\endxy
\]
    \caption{Decomposition of an $(n,(m,h),n-1)$-shuffle $\gamma$ into a sequence of two permutations: the permutation $\omega$, followed by a $(n-1,h)$-shuffle $\beta$ on the integer interval $\llbracket 1,n+h-1 \rrbracket$. The lift $\tilde{\gamma}$ is given by $\tilde{\gamma} = \tilde{\beta}\tilde{\omega}$. The factor $w$ (in bold) is mapped from $n$ to $n+h$.}
    \label{fig:pqhj_decomp_1dim}
\end{figure}

\begin{prop} The set $\{y_k x_{n-k}\}^{n}_{k=1}$ forms a basis for $\mathfrak{M}_n$. Consequently, $\mathfrak{M}$ is a free right $\mathfrak{A}$-module with respect to the basis $\mathcal{Y}$.
\end{prop}

\begin{proof}
This proof follows the same logic as the proof of Proposition~\ref{prop:M_fr_left}. Observe that for any given $1 \le k \le  n$, the coefficient of $y_{i,n}$ in the expansion of $y_k x_{n-k}$ given by Lemma~\ref{lem:yx_to_y} vanishes for all $1 \le i \le k-1$, and that of $y_{k,n}$ is $(pu^{-1})^0 = 1$. Thus the change-of-basis matrix from the basis $\{y_{k,n}\}_{k=1}^{n}$ to $\{ y_k x_{n-k} \}_{k=1}^{n}$ is given by
\begin{center}
{\renewcommand{\arraystretch}{1.5}%
\begin{tabular}{ c|c c c c c } 
  & $y_{n}$ & $y_{n-1} x_1$ & $ y_{n-2} x_2$ & $\cdots$ & $y_1 x_{n-1}$ \\
 \hline
 $y_{n}$ & 1 & * & * & $\cdots$ & * \\ 
 $y_{n-1,n}$ & 0 & $1$ & * & $\cdots$ & * \\
 $y_{n-2,n}$ & 0 & 0 & $1$ & $\cdots$ & * \\
 $\vdots$ & $\vdots$ & $\vdots$ & $\vdots$ & $\ddots$ & $\vdots$ \\
 $y_{1,n}$ & 0 & 0 & 0 & $\cdots$ & 1 \\
\end{tabular}}
\end{center}
which is upper-triangular with all non-zero diagonal entries and hence invertible. It follows that $\{ y_k x_{n-k} \}_{k=1}^{n}$ forms a basis for $\mathfrak{M}_n$, and more generally the $\mathfrak{A}$-module $\mathfrak{M}$ is generated (as a $k$-module) by the basis $\bigcup_{n=1}^\infty \{ y_k x_{n-k} \}_{k=1}^{n}$. This basis of $\mathfrak{M}$ can be freely generated from the set $\mathcal{Y}$ by the right multiplication by generators of the algebra $\mathfrak{A}$.
\end{proof}

So far in this discussion, we have explored three different bases for the module $\mathfrak{M}$ as a $k$-module: $\{y_{k,n}\}$, $\{ x_{n-k} y_k \}$, and $\{ y_k x_{n-k}\}$. The following lemma establishes a formula for the change of basis between the latter two.

\begin{lem}\label{lem:yx_to_xy}
For $n, m \ge 1$,
\[ y_n x_m = \frac{1}{u^m}\sum_{h=0}^m \left[ \dfrac{1}{q^{(m-h)(n-1+h)}} {n-1+h \choose h}_q \prod_{k=0}^{h-1} \left( p - \dfrac{1}{q^{n-1+k}} \right) \right]  x_{m-h}y_{n+h}. \]
\end{lem}

\begin{proof}
We will prove this formula by induction on $m$. For $m = 1$, first we apply Lemma~\ref{lem:xy_to_y} to get
\[ x_1 y_n = \sum^1_{h=0} u^{1-h}q^{(n-1)(1-h)} {n-1+h \choose h }_q y_{n+h,n+1} = {n \choose 1 }_q y_{n+1} + uq^{n-1} y_{n,n+1}\]
and apply Lemma~\ref{lem:yx_to_y} to get
\[ y_n x_1 = \sum^1_{h=0} \left(\frac{p}{u}\right)^h {n-1+h \choose h }_q y_{n+h,n+1} = \frac{p}{u} {n \choose 1}_q y_{n+1} + y_{n,n+1}.\]
By solving for $y_{n,n+1}$ from the first equation and substitute it into the second equation, we obtain
\[ \begin{split} 
y_n x_1 & = \frac{p}{u} {n \choose 1}_q y_{n+1} + \dfrac{1}{uq^{n-1}} \left( x_1 y_n - {n \choose 1}_q y_{n+1} \right) \\
& = \frac{1}{u} \left[ {n \choose 1}_q \left(p - \frac{1}{q^{n-1}} \right) y_{n+1} + \dfrac{1}{q^{n-1}} x_1 y_n \right]
\end{split}\]
which proves the base case.

Assume that the formula holds for all $1 \le k \le m$. The key observation here is that we can write $x_{m+1} = \frac{1}{[m+1]_q} x_m x_1$ in the quantum divided power algebra $\Gamma_q [x]$. By applying the induction hypothesis, we have
\[ \begin{array}{l}
y_n x_{m+1} = \dfrac{1}{[m+1]_q} (y_n x_m) x_1 \\[15pt]
= \dfrac{1}{[m+1]_q} \dfrac{1}{u^m}\displaystyle\sum_{h=0}^m \left[ \dfrac{1}{q^{(m-h)(n-1+h)}} {n-1+h \choose h}_q \prod_{k=0}^{h-1} \left( p - \dfrac{1}{q^{n-1+k}} \right) \right]  x_{m-h}y_{n+h} x_1 \\[15pt]
= \dfrac{1}{[m+1]_q} \dfrac{1}{u^m} \displaystyle\sum_{h=0}^m \left\{ \left[ \dfrac{1}{q^{(m-h)(n-1+h)}} {n-1+h \choose h}_q \prod_{k=0}^{h-1} \left( p - \dfrac{1}{q^{n-1+k}} \right) \right]  x_{m-h} \right. \\[15pt]
\quad \left. \dfrac{1}{u}\left[ \displaystyle {n+h \choose 1}_q \left(p - \frac{1}{q^{n+h-1}} \right) y_{n+h+1} + \dfrac{1}{q^{n+h-1}} x_1 y_{n+h} \right] \right\} \\[15pt]
= \dfrac{1}{[m+1]_q} \dfrac{1}{u^{m+1}} \displaystyle\sum_{h=0}^m \left\{ \left[ \dfrac{[n+h]_q}{q^{(m-h)(n-1+h)}} {n-1+h \choose h}_q \prod_{k=0}^{h} \left( p - \dfrac{1}{q^{n-1+k}} \right) \right]  x_{m-h} \right. \\[15pt]
\quad \left. y_{n+h+1} \displaystyle  + \left[ \dfrac{[m-h+1]_q}{q^{(m-h+1)(n-1+h)}} {n-1+h \choose h}_q \prod_{k=0}^{h-1} \left( p - \dfrac{1}{q^{n-1+k}} \right) \right] x_{m-h+1} y_{n+h} \right\}.
\end{array}\]
This is an expansion of $y_n x_{m+1}$ in terms of the basis $\{ x_{m+1-h} y_{n+h} \}_{h=0}^{m+1}$. In particular, the coefficient of $x_{m+1-h} y_{n+h}$ comes from those of the terms $x_{m-{h-1}}y_{n+{h-1}+1}$ and $x_{m+1-h}y_{n+h}$:
\[ \begin{array}{l}
\dfrac{1}{[m+1]_q} \dfrac{1}{u^{m+1}}\left\{ \displaystyle \left[ \dfrac{[n+h-1]_q}{q^{(m-(h-1))(n-1+h-1)}} {n-1+h-1 \choose h-1}_q \prod_{k=0}^{h-1} \left( p - \dfrac{1}{q^{n-1+k}} \right) \right] \right. \\[15pt]
\quad \left. \displaystyle  + \left[ \dfrac{[m-h+1]_q}{q^{(m-h+1)(n-1+h)}} {n-1+h \choose h}_q \prod_{k=0}^{h-1} \left( p - \dfrac{1}{q^{n-1+k}} \right) \right] \right\} \\[15pt]
= \dfrac{q^{m+1-h} [h]_q + [m-h+1]_q}{u^{m+1} [m+1]_q} \dfrac{1}{q^{(m+1-h)(n-1+h)}} \displaystyle {n-1+h \choose h}_q \prod_{k=0}^{h-1} \left( p - \dfrac{1}{q^{n-1+k}} \right).
\end{array}\]
Finally, observe that $q^{m+1-h} [h]_q + [m-h+1]_q = q^{m+1-h} (1+ q + ...+ q^{h-1}) + (1 + q + ... + q^{m-h}) = 1 + q + ... + q^{m-h} + q^{m-h+1} + ... + q^m = [m+1]_q$, so the coefficient of $x_{m+1-h} y_{n+h}$ computed above matches the desired formula, completing our proof.
\end{proof}

\begin{cor}\label{cor:yx_to_y_q=unit}
If $q$ is a primitive $m^\mathrm{th}$ root of unity, then
\[\displaystyle y_n x_m = \dfrac{1}{u^m} \left[ \left\lceil \frac{n}{m} \right\rceil \prod_{k=0}^{m-1} \left( p - \dfrac{1}{q^{n-1+k}} \right) \right]  y_{n+m} + \langle x_{h} y_{n+m-h} \rangle \]
where $\lceil x \rceil$ denotes the ceiling function of $x$, and $\langle x_{h} y_{n+m-h} \rangle$ consists of all terms of the form $x_{h} y_{n+m-h}$ for $1 \le h \le m$.
\end{cor}

\begin{proof}
It suffices to show that ${n-1+m \choose m}_q = \lceil \frac{n}{m} \rceil$. By definition, we have
\[ {n-1+m \choose m}_q = \frac{(1-q^{n+m-1})(1-q^{n+m-2})...(1-q^n)}{(1-q)(1-q^2)(1-q^m)}. \]
Among $m$ consecutive numbers in $\llbracket n, n+m-1 \rrbracket$, there exists exactly one number of the form $N = ma$ for some $a \ge 1$, for which we get
\[ \frac{1-q^{ma}}{1-q^m} = \frac{(1-q^m)(1+ q^m + ... + q^{m(a-1)})}{1-q^m} = a \]
since $q^m = 1$. Each of the other terms in the numerator has the form $1-q^{mb+r} = 1-q^r \ne 0$ for a unique $1 \le r \le m-1$, which cancels out with the corresponding $1-q^r$ term in the denominator; here $b$ is either $a$ or $a-1$. It follows that ${n-1+m \choose m}_q = a$. Since $N$ is the smallest multiple of $m$ that is greater than or equal to $n$, we deduce that $a = \lceil \frac{n}{m} \rceil$.
\end{proof}

We will now proceed to compute the homology of the $\mathfrak{A}$-bimodule $\mathfrak{M}$ using the description above. Including the sign twist $\epsilon$ into the braidings $\sigma$ and $\varphi$ is the same as replacing $q$ and $u$ with $-q$ and $-u$, respectively. Let $\mathfrak{A}(V^*_\epsilon) \cong \Gamma_{-q}[x] =: \Gamma$ and $\mathfrak{M} = \mathfrak{M}(V^*_\epsilon, W^*)$. It follows from Proposition~\ref{prop:M_fr_left} that
\[ \mathrm{Tor}^{\Gamma^e}_{*,*} (\mathfrak{M},k) = \mathfrak{M} \underset{\Gamma \otimes \Gamma^{op}}{\overset{L}{\otimes}} k = (\Gamma \otimes k\mathcal{Y}) \underset{\Gamma \otimes \Gamma^{op}}{\overset{L}{\otimes}} k = k\mathcal{Y} \underset{\Gamma^{op}}{\overset{L}{\otimes}} k\]
where $\mathcal{Y} = \{ y_1, y_2, ... \}$ is the chosen basis for $\mathfrak{M}$ as a free left $\Gamma$-module. We will continue this computation in the following cases:

\textit{Case 1: $-q$ is not a root of unity in $k$.} Recall from Proposition~\ref{prop:q_div_pow_alg} that $\Gamma \cong k[x_1]$, so it suffices to only look at the right multiplication of the generators $y_n$ by $x_1$. Since all $x_m y_n$ terms vanish in $k\mathcal{Y}$ for $m \ge 1$, Lemma~\ref{lem:yx_to_xy} gives
    \[ y_n x_1 = \dfrac{1}{-u}{n \choose 1}_{-q} \left( p - \frac{1}{(-q)^{n-1}} \right) y_{n+1} = \frac{1-(-q)^n}{-u(1+q)} \left( p - \frac{1}{(-q)^{n-1}} \right) y_{n+1}. \]
    If $p$ is not a power of $-q^{-1}$ (including $p = (-q^{-1})^0 = 1$), $k\mathcal{Y}$ is freely generated as a $\Gamma^{op}$-module by a single generator $y_1$, and hence
    \[ \mathrm{Tor}^{\Gamma^e}_{*,*} (\mathfrak{M},k) = (k\{y_1\} \otimes \Gamma^{op})  \underset{\Gamma^{op}}{\overset{L}{\otimes}} k = k\{y_1\} \overset{L}{\otimes} k = k\{y_1\} = \Sigma^1 k\]
    where $\Sigma^i M$ denotes the shift by $i$ internal degrees of a graded module $M$, i.e.
    \[\mathrm{Tor}^{\Gamma^e}_{j,n} (\mathfrak{M},k) = \begin{cases}
    k \quad \text{ for } j=0, n=1\\
    0 \quad \text{ else.}
    \end{cases}\]
    
    If $p = (-q)^{-(r-1)}$ for some $r \ge 1$, $y_r x_1 = 0$ whereas $y_n x_1 \ne 0$ for all $n \ne r$. It follows that the elements $y_n$ for $1 \le n \le r$ can be generated from $y_1$ by multiplying with $x_1^{n-1}$, but $y_{r+1}$ cannot. All elements $y_{\ge r+1}$ can be freely generated from $y_{r+1}$ by multiplying with powers of $x_1$. Thus as a right $\Gamma$-module, $k\mathcal{Y} \cong k\{y_1\}[x_1]/x_1^r \oplus k\{y_{r+1}\}[x_1] \cong \Sigma^1 k[x_1]/x_1^r \oplus \Sigma^{r+1} k[x_1]$. Since $\Gamma^{op} \cong \Gamma$, we have
    \[\mathrm{Tor}^{\Gamma^e}_{*,*} (\mathfrak{M},k) = \left(\Sigma^1 \Gamma/x_1^r \oplus \Sigma^{r+1} \Gamma \right) \underset{\Gamma}{\overset{L}{\otimes}} k = \Sigma^1 \left( \Gamma/x_1^r \underset{\Gamma}{\overset{L}{\otimes}} k \right) \oplus \Sigma^{r+1} k. \]
    To compute the first summand, we use the graded free resolution of $k$ as a $k[x_1]$-module
    \[0 \to \Sigma^1 k[x_1] \xrightarrow{x_1} k[x_1] \xrightarrow{\epsilon} k \to 0\]
    where $\epsilon$ is the augmentation map. Applying $- \otimes_{k[x_1]} \Gamma/x_1^r$, we get
    \[ 0 \to \Sigma^1 k[x_1]/x_1^r \xrightarrow{x_1} k[x_1]/x_1^r \to 0.\]
    Multiplication by $x_1$ has image $(x_1)$ and kernel $k\{x_1^{r-1}\}$, so
    \[ \mathrm{Tor}_{j,*}^\Gamma (\Gamma/x_1^r, k) = \begin{cases}
        k \hfill \text{for } j = 0\\
        \Sigma^1 k\{x_1^{r-1}\} \cong \Sigma^r k \quad \text{for } j = 1\\
        0 \hspace{1.15in} \text{else,}
    \end{cases} \]
    or
    \[ \mathrm{Tor}_{j,n}^\Gamma (\Gamma/x_1^r, k) = \begin{cases}
        k \quad \text{for } j = 0, n=0\\
        k \quad \text{for } j = 1, n=r\\
        0 \quad \text{else.}
    \end{cases} \]
    Hence
    \[\mathrm{Tor}^{\Gamma^e}_{j,n} (\mathfrak{M},k) = \begin{cases}
        k \quad \text{for } j = 0, \text{ and } n = 1, r+1\\
        k \quad \text{for } j = 1, \text{ and } n = r+1\\
        0 \quad \text{else.}
    \end{cases} \]
    
\textit{Case 2: $-q$ is a primitive $m^\mathrm{th}$ root of unity in $k$.} By Proposition~\ref{prop:q_div_pow_alg}, $\Gamma = k[x_1]/x_1^{m} \otimes \Gamma[x_m]$, so it suffices to study the right multiplication of the generators $y_n$ by $x_1$ and $x_m$. Let $\Lambda_{m} := k[x_1]/x_1^m$ denote the degree-$m$ truncated polynomial algebra in variable $x_1$. If $k$ has characteristic $0$, there is an isomorphism $\Gamma[x_m] \cong k[x_m]$. Consider the multiplication by $x_m$. In this case, Corollary~\ref{cor:yx_to_y_q=unit} gives
    \[\displaystyle y_n x_m = \dfrac{1}{(-u)^m}\left[ \left\lceil \frac{n}{m} \right\rceil \prod_{k=0}^{m-1} \left( p - \dfrac{1}{(-q)^{n-1+k}} \right) \right]  y_{n+m}.\]
    Since the power of $-q$ in the product cycles through $m$ consecutive values, we see that $y_n x_m = 0$ if and only if $p$ is a power of $-q$.
    
    If $p$ is not a power of $-q$, observe that
    \[ y_n x_1 = \frac{1-(-q)^n}{-u(1+q)} \left( p - \frac{1}{(-q)^{n-1}} \right) y_{n+1} \]
    vanishes precisely when $m$ is a divisor of $n$. Observe that for every $n \ge 1$, if we write $n = ma + r$ where $1 \le r \le m$, then the element $y_n$ of $\mathcal{Y}$ can be generated uniquely as $y_n = C y_1 x_m^a x_1^{r-1}$ for some nonzero constant $C$. Thus we may identify
    \[ k\mathcal{Y} \cong k\{y_1\} (k[x_1]/x_1^m \otimes k[x_m]) \cong \Sigma^1 (\Lambda_m \otimes k[x_m]) = \Sigma^1 \Gamma, \]
    so
    \[ \mathrm{Tor}^{\Gamma^e}_{*,*} (\mathfrak{M},k) = \Sigma^1 \Gamma  \underset{\Gamma}{\overset{L}{\otimes}} k = \Sigma^1 k.\]
    
    On the other hand, if $p = (-q)^r$ for some $1 \le r \le m$, $y_n x_m = 0$ for all $n$, so $k\mathcal{Y}$ is trivial as a $k[x_m]$-module. In this case,
    \[ y_n x_1 = \frac{1-(-q)^n}{-u(1+q)} \left( p - \frac{1}{(-q)^{n-1}} \right) y_{n+1} \]
    vanishes iff $m$ divides $n$ or $n+r-1$. If $r = 1$, i.e. $p = -q$, these two conditions coincide. It follows that $k\mathcal{Y}$ is freely generated by $\{ y_1, y_{m+1}, ...\}$ as a $\Lambda_m$-module, i.e.
    \[ k\mathcal{Y} \cong k\{y_1, y_{m+1}, ...\} [x_1]/x_1^m \cong \bigoplus_{a=0}^\infty \Lambda_m \{y_{ma+1}\} \cong \bigoplus_{a=0}^\infty \Sigma^{ma+1} \Lambda_m. \]
    We then have
    \[\begin{array}{r l}
         \mathrm{Tor}^{\Gamma^e}_{*,*} (\mathfrak{M},k) & = \left( \displaystyle \bigoplus_{a=0}^\infty \Sigma^{ma+1} \Lambda_m \right) \underset{\Lambda_m \otimes k[x_m]}{\overset{L}{\otimes}} k\\[15pt]
         & = \displaystyle \bigoplus_{a=0}^\infty \Sigma^{ma+1} k \underset{k[x_m]}{\overset{L}{\otimes}} k = \displaystyle \bigoplus_{a=0}^\infty \Sigma^{ma+1} \Lambda[z_m]
    \end{array} \]
    for some $z_m \in \mathrm{Tor}_{1,m}$, i.e.
    \[\mathrm{Tor}^{\Gamma^e}_{j,n} (\mathfrak{M},k) = \begin{cases}
        k \quad \text{for } j = 0, \text{ and } n = ma + 1 \text{ for } a \ge 0\\
        k \quad \text{for } j = 1, \text{ and } n = ma + 1 \text{ for } a \ge 1\\
        0 \quad \text{else.}
    \end{cases} \]
    
    Finally, if $r > 1$, we then have $y_n x_1 = 0$ whenever $n = ma$ or $n = ma-r+1$. Following the same analysis above, we see that as a $\Lambda_m$-module,
    \[\begin{array}{r l}
    k\mathcal{Y} & \cong k\{y_1\}[x_1]/x_1^{m-r+1} \oplus k\{y_{m-r+2}\}[x_1]/x_1^{r-1} \oplus k\{y_{m+1}\}[x_1]/x_1^{m-r+1} \oplus ... \\[10pt]
    & = \left( \displaystyle \bigoplus_{a=0}^\infty \Lambda_{m-r+1}\{y_{ma+1}\} \right) \oplus \left( \displaystyle \bigoplus_{a=1}^\infty \Lambda_{r-1}\{y_{ma-r+2}\} \right) \\[15pt]
    & \cong \left( \displaystyle \bigoplus_{a=0}^\infty \Sigma^{ma+1} \Lambda_{m-r+1} \right) \oplus \left( \displaystyle \bigoplus_{a=1}^\infty \Sigma^{ma-r+2} \Lambda_{r-1} \right).
    \end{array} \]
    Since $k\mathcal{Y}$ is trivial as a $k[x_m]$-module, by using the universal property of the tensor product (see, e.g., Prop. 13.104.1 in the ancillary file of \cite{gr20}), we may write
    \[ \begin{array}{r l}
    k\mathcal{Y} \underset{\Lambda_m \otimes k[x_m]}{\overset{L}{\otimes}} k & = ( k\mathcal{Y} \otimes k ) \underset{\Lambda_m \otimes k[x_m]}{\overset{L}{\otimes}} (k \otimes k)\\[15pt] 
    & = \left( k\mathcal{Y} \underset{\Lambda_m}{\overset{L}{\otimes}} k \right) \otimes \left( k \underset{k[x_m]}{\overset{L}{\otimes}} k \right) = \left( k\mathcal{Y} \underset{\Lambda_m}{\overset{L}{\otimes}} k \right) \otimes \Lambda[z_m].
    \end{array}\]
    It follows that
    \[\begin{array}{l}
    \mathrm{Tor}^{\Gamma^e}_{*,*} (\mathfrak{M},k) \cong \left[ \left( \left( \displaystyle \bigoplus_{a=0}^\infty \Sigma^{ma+1} \Lambda_{m-r+1} \right) \oplus \left( \displaystyle \bigoplus_{a=1}^\infty \Sigma^{ma-r+2} \Lambda_{r-1} \right) \right) \underset{\Lambda_m}{\overset{L}{\otimes}} k \right] \otimes \Lambda[z_m]\\[15pt]
    \cong \left[ \left( \displaystyle \bigoplus_{a=0}^\infty \Sigma^{ma+1} \left( \Lambda_{m-r+1}\underset{\Lambda_m}{\overset{L}{\otimes}} k \right) \right) \oplus \left( \displaystyle \bigoplus_{a=1}^\infty \Sigma^{ma-r+2} \left( \Lambda_{r-1} \underset{\Lambda_m}{\overset{L}{\otimes}} k \right) \right) \right] \otimes \Lambda[z_m].
    \end{array} \]
    It is left to compute $\Lambda_s \underset{\Lambda_m}{\overset{L}{\otimes}} k$ for $1 \le s \le m-1$. We shall use the resolution
    \[ ...  \xrightarrow{x^s} \Sigma^{2m} \Lambda_m \xrightarrow{x^{m-s}} \Sigma^{m+s} \Lambda_m \xrightarrow{x^s} \Sigma^m \Lambda_m \xrightarrow{x^{m-s}} \Sigma^s \Lambda_m \xrightarrow{x^s} \Lambda_m \xrightarrow{\epsilon} \Lambda_s \xrightarrow{} 0\]
    where the augmentation map $\epsilon$ is the quotient map. Applying $- \underset{\Lambda_m}{\overset{L}{\otimes}} k$, we get
    \[ ... \xrightarrow{0} \Sigma^{2m} k \xrightarrow{0} \Sigma^{m+s} k \xrightarrow{0} \Sigma^m k \xrightarrow{0} \Sigma^s k \xrightarrow{0} k \xrightarrow{} 0.\] 
    Hence
    \[\mathrm{Tor}^{\Lambda_m}_{j,n} (\Lambda_s,k) = \begin{cases}
        k \quad \text{for } j = 2a \text{ and } n = ma, \text{ for } a \ge 0\\
        k \quad \text{for } j = 2a+1 \text{ and } n = ma+s, \text{ for } a \ge 0\\
        0 \quad \text{else.}
    \end{cases} \]

By applying Corollary~\ref{cor:main} and the duality of Tor and Ext functors to the computation of $\mathrm{Tor}_{*,*}^{\Gamma^e}(\mathfrak{M},k)$ above, we obtain a characterization of the homology of type-\textit{B} Artin groups with coefficients in one-dimensional braid representations.

\begin{thm}\label{thm:1dim}
Let $(V,W,\sigma,\tau)$ be the left-braided vector space where $V = W = k$, and the braidings $\sigma$ and $\tau$ are given by multiplications by units $q$ and $p$, respectively. The homology of the type-\textit{B} Artin group $B_n$ with coefficients in the braid representation $V^{\otimes n} \otimes W$ is given in the following cases:
\begin{enumerate}
    \item If $-q$ is not a root of unity in $k$, and
    \begin{enumerate}
        \item $p$ is not a power of $-q^{-1}$:
        \[H_j (B_n; V^{\otimes n} \otimes W) = \begin{cases}
        k \quad \text{ for } n=0, j=0\\
        0 \quad \text{ else;}
        \end{cases}\]
        \item $p = (-q)^{-(r-1)}$ for some $r \ge 1$:
        \[H_j (B_n; V^{\otimes n} \otimes W) = \begin{cases}
        k \quad \text{ for } n=0, j=0\\
        k \quad \text{ for } n=r, \text{ and } j=r-1, r\\
        0 \quad \text{ else.}
        \end{cases}\]
    \end{enumerate}
    \item If $-q$ is a primitive $m^\mathrm{th}$ root of unity in $k$, and
    \begin{enumerate}
        \item $p$ is not a power of $-q$:
        \[H_j (B_n; V^{\otimes n} \otimes W) = \begin{cases}
        k \quad \text{ for } n=0, j=0\\
        0 \quad \text{ else;}
        \end{cases}\]
        \item $p = -q$:
        \[H_j (B_n; V^{\otimes n} \otimes W) = \begin{cases}
        k \quad \text{ for } n =0, j=0\\
        k \quad \text{ for } n=ma \text{ } (a \ge 1), \text{ and } j=n-1,n\\
        0 \quad \text{ else;}
        \end{cases}\]
        \item $p = (-q)^r$ for $2 \le r \le m$: for $n = mk$ $(k \ge 0)$
        \[H_j (B_n; V^{\otimes n} \otimes W) = \begin{cases}
        k \hspace{.4in} \text{ for } j= n-2k,n\\
        k \oplus k \quad \text{ for } n-2k+1 \le j \le n-1\\
        0 \hspace{.4in} \text{ else,}
        \end{cases}\]
        and for $n = mk+m-r+1$ $(k \ge 0)$
        \[H_j (B_n; V^{\otimes n} \otimes W) = \begin{cases}
        k \hspace{.4in} \text{ for } j= n-2k-1,n\\
        k \oplus k \quad \text{ for } n-2k \le j \le n-1\\
        0 \hspace{.4in} \text{ else,}
        \end{cases}\]
        while for all other $n$ it vanishes.
    \end{enumerate}
\end{enumerate}
\end{thm}

Observe that in all cases, the homology of the Artin groups $B_n$ with coefficients in one-dimensional braid representations over $k$ has very large vanishing ranges. Interesting phenomena happen when $-q$ is a primitive $m^\mathrm{th}$ root of unity and $p$ is a power of $-q$. The following is a direct consequence of the previous theorem.

\begin{cor}\label{cor:1dim}
Let $-q$ be a primitive $m^\mathrm{th}$ root of unity. If $p = -q$, $H_j (B_n; V^{\otimes n} \otimes W)$ has a lower vanishing line $j < n-1$. If $p = (-q)^r$ for $2 \le r \le m$, this vanishing line is $j \le \frac{m-2}{m}n -1$.
\end{cor}

There are two specific cases of the computation in this section that are worth mentioning. When $q = 1$ and $p = 1$ (i.e. $m = r =2)$, we recover the homology of the type-\textit{B} Artin groups with trivial coefficients. In this case, $\Lambda_2 = k[x_1]/x_1^2 \cong \Lambda[x_1]$ is the exterior algebra on a generator $x_1$, so we have
\[ k\mathcal{Y} \cong \left( \displaystyle \bigoplus_{a=0}^\infty \Sigma^{2a+1} \Lambda[x_1]/x_1 \right) \oplus \left( \displaystyle \bigoplus_{a=1}^\infty \Sigma^{2a} \Lambda[x_1]/x_1 \right) \cong \displaystyle \bigoplus_{a=1}^\infty \Sigma^{a} k \]
and hence
\[\begin{array}{l}
    \mathrm{Tor}^{\Gamma^e}_{*,*} (\mathfrak{M},k) \cong \displaystyle \bigoplus_{a=1}^\infty \Sigma^a \mathrm{Tor}^\Gamma_{*,*} (k,k) \cong \displaystyle \bigoplus_{a=1}^\infty \Sigma^a k[y_1] \otimes \Lambda[z_2]
\end{array} \]
where $y_1 \in \mathrm{Tor}_{1,1}$ and $z_2 \in \mathrm{Tor}_{1,2}$ \cite{etw17}. Observe that this is an infinite sum of the cohomologies of braid groups, which exhibits the same pattern as Gorjunov's classical computation over $\mathbb{Z}$ \cite{gor78}. More explicitly, Theorem~\ref{thm:1dim} gives
\[H_j (B_n; V^{\otimes n} \otimes W) = \begin{cases}
        k \hspace{.4in} \text{ for } j= 0,n\\
        k \oplus k \quad \text{ for } 1 \le j \le n-1\\
        0 \hspace{.4in} \text{ else.}
        \end{cases}\]

The second case of interest is when $q = 1$ and $p = -1$ (i.e. $m = 2, r = 1$); in this case, the braid representation of the Artin group $B_n$ tracks the parity of the number of generators $\tau_n$ in the decomposition of a braid, i.e. whether the number of times a braid wraps around its pure last strand is even or odd. Theorem~\ref{thm:1dim} gives
\[H_j (B_n; V^{\otimes n} \otimes W) = \begin{cases}
        k \quad \text{ for } n =0, j=0\\
        k \quad \text{ for } n=2a \text{ } (a \ge 1), \text{ and } j=n-1,n\\
        0 \quad \text{ else.}
        \end{cases}\]
In particular, the homology of all odd-numbered Artin groups $B_{2a+1}$ vanishes. When $n = 2a$, only the top two homology groups are nontrivial and have rank $1$; the vanishing line of the homology of the even-numbered Artin groups $B_{2a}$ with these coefficients is $j < n-1$, as deduced in Corollary~\ref{cor:1dim}.


\section{Cellular homology of configuration spaces}\label{sec:fnf}

In this section, we will give a geometric explanation for Corollary~\ref{cor:etw_rewrite}. We will construct a cellular stratification for configuration spaces of the punctured complex plane Conf$_{n}(\mathbb{C}^\times)$, based on the Fox-Neuwirth/Fuks cells of Conf$_n (\mathbb{C})$ introduced in Section~\ref{ssec:fnf}. We will also discuss a geometric motivation for the algebraic structures developed in Section~\ref{ssec:alg_req}.


\subsection{Fox-Neuwirth cellular stratification of Conf$_n (\mathbb{C}^{\times})$}\label{ssec:FNF_Cx}

First, we observe that there is an canonical embedding Conf$_n (\mathbb{C}^\times) \hookrightarrow \mathrm{Conf}_{n+1}(\mathbb{C})$ by inserting the removed origin. This gives a homeomorphic image of Conf$_n(\mathbb{C}^\times)$ as a subspace of Conf$_{n+1} (\mathbb{C})$ consisting of all configurations where the point $O$ at the origin is always fixed. We will give a stratification of this subspace based on the Fox-Neuwirth cellular stratification of Conf$_{n+1} (\mathbb{C})$ developed above. For the rest of this paper, we will indiscriminately use the notation Conf$_n (\mathbb{C}^\times)$ for both the original configuration space of the punctured complex plane and its homeomorphic image in Conf$_{n+1} (\mathbb{C})$.

Given a composition $\lambda$ of $n+1$, we consider the intersection of the cell Conf$_{\lambda}(\mathbb{C})$ and the subspace Conf$_n (\mathbb{C}^\times)$ of Conf$_{n+1}(\mathbb{C})$. Starting with a configuration in Conf$_{\lambda}(\mathbb{C})$, we insist that one of the points must be the fixed point $O$ at the origin. This requirement restricts the configuration in two ways. First, the vertical column that contains $O$ must coincide with the imaginary axis, i.e. the real part of all points on that column must be $0$. This column plays a special role in our stratification and will be recorded by the index $i$. Secondly, as we let the points in a configuration move along the vertical line without colliding in a cell, fixing $O$ implies that it is impossible for points on the imaginary axis to move past it. The number of points on this vertical line with a negative imaginary part is hence fixed and denoted by the index $j$. Therefore, the connected components in the above intersection can be denoted by $e_{(\lambda, i, j)} = \text{Conf}_{(\lambda, i, j)} (\mathbb{C})$ where $\lambda$ is a composition of $n+1$, $i$ is the index of the vertical column that contains $O$, i.e. the imaginary axis ($1 \le i \le l(\lambda)$), and $j$ is the number of points lying below $O$ on the imaginary axis ($0 \le j \le \lambda_i - 1$) (see Figure~\ref{fig:FNF_Cx}). The total order of a configuration in the embedded image of $\mathrm{Conf}_n(\mathbb{C}^\times)$ is inherited from the parent space $\mathrm{Conf}_{n+1}(\mathbb{C})$; in particular, the overall position of $O$ in a configuration in $e_{(\lambda,i,j)}$ is $\iota = j+1+\sum_{m=1}^{i-1} \lambda_m$.

\begin{figure}[t]
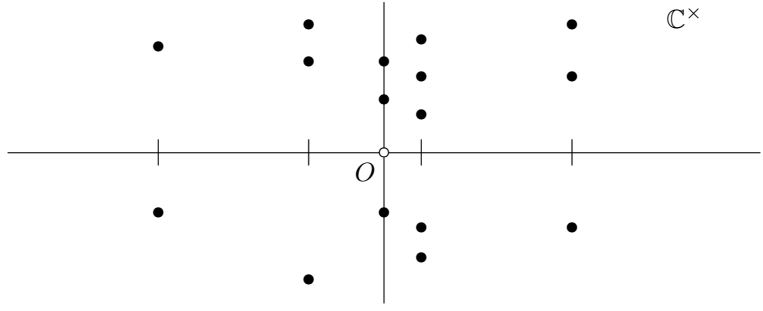

\begin{centering}
\[ \vcenter{	
	\xy
		(-30,14)*{\bullet}; (-30,-8)*{\bullet}; (-30,0)*{|};
		(-10,12)*{\bullet}; (-10,-17)*{\bullet}; (-10,17)*{\bullet}; (-10,0)*+{|};
		(0,7)*{\bullet}; (0,12)*{\bullet}; (0,-8)*{\bullet}; (0,-0.08)*{\circ}; (-2.5,-2.5)*{O};
		(5,10)*{\bullet}; (5,-10)*{\bullet}; (5,15)*{\bullet}; (5,-14)*{\bullet}; (5,5)*{\bullet}; (5,0)*+{|};
		(25,10)*{\bullet}; (25,-10)*{\bullet}; (25,17)*{\bullet}; (25,0)*+{|};
		{\ar@{-} (-50,0)*{}; (-0.65,0)*{}}; {\ar@{-} (0.65,0)*{}; (50,0)*{}};
		{\ar@{-} (0,-20)*{}; (0,-0.65)*{}}; {\ar@{-} (0,0.65)*{}; (0,20)*{}};
		(40,18)*{ \mathbb{C}^\times};
	\endxy
	}
\]
	\caption{A configuration in $e_{((2,3,4,5,3),3,1)} \subset$ Conf$_{16} (\mathbb{C}^\times)$. The removed origin $O$ (mapped to a fixed point in the embedded image of Conf$_{16} (\mathbb{C}^\times)$ in Conf$_{17} (\mathbb{C})$) lies on the third vertical column from the left with one point below.}
	\label{fig:FNF_Cx}
\end{centering}
\end{figure}

The spaces $e_{(\lambda,i,j)}$ then provide the positive dimension cells for our cellular decomposition of Conf$_n(\mathbb{C}^\times) \cup \{ \infty \}$. Each cell $e_{(\lambda,i,j)}$ is homeomorphic to the product
\[\Big[ \mathrm{Conf}_{i-1}(\mathbb{R}) \times \mathrm{Conf}_{l(\lambda)-i}(\mathbb{R}) \Big] \times \displaystyle \prod^{l(\lambda)}_{k=1,k\ne i} \mathrm{Conf}_{\lambda_k} (\mathbb{R}) \times \Big[ \mathrm{Conf}_j (\mathbb{R}) \times \mathrm{Conf}_{\lambda_i - j-1} (\mathbb{R}) \Big].\]
The first bracket represents the configurations of the vertical columns on the left and right of the imaginary axis, i.e. recording the real parts of the points. The latter bracket keeps track of the imaginary parts of points below and above $O$ on the imaginary axis, while the middle product records the same information for those on all other vertical lines. By applying the isomorphism Conf$_k(\mathbb{R}) \cong \mathbb{R}^k$, we see that the cell $e_{(\lambda,i,j)}$ has dimension $n+l(\lambda)-1$; loosely speaking, compared to the classical Fox-Neuwirth cells indexed by the same composition $\lambda$, we lost two dimensions due to fixing the real and imaginary parts of the point $O$.

As in the Fox-Neuwirth cellular decomposition of Conf$_n(\mathbb{C})$, the boundary of a cell is obtained in two ways. For the first type, besides letting points in a configuration approach another or infinity, we also allow moving points towards the punctured origin; in this case, the boundary is still the point at infinity for the same reason. The second type of boundary again occurs by horizontally joining two adjacent vertical columns of the configuration without colliding the points. However, due to the second restriction on a configuration in $e_{(\lambda,i,j)}$, namely points below the fixed point $O$ cannot move across it on the imaginary axis, the boundary cells obtained this way have four general forms, depending on the positions of the columns relative to the imaginary axis and on whether this axis itself is among those combined. In particular, when combining an adjacent vertical column with the imaginary axis, we must keep track of the number of points going below $O$, i.e. adding to the index $j$. In summary:

\begin{prop}
The space $\mathrm{Conf}_n(\mathbb{C}^\times) \cup \{ \infty \}$ may be presented as a CW complex whose positive dimension cells $e_{(\lambda, i, j)} = \mathrm{Conf}_{(\lambda, i, j)}(\mathbb{C})$ (of dimension $n+l(\lambda)-1$) are indexed by triples $(\lambda,i,j)$, where $\lambda$ is an ordered partition of $n+1$, $i$ is the index of the imaginary axis in the configuration ($1 \le i \le l(\lambda)$), i.e. there are $i-1$ vertical columns to the left of the imaginary axis, and $j$ is the number of points with zero real parts and negative imaginary parts ($0 \le j \le \lambda_i - 1$).

Let $\rho^m = (\lambda_1, ..., \lambda_m + \lambda_{m+1}, ..., \lambda_{l(\lambda)})$ be the coarsening of $\lambda$ obtained by summing $\lambda_m$ and $\lambda_{m+1}$ $(1 \le m < l(\lambda))$. The codimension-$1$ boundary cells of $e_{(\lambda, i, j)}$ have four general forms:
\begin{enumerate}
\item $e_{(\rho^m, i-1, j)}$ \hfill $1 \le m < i-1$,
\item $e_{(\rho^m, i, j})$ \hfill $i < m < l(\lambda)$,
\item $e_{(\rho^{i-1}, i-1, j+h)}$ \hfill $0 \le h \le \lambda_{i-1}$,
\item $e_{(\rho^i, i, j + h)}$ \hfill $0 \le h \le \lambda_{i+1}$,
\end{enumerate}
where $h$ denotes the number of points going below the origin $O$ when combining the imaginary axis with the column on the left (3) or right (4).
\end{prop}

From this we can write down an explicit cellular chain complex for Conf$_n (\mathbb{C}^\times) \cup \{ \infty \}$. First, we will develop some combinatorial concepts needed to state the definition of this complex.

Recall that a $(p,q)$-shuffle $\gamma : \{1,...,p\} \sqcup \{1,...,q\} \to \{1,...,p+q\}$ is a bijection that preserves orders on both $\{1,...,p\}$ and $\{1,...,q\}$. Alternatively, a $(p,q)$-shuffle can be treated as a permutation in $S_{p+q}$ that preserves orders on the first $p$ and the last $q$ elements. In the discussion below, we will primarily refer to them by the latter definition.

\begin{defn}\label{defn:pqhj_shuf}
For $0 \le h \le q$ and $0 \le j \le p-1$, a $(p,(q,h),j)$\textit{-shuffle} is defined to be a $(p,q)$-shuffle that sends $j+1$ to $j+h+1$. Similarly, for $0 \le h \le p$ and $0 \le j \le q-1$, a $((p,h),q,j)$\textit{-shuffle} is a $(p,q)$-shuffle that sends $p+j+1$ to $h+j+1$.
\end{defn}

For convenience, we will refer to the designated elements in the above definition as the \textit{marked elements}; they will later correspond to the fixed point $O$ in a configuration in Conf$_n(\mathbb{C}^\times)$. The naming conventions of these shuffles are geometrically motivated: for example, the $(p,(q,h),j)$-shuffles arise when we combine the imaginary axis with $p$ points in total and $j$ points below the fixed point $O$ from the right with a column containing $q$ points, sending $h$ out of $q$ points below $O$ in the process. Similarly, the $((p,h),q,j)$-shuffles occur when the imaginary axis with $q$ points in total and $j$ points below $O$ is joined from the left by a column with $p$ points, sending $h$ out of $p$ points below $O$.

Consider a $(p,(q,h),j)$-shuffle $\gamma$. Since $\gamma$ as a $(p,q)$-shuffle preserves orders on the first $p$ and the last $q$ elements, and $\gamma$ sends the marked element from $j+1$ to $j+h+1$, it must follows that
\[ \gamma(p+1) < \gamma(p+2) < ... < \gamma(p+h) < j+h+1;\]
loosely speaking, the elements in the integer interval $\llbracket p+1, p+h \rrbracket$ must ``move left'' on the number line to fill in the $h$ holes left behind by the move of the marked element. This observation leads to a useful decomposition of the $(p,(q,h),j)$-shuffles.

\begin{prop}\label{prop:pqhj_decomp}
There is a unique decomposition of a $(p,(q,h),j)$-shuffle into a sequence of three permutations: a fixed $(p,q)$-shuffle that maps $\llbracket p+1, p+h \rrbracket$ onto $\llbracket j+1, j+h\rrbracket$, followed by a $(j,h)$-shuffle on $\llbracket 1, j+h \rrbracket$ and a $(p-j-1,q-h)$-shuffle on the last $\llbracket j+h+2,p+q \rrbracket$ (see Figure~\ref{fig:pqhj_decomp}). As a consequence, there is a bijection
\[\mathrm{Sh}(p,(q,h),j) \cong \mathrm{Sh}(j,h) \times \mathrm{Sh}(p-j-1,q-h)\]
where $\mathrm{Sh}(p,(q,h),j)$, $\mathrm{Sh}(j,h)$ and $\mathrm{Sh}(p-j-1,q-h)$ are the sets of $(p,(q,h),j)$-, $(j,h)$- and $(p-j-1,q-h)$-shuffles, respectively.
\end{prop}

\begin{figure}[t]
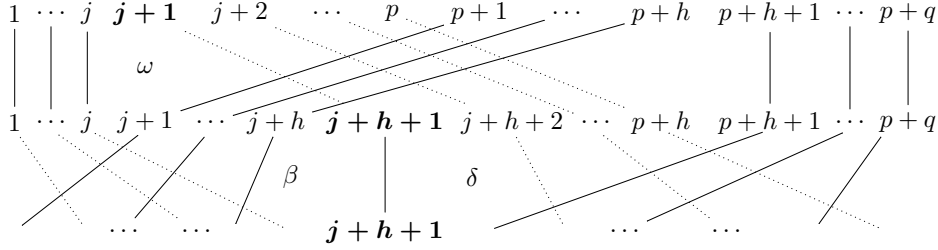

    \centering
    \[ 
\resizebox{\textwidth}{!}{
	\xy
		{\ar@{-} (0,15)*+{1}; (0,30)*+{1} };
		{\ar@{-} (5,15)*+{\cdots}; (5,30)*+{\cdots} }; 
		{\ar@{-} (10,15)*+{j}; (10,30)*+{j} }; 
		{\ar@{-} (18,15)*+{j+1}; (64,30)*+{p+1} }; 
		{\ar@{..} (18,30)*+{\boldsymbol{j+1}}; (51,15)*+{\boldsymbol{j+h+1}} };
		{\ar@{-} (27,15)*+{\cdots}; (76,30)*+{\cdots} };
		{\ar@{-} (36,15)*+{j+h}; (89,30)*+{p+h} };
		{\ar@{..} (31,30)*+{j+2}; (68.5,15)*+{j+h+2} };
		{\ar@{..} (43,30)*+{\cdots}; (80,15)*+{\cdots} };
		{\ar@{..} (52,30)*+{p}; (89,15)*+{p+h} };
		{\ar@{-} (104,15)*+{p+h+1}; (104,30)*+{p+h+1} };
		{\ar@{-} (115,15)*+{\cdots}; (115,30)*+{\cdots} };
		{\ar@{-} (123,15)*+{p+q}; (123,30)*+{p+q} };
		{\ar@{-} (51,14)*+{}; (51,0)*+{\boldsymbol{j+h+1}} };
		{\ar@{..} (0,14)*+{}; (10,0)*+{} };
		{\ar@{..} (5,14)*+{}; (25,0)*+{\cdots} }; 
		{\ar@{..} (10,14)*+{}; (38,0)*+{} };
		{\ar@{-} (18,14)*+{}; (0,0)*+{} };
		{\ar@{-} (36,14)*+{}; (30,0)*+{} };
		{\ar@{-} (27,14)*+{}; (15,0)*+{\cdots} };
		{\ar@{..} (68.5,14)*+{}; (76,0)*+{} };
		{\ar@{..} (89,14)*+{}; (120,0)*+{} };
		{\ar@{..} (80,14)*+{}; (98,0)*+{\cdots} };
		{\ar@{-} (104,14)*+{}; (65,0)*+{} };
		{\ar@{-} (115,14)*+{}; (84,0)*+{\cdots} };
		{\ar@{-} (120,14)*+{}; (110,0)*+{} };
		{(18,22.5)*+{\omega}};
		{(38,7.5)*+{\beta}};
		{(63,7.5)*+{\delta}};
	\endxy
	}
\]
  \caption{Decomposition of a $(p,(q,h),j)$-shuffle $\gamma$ into a sequence of three permutations: the fixed permutation $\omega$, followed by a $(j,h)$-shuffle $\beta$ on the integer interval $\llbracket 1,j+h \rrbracket$ and a $(p-j-1,q-h)$-shuffle $\delta$ on $\llbracket j+h+2,p+q \rrbracket$. The marked element (in bold) is mapped from $j+1$ to $j+h+1$ as stated in Definition~\ref{defn:pqhj_shuf}. The lift $\tilde{\gamma}$ is equivalent to the product of lifts $\tilde{\delta}\tilde{\beta}\tilde{\omega}$.}
  \label{fig:pqhj_decomp}
\end{figure}

\begin{proof}
Let $\gamma$ be a $(p,(q,h),j)$-shuffle. Explicitly, $\omega$ is the permutation of $p+q$ elements defined by
\[\omega(m) = \begin{cases} m \hspace{2in} 1 \le m \le j \\
m+h \hfill j+1 \le m \le p \\
m-p+j \hfill p+1 \le m \le p+h \\
m \hfill p+h+1 \le m \le p+q. \end{cases} \]
Observe that $\omega$ maps the integer interval $\llbracket p+1, p+h \rrbracket$ to $\llbracket j+1,j+h \rrbracket$, thus sending the first $h$ elements in the second set (while maintaining their order) left past the marked element now at $j+h+1$.

Consider $\gamma$ as a permutation of $p+q$ elements. Since $\gamma$ preserves order on the first $p$ elements, $\gamma(m) \ge \gamma(j+1) = j+h+1$ for all $j+1 \le m \le p$. On the other hand, $\gamma(p+h)$ must be smaller than $j+h+1$, since otherwise $\gamma(m) \ge \gamma(p+h) \ge j+h+1$ for all $p+h \le m \le p+q$ and hence $\gamma$ maps at least $p+q-j-h+1$ elements bijectively onto $\llbracket j+h+1, p+q \rrbracket$, a contradiction. Since $\gamma$ also preserves order on the last $q$ elements, $\gamma(m) \le \gamma(p+h) < j+h+1$ for all $p+1 \le m \le p+h$. It follows that $\gamma(\llbracket 1, j\rrbracket \sqcup \llbracket p+1, p+h\rrbracket) = \llbracket 1, j+h\rrbracket$ and $\gamma(\llbracket j+1, p\rrbracket \sqcup \llbracket p+h+1, p+q\rrbracket) = \llbracket j+h+1, p+q\rrbracket$; moreover, $\gamma$ is order-preserving on each of the four component intervals in these two disjoint unions. Observe that $\omega^{-1}$ restricts to order-preserving bijections:
\[\begin{array}{c} \llbracket 1, j\rrbracket \xrightarrow{\mathrm{id}} \llbracket 1, j\rrbracket, \\[5pt]
\llbracket j+1, j+h\rrbracket \xrightarrow{+(p-j)} \llbracket p+1, p+h\rrbracket, \\[5pt]
\llbracket j+h+1, p+h\rrbracket \xrightarrow{-h} \llbracket j+1, p\rrbracket, \text{ and}\\[5pt]
\llbracket p+h+1, p+q\rrbracket \xrightarrow{\mathrm{id}} \llbracket p+h+1, p+q\rrbracket.\end{array}\]
Hence $\gamma \omega^{-1}$ maps each of the intervals $\llbracket 1, j+h \rrbracket$ and $\llbracket j+h+1, p+q \rrbracket$ bijectively onto itself, while preserving orders on the subintervals $\llbracket 1, j \rrbracket$, $\llbracket j+1, j+h \rrbracket$, $\llbracket j+h+1, p+h \rrbracket$, and $\llbracket p+h+1, p+q \rrbracket$. Since $\gamma \omega^{-1}(j+h+1) = j+h+1$, the restrictions of $\gamma \omega^{-1}$ on $\llbracket 1, j+h \rrbracket$ and $\llbracket j+h+2, p+q \rrbracket$ then form a $(j,h)$-shuffle and a $(p-j-1,q-h)$-shuffle, respectively. Thus we may write $\gamma \omega^{-1} = \delta \beta$, where $\beta$ nontrivially only on $\llbracket 1, j+h \rrbracket$ by a $(j,h)$-shuffle, and $\delta$ acts nontrivially only on $\llbracket j+h+2, p+q\rrbracket$ by a $(p-j-1,q-h)$-shuffle. This gives a decomposition $\gamma = \delta \beta \omega$ as desired. Since $\omega$ is fixed, this decomposition is unique if so is the decomposition of $\gamma \omega^{-1}$ into the shuffles $\beta$ and $\delta$ for any $(p,(q,h),j)$-shuffle $\gamma$, which is evident.

The existence and uniqueness of the decomposition of $(p,(q,h),j)$-shuffles above give an injection
\[\mathrm{Sh}(p,(q,h),j) \hookrightarrow \mathrm{Sh}(j,h) \times \mathrm{Sh}(p-j-1,q-h).\]
On the other hand, given a $(j,h)$-shuffle $\beta$ and a $(p-j-1,q-h)$-shuffle $\delta$, from the previous argument it is easy to see that $\gamma = \delta' \beta' \omega$ is a $(p,(q,h),j)$-shuffle, where $\beta'$ acts as $\beta$ on $\llbracket 1, j+h\rrbracket$ and $\mathrm{id}$ else, while $\delta'$ acts as $\delta$ on $\llbracket j+h+2,p+q \rrbracket$ and $\mathrm{id}$ else. This shows surjectivity of the set identification and thus completes our proof.
\end{proof}

Observe that given a $(p,(q,h),j)$-shuffle $\gamma = \delta \beta \omega$, for every $m \in \llbracket 1, p \rrbracket$ we have $m \le \omega(m) \le \beta\omega(m) \le \delta\beta\omega(m)$. Thus in the product of lifts $\tilde{\delta} \tilde{\beta} \tilde{\omega}$, the strand originally starting at $m$ is always behind throughout all component lifts, and hence matches the strand starting at $m$ in the lift of $\gamma$. Similar observation shows that the strand starting at each $m \in \llbracket p+1, p+q \rrbracket$ is always in the front for all component lifts. Therefore, the lift of the shuffle $\gamma$ is in fact given by the product of the lifts of the component shuffles, i.e. $\tilde{\gamma} = \tilde{\delta} \tilde{\beta} \tilde{\omega}$.

A similar observation applies to the $((p,h),q,j)$-shuffles: since a $((p,h),q,j)$-shuffle $\gamma$ sends the marked element from $p+j+1$ to $h+j+1$, it must follows that
\[ j+h+1 < \gamma(h+1) < \gamma(h+2) < ... < \gamma(p);\]
namely, the elements in $\llbracket h+1, p \rrbracket$ must ``move right'' on the number line to fill in the $p-h$ holes left behind by the move of the marked element. The same argument as above proves the following decomposition theorem for the $((p,h),q,j)$-shuffles.

\begin{figure}[t]
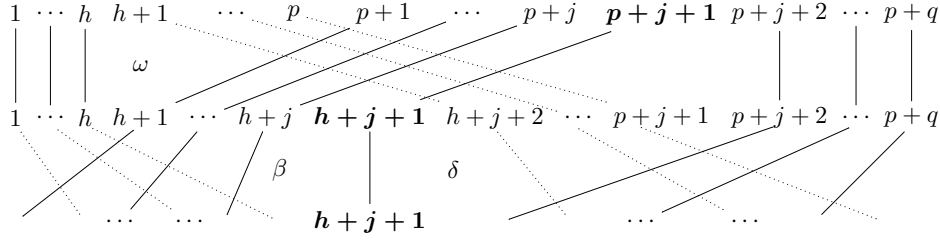

    \centering
    \[ 
\resizebox{\textwidth}{!}{
	\xy
		{\ar@{-} (0,15)*+{1}; (0,30)*+{1} };
		{\ar@{-} (5,15)*+{\cdots}; (5,30)*+{\cdots} }; 
		{\ar@{-} (10,15)*+{h}; (10,30)*+{h} }; 
		{\ar@{-} (18,15)*+{h+1}; (53,30)*+{p+1} }; 
		{\ar@{..} (18,30)*+{h+1}; (69,15)*+{h+j+2} };
		{\ar@{-} (27,15)*+{\cdots}; (65,30)*+{\cdots} };
		{\ar@{-} (36,15)*+{h+j}; (77,30)*+{p+j} };
		{\ar@{-} (93,30)*+{\boldsymbol{p+j+1}}; (51,15)*+{\boldsymbol{h+j+1}} };
		{\ar@{..} (31,30)*+{\cdots}; (81,15)*+{\cdots} };
		{\ar@{..} (40,30)*+{p}; (93,15)*+{p+j+1} };
		{\ar@{-} (110,15)*+{p+j+2}; (110,30)*+{p+j+2} };
		{\ar@{-} (121,15)*+{\cdots}; (121,30)*+{\cdots} };
		{\ar@{-} (129,15)*+{p+q}; (129,30)*+{p+q} };
		{\ar@{-} (51,14)*+{}; (51,0)*+{\boldsymbol{h+j+1}} };
		{\ar@{..} (0,14)*+{}; (10,0)*+{} };
		{\ar@{..} (5,14)*+{}; (25,0)*+{\cdots} }; 
		{\ar@{..} (10,14)*+{}; (38,0)*+{} };
		{\ar@{-} (18,14)*+{}; (0,0)*+{} };
		{\ar@{-} (36,14)*+{}; (30,0)*+{} };
		{\ar@{-} (27,14)*+{}; (15,0)*+{\cdots} };
		{\ar@{..} (68.5,14)*+{}; (80,0)*+{} };
		{\ar@{..} (89,14)*+{}; (125,0)*+{} };
		{\ar@{..} (80,14)*+{}; (105,0)*+{\cdots} };
		{\ar@{-} (110,14)*+{}; (70,0)*+{} };
		{\ar@{-} (121,14)*+{}; (90,0)*+{\cdots} };
		{\ar@{-} (129,14)*+{}; (115,0)*+{} };
		{(18,22.5)*+{\omega}};
		{(38,7.5)*+{\beta}};
		{(63,7.5)*+{\delta}};
	\endxy
	}
\]
    \caption{Decomposition of a $((p,h),q,j)$-shuffle $\gamma$ into a sequence of three permutations: a  fixed $(p,q)$-shuffle $\omega$ that sends $\llbracket h+1, p \rrbracket$ to $\llbracket h+j+2, p+j+1 \rrbracket$, followed by an $(h,j)$-shuffle $\beta$ on $\llbracket 1,h+j \rrbracket$ and a $(p-h,q-j-1)$-shuffle $\delta$ on $\llbracket h+j+2,p+q \rrbracket$. The marked element (in bold) is mapped from $p+j+1$ to $h+j+1$ as stated in Definition~\ref{defn:pqhj_shuf}. The lift $\tilde{\gamma}$ can be observed to be equivalent to the product of lifts $\tilde{\delta}\tilde{\beta}\tilde{\omega}$.}
    \label{fig:phqj_decomp}
\end{figure}

\begin{prop}\label{prop:phqj_decomp}
There is a unique decomposition of a $((p,h),q,j)$-shuffle into a sequence of three permutations: a  fixed $(p,q)$-shuffle that sends $\llbracket h+1, p \rrbracket$ to $\llbracket h+j+2, p+j+1 \rrbracket$, followed by an $(h,j)$-shuffle on $\llbracket 1,h+j \rrbracket$ and a $(p-h,q-j-1)$-shuffle on $\llbracket h+j+2,p+q \rrbracket$ (see Figure~\ref{fig:phqj_decomp}). As a consequence, there is a bijection
\[\mathrm{Sh}((p,h),q,j) \cong \mathrm{Sh}(h,j) \times \mathrm{Sh}(p-h,q-j-1)\]
where $\mathrm{Sh}((p,h),q,j)$ denotes the set of $((p,h),q,j)$-shuffles.
\end{prop}

Recall that $c_{p,q} = \sum_\gamma (-1)^{|\gamma|}$ denotes the sum of the signs of all $(p,q)$-shuffles $\gamma$. Since $(p,(q,h),j)$- and $((p,h),q,j)$-shuffles are essentially elements of the symmetric group, their signs are also well-defined. Let $c_{p,(q,h),j}$ and $c_{(p,h),q,j}$ be the sums of the signs of all $(p,(q,h),j)$- and $((p,h),q,j)$-shuffles, respectively. Based on the decompositions described above, we may express these constants in terms of the constants $c_{i,j}$ corresponding to the component shuffles.

\begin{lem}\label{lem:c_pqhj}
The sums of the signs of all $(p,(q,h),j)$- and $((p,h),q,j)$-shuffles can be computed by the following formulae:
\begin{enumerate}
    \item $c_{p,(q,h),j} = (-1)^{h(p-j)} c_{j,h} c_{p-j-1,q-h}$
    \item $c_{(p,h),q,j} = (-1)^{(j+1)(p-h)} c_{h,j} c_{p-h,q-j-1}$.
\end{enumerate}
\end{lem}

\begin{proof}
Let $\gamma$ be a $(p,(q,h),j)$-shuffle, and suppose we have a decomposition $\gamma = \delta' \beta' \omega$ as described in the proof of Proposition~\ref{prop:pqhj_decomp} (here $\beta'$ and $\delta'$ are respectively the previously defined lifts of the $(j,h)$-shuffle $\beta$ and the $(p-j-1,q-h)$-shuffle $\delta$ to the symmetric group $S_{p+q}$). Observe that the shuffle $\omega$ swaps $p-j$ points in $\llbracket j+1, p \rrbracket$ across $h$ points in $\llbracket p+1, p+h \rrbracket$ (order-preservingly on each segment), so the number of crossings in $\omega$ is $h(p-j)$; hence the sign of $\omega$ is $(-1)^{h(p-j)}$. Meanwhile, the signs of $\beta'$ and $\delta'$ are equal to those of $\beta$ and $\delta$, respectively. It follows that
\[\begin{split}
    c_{p,(q,h),j} & = \sum_\gamma (-1)^{|\gamma|}  = \sum_{\beta',\delta'}(-1)^{|\omega|}(-1)^{|\beta'|}(-1)^{|\delta'|} = \sum_{\beta,\delta}(-1)^{h(p-j)}(-1)^{|\beta|}(-1)^{|\delta|} \\
    & = (-1)^{h(p-j)} \sum_\beta (-1)^{|\beta|} \sum_\delta (-1)^{|\delta|} = (-1)^{h(p-j)} c_{j,h} c_{p-j-1,q-h}
\end{split}\]
thus the first identity holds. A similar argument proves the second claim.
\end{proof}

We now give the definition for the Fox-Neuwirth cellular chain complex for $\mathrm{Conf}_n(\mathbb{C}^\times) \cup \{\infty\}$.

\begin{defn}[The Fox-Neuwirth complex for Conf$_n (\mathbb{C}^\times) \cup \{\infty\}$]
    Let $D(n)_*$ denote the chain complex which in degree $q$ is generated over $\mathbb{Z}$ by the set of triples $(\lambda,i,j)$ where $\lambda = (\lambda_1, ..., \lambda_{q-n+1})$ is a composition of $n+1$ of length $q-n+1$, $1 \le i \le l(\lambda)$ and $0 \le j \le \lambda_i -1$. The differential $d: D(n)_q \to D(n)_{q-1}$ is given by the formula
    \[\begin{array}{l}
    \displaystyle d(\lambda, i, j) = \sum_{m = 1}^{i-2} (-1)^{m-1} c_{\lambda_m, \lambda_{m+1}} (\rho^m, i-1, j) + \sum_{m = i+1}^{q-n} (-1)^{m-1} c_{\lambda_m, \lambda_{m+1}} (\rho^m, i, j)\\[15pt]
    + (-1)^{i-2} \displaystyle\sum_{h=0}^{\lambda_{i-1}} c_{(\lambda_{i-1},h),\lambda_i, j} (\rho^{i-1}, i-1, j+h) + (-1)^{i-1}\displaystyle\sum_{h=0}^{\lambda_{i+1}} c_{\lambda_i, (\lambda_{i+1},h),j} (\rho^i, i, j+h).\end{array}\]
\end{defn}

As in the original Fox-Neuwirth chain complex, the signs in the formula of the differential result from the induced orientations on the boundary strata, following the general scheme described in \cite{gs12}. Notice that the differential is more complicated than that of the classical Fox-Neuwirth complex, which is a consequence of a larger collection of boundary cells. The proof that the chain complex $D(n)_*$ is well-defined is nontrivial and exhibits the usefulness of the combinatorial identities introduced above.

\begin{prop}
$d^2 = 0$.
\end{prop}

\begin{proof}
The general strategy is to enumerate all types of boundary cells in $d^2 (\lambda,i,j)$ and show that their coefficients all vanish. Loosely speaking, cells in $d^2 (\lambda,i,j)$ are formed by subsequently performing two column-combining operations on $e_{(\lambda,i,j)}$. In general, there are two types of results: either (1) two pairs of columns in $e_{(\lambda,i,j)}$ are combined separately, or (2) three adjacent columns are combined into a single column. If none of these columns is the imaginary axis, the coefficient of the boundary cell vanishes in the exact same way as in the classical Fox-Neuwirth complex (one may treat the imaginary axis as a normal column in this case). There are four different subtypes of (1) and three subtypes of (2) that involve the imaginary axis. We will exhibit the argument for a subtype of each case.

The representative boundary cell we choose for type (1) is
\[((\lambda_1, ..., \lambda_{i-1}+\lambda_i, ..., \lambda_m + \lambda_{m+1}, ..., \lambda_{l(\lambda)}),i-1,j+h),\]
obtained by joining two pairs of columns indexed by $\{i-1,i\}$ and $\{m,m+1\}$ $(m \ge i+1)$. There are two orders to perform the operations: either (1a) combining the first pair then the second pair, or (1b) combining the second pair first. Thus the coefficient for the cell above in $d^2(\lambda,i,j)$ is
\[\begin{array}{c}
(-1)^{i-2} c_{(\lambda_{i-1},h),\lambda_{i}, j} \cdot (-1)^{m-2}c_{\lambda_m,\lambda_{m+1}}\\[5pt]
+ (-1)^{m-1}c_{\lambda_m,\lambda_{m+1}} \cdot (-1)^{i-2}  c_{(\lambda_{i-1},h),\lambda_{i}, j} = 0.\end{array}\]
A similar argument shows the same result for the other subtypes of case (1).

For an example of type (2), consider a boundary cell of the form
\[((\lambda_1, ..., \lambda_{i-2}+\lambda_{i-1}+\lambda_i, ..., \lambda_{l(\lambda)}),i-2,j+h),\]
obtained by joining three adjacent columns indexed by $\{i-2,i-1,i\}$. Similarly, there are two orders to perform the operations: either (2a) combining the first two columns then combining the joint column with the third, or (2b) combining the last two columns first. Particularly in case (2b), the imaginary axis involves in both column combinations, so the $h$ points that move below the fixed point $O$ in the final configuration can be split into two steps: $s$ points in the first operation followed by $h-s$ points in the second ($0\le s\le h$). The coefficient of the cell above in $d^2(\lambda,i,j)$ hence contains a sum over all $s$:
\[\begin{array}{c}
(-1)^{i-3} c_{\lambda_{i-2},\lambda_{i-1}} \cdot (-1)^{i-3} c_{(\lambda_{i-2} +\lambda_{i-1},h), \lambda_{i}, j} \\
+\displaystyle \sum_{s=0}^h (-1)^{i-2} c_{(\lambda_{i-1},s),\lambda_{i}, j} \cdot (-1)^{i-3}  c_{(\lambda_{i-2},h-s),\lambda_{i-1} + \lambda_{i}, j+s}.\end{array}\]
We apply Lemma~\ref{lem:c_pqhj} to write this coefficient completely in terms of the constants $c_{p,q}$ and observe that in order for it to vanish, the identity
\[\begin{array}{c}
c_{\lambda_{i-2},\lambda_{i-1}} c_{h,j} c_{\lambda_{i-2}+\lambda_{i-1}-h,\lambda_i-j-1} = \\
\displaystyle \sum^h_{s=0} (-1)^{s(\lambda_{i-2}-h+s)}c_{s,j} c_{\lambda_{i-1}-s,\lambda_i-j-1} c_{h-s,j+s} c_{\lambda_{i-2}-h+s,\lambda_{i-1}+\lambda_i-j-s-1}
\end{array}\]
must hold true.
This can be proved using a series of arithmetic manipulations and application of the shuffles' properties.

First, to simplify the notation, set $p = \lambda_{i-2}$, $q = \lambda_{i-1}$ and $r = \lambda_i - j -1$, then the identity becomes
\[ c_{p,q} c_{h,j} c_{p+q-h,r} = \displaystyle \sum^h_{s=0} (-1)^{s(p-h+s)}c_{s,j} c_{q-s,r} c_{h-s,j+s} c_{p-h+s,q-s+r}.\]
Applying the identity $c_{p,q} c_{p+q,r} = c_{q,r} c_{p,q+r}$ (a consequence of the associativity of $(p,q,r)$-shuffles) to the equation above yields
\[\begin{split}
    c_{p,q} c_{h,j} c_{p+q-h,r} & = \displaystyle \sum^h_{s=0} (-1)^{s(p-h+s)}(c_{s,j} c_{h-s,j+s}) (c_{q-s,r} c_{p-h+s,q-s+r}) \\
    & = \displaystyle \sum^h_{s=0} (-1)^{s(p-h+s)}(c_{h-s,s} c_{h,j}) (c_{p-h+s,q-s} c_{p+q-h,r}) \\
    & = c_{h,j} c_{p+q-h,r} \displaystyle \sum^h_{s=0} (-1)^{s(p-h+s)} c_{h-s,s} c_{p-h+s,q-s};
\end{split}\]
therefore it suffices to show that $c_{p,q} = \sum^h_{s=0} (-1)^{s(p-h+s)} c_{h-s,s} c_{p-h+s,q-s}$. We will prove this identity by induction on $h$.

First, recall that there is a bijection of sets $\mathrm{Sh}(p,q) \cong \mathrm{Sh}(p,q-1) \sqcup \mathrm{Sh}(p-1,q)$ (see, e.g., \cite{bae94,lod95}), which results in the shuffle identity $c_{p,q} = (-1)^p c_{p,q-1} + c_{p-1,q} = (-1)^q c_{p-1,q} + c_{p,q-1}$. The base cases are straightforward: for $h = 0$, $c_{p,q} = (-1)^0 c_{0,0} c_{p,q}$, while for $h = 1$, we recover the identity above $c_{p,q} = c_{1,0}c_{p-1,q} + (-1)^p c_{0,1}c_{p,q-1} = (-1)^p c_{p,q-1} + c_{p-1,q}$. Suppose it holds for $h$, then we may apply the identity in the following way:
\[\begin{split}
    c_{p,q} & = \displaystyle \sum^h_{s=0} (-1)^{s(p-h+s)} c_{h-s,s} c_{p-h+s,q-s} \\
    & = \displaystyle \sum^h_{s=0} (-1)^{s(p-h+s)} c_{h-s,s} [(-1)^{p-h+s}c_{p-h+s,q-s-1} + c_{p-h+s-1,q-s}] \\
    & = \displaystyle \sum^h_{s=0} (-1)^{s(p-h+s)} c_{h-s,s} [(-1)^{p-h+s}c_{p-(h+1)+(s+1),q-(s+1)} + c_{p-(h+1)+s,q-s}].
\end{split}\]
Thus $c_{p,q}$ can be written as a linear combination of terms of the form $c_{p-(h+1)+s,q-s}$ where $0 \le s \le h+1$. In the sum above, each term $c_{p-(h+1)+s,q-s}$ is derived from two terms $c_{p-h+s,q-s}$ and $c_{p-h+(s-1),q-(s-1)}$, hence its coefficient can computed to be
\[\begin{array}{c}
(-1)^{s(p-h+s)} c_{h-s,s} + (-1)^{(s-1)(p-h+s-1)} (-1)^{p-h+s-1} c_{h-s+1,s-1} = \\[5pt]
(-1)^{s(p-(h+1)+s)}[(-1)^s c_{h-s,s}+c_{h-s+1,s-1}] = (-1)^{s(p-(h+1)+s)}c_{(h+1)-s,s}.\end{array}\]
This completes our induction argument.
\end{proof}

By construction, the complex $D(n)_*$ is isomorphic to the relative cellular chain complex of Conf$_n (\mathbb{C}^\times) \cup \{ \infty \}$, relative to the point at infinity. In particular,
\[H_*(\mathrm{Conf}_n(\mathbb{C^\times}) \cup \{\infty\}, \{\infty\}) \cong H_*(D(n)_*).\]


\subsection{Cellular chain complex with local coefficients}

Let $L$ be a representation of $B_{n}$, and $\mathcal{L}$ be the associated local system over Conf$_{n}(\mathbb{C}^\times)$. Since $\mathcal{L}$ trivializes on the open cells of the Fox-Neuwirth stratification for Conf$_n (\mathbb{C}^\times)$, as graded groups the cellular chain complex with local coefficients $C_*(\mathrm{Conf}_*(\mathbb{C}^\times) \cup \{\infty\}, \{\infty\}; \mathcal{L})$ is isomorphic to $D(n)_* \otimes L$. To show their isomorphism as chain complexes, we need to study the differential; in particular, we must incorporate the braid action on $L$.

Recall that in the induced representation developed in Section~\ref{ssec:ind_rep}, we have a full set of representatives $\{ \alpha_i\}^{n+1}_{i=1}$ in the braid group $A_{n+1}$ for the left cosets of $B_n$. Define the map $\eta_i : A_{n+1} \to B_{n} \subset A_{n+1}$ by sending $a$ to $\alpha_{\underline{a}(i)}^{-1} a \alpha_i$, i.e. the image of $a$ is exactly the element $b_i \in B_n$ chosen in the proof of Proposition~\ref{prop:act_ind_rep}. This map is not a group homomorphism; however, observe that $\eta_i (ab) = \eta_{\underline{b}(i)}(a) \eta_i (b)$, as the result of the definition of this induced $A_{n+1}$-action. We may then define the differential of $D(n)_* \otimes L$ by
\[\begin{split}
    \displaystyle d[(\lambda, i, j) \otimes \ell)] & = \sum_{m = 1}^{i-2} (-1)^{m-1} \Bigg[ (\rho^m, i-1, j) \otimes \sum_{\gamma_m} (-1)^{|\gamma_m|} \eta_\iota(\widetilde{\gamma_m})(\ell) \Bigg] \\
    & + \sum_{m = i+1}^{q-n} (-1)^{m-1} \Bigg[ (\rho^m, i, j) \otimes \sum_{\gamma_m} (-1)^{|\gamma_m|} \eta_\iota(\widetilde{\gamma_m})(\ell) \Bigg] \\
    & + (-1)^{i-2} \sum_{h=0}^{\lambda_{i-1}} \Bigg[ (\rho^{i-1}, i-1, j+h) \otimes \sum_{\gamma_{i-1,h}} (-1)^{|\gamma_{i-1,h}|} \eta_\iota(\widetilde{\gamma_{i-1,h}})(\ell) \Bigg] \\
    & + (-1)^{i-1}\sum_{h=0}^{\lambda_{i+1}} \Bigg[ (\rho^i, i, j+h) \otimes \sum_{\gamma_{i,h}} (-1)^{|\gamma_{i,h}|} \eta_\iota(\widetilde{\gamma_{i,h}})(\ell) \Bigg]
\end{split} \]
where $\iota = j+1+\sum^{i-1}_{m = 1} \lambda_m$ is the overall position of the fixed point $O$ in the configuration $(\lambda,i,j)$; $\gamma_{i-1,h}$ runs over all $((\lambda_{i-1},h),\lambda_i, j)$-shuffles; $\gamma_{i,h}$ runs over all $(\lambda_i, (\lambda_{i+1},h),j)$-shuffles; and $\gamma_m$ runs over all $(\lambda_m, \lambda_{m+1})$-shuffles for all $m \ne i-1, i$. The lift $\widetilde{\gamma_m}$ (defined similarly for $\widetilde{\gamma_{i-1,h}}$ and $\widetilde{\gamma_{i,h}}$) in this differential is the lift of the shuffle $\gamma_m$ (as described in Section~\ref{ssec:bvs_qsa}) to the copy $A_{\lambda_m + \lambda_{m+1}} \le A_{n+1}$ consisting of braids that are only nontrivial on the $\lambda_m + \lambda_{m+1}$ strands starting with the $\lambda_1 + ... + \lambda_{m-1} +1^\mathrm{st}$.

\begin{thm}\label{thm:hom_FNF_Cx}
There is an isomorphism
\[H_* (\mathrm{Conf}_{n}(\mathbb{C}^\times) \cup \{ \infty \}, \{\infty\} ; \mathcal{L}) \cong H_*(D(n)_* \otimes L).\]
\end{thm}

\begin{proof}
Our argument will follow the outline of the proof of Theorem 4.3 in \cite{etw17}. Let $\widetilde{D(n)_*}$ be the cellular chain complex of the universal cover on $\mathrm{Conf}_n(\mathbb{C}^\times)$ obtained by lifting the Fox-Neuwirth cells. It suffices to describe an identification
\[ \widetilde{D(n)_q} \cong \mathbb{Z} \{ (((\lambda_1, ..., \lambda_{q-n+1}),i,j),b)|b\in B_n \} \]
as right $B_n$-representations which gives the desired description of the differentials.

The top dimensional cells of $\widetilde{D(n)_*}$ occur when $q = 2n$ and have the general form $(((1,...,1),i,0),b)$. Consider the codimension-1 faces of this cell obtained by combining the $m^\mathrm{th}$ and $m+1^\mathrm{st}$ columns, i.e. putting the $m^\mathrm{th}$ and $m+1^\mathrm{st}$ points on the same vertical line. There are two main outcomes of this operation: either the $m^\mathrm{th}$ point lies below the $m+1^\mathrm{st}$ point, or vice versa. Each of these are divided into subcases, depending on whether the fixed point $O$ is involved. Recall that for any configuration in $\mathrm{Conf}_n(\mathbb{C}^\times)$, the total order of points in the configuration is obtained by indexing them from bottom to top for each subsequent column starting with the leftmost one. We then label the braid element of a face based on its effect on the total order of points in the configuration: if the total order is preserved, we apply $\eta_i(\mathrm{id}) = \mathrm{id}$ to $b$ on the left; if it changes, we apply $\eta_i(\tilde{\gamma})$ where $\gamma$ is the corresponding permutation. The complete labelling system for these codimension-1 faces is as follows:
\begin{center}
\addtolength{\leftskip}{-1.2cm}
\addtolength{\rightskip}{-1.2cm}
{\renewcommand{\arraystretch}{1.5}%
\begin{tabular}{| c | c | c |}
\hline
 Case & $m^\mathrm{th}$ point below $m+1^\mathrm{st}$ point & $m^\mathrm{th}$ point above $m+1^\mathrm{st}$ point \\
 \hline
 $m < i-1$ & $(((1,...,1,2^{(m)},1,...,1),i-1,0),b)$ & $(((1,...,1,2^{(m)},1,...,1),i-1,0),\eta_i(\sigma_m) b)$ \\
 \hline
 $m = i-1$ & $(((1,...,1,2^{(i-1)},1,...,1),i-1,1),b)$ & $(((1,...,1,2^{(i-1)},1,...,1),i-1,0),\eta_i(\sigma_{i-1})b)$\\
 \hline
 $m = i$ & $(((1,...,1,2^{(i)},1,...,1),i,0),b)$ & $(((1,...,1,2^{(i)},1,...,1),i,1),\eta_i(\sigma_i)b)$\\
 \hline
 $m > i$ & $(((1,...,1,2^{(m)},1,...,1),i,0),b)$ & $(((1,...,1,2^{(m)},1,...,1),i,0),\eta_i(\sigma_m)b)$\\
 \hline
\end{tabular}}
\end{center}
where $(1,...,1,2^{(m)},1,...,1)$ denotes the composition of $n+1$ where the only non-1 part is $\lambda_m = 2$. Note that this choice of labelling is consistent with the right action of $B_n$.

More generally, the cell $((\lambda,i,j),b)$ corresponds to the face of $(((1,...,1),\iota,0),b)$ obtained by putting points into columns according to the configuration $\lambda$ while preserving the total order of points. Here the number $\iota = j+1+\sum_{m=1}^{i-1} \lambda_m$ is the overall position of $O$ in the configuration. However, if we arrange the face so that the total order is altered by a permutation $\gamma$, we need to multiply the element of $B_n$ in the cell's label on the left with $\eta_\iota(\tilde{\gamma})$. Note that this labelling system is compatible with the decomposition of braid elements into generators precisely because $\eta_i (ab) = \eta_{\underline{b}(i)}(a) \eta_i (b)$ for any $a,b \in A_{n+1}$.

It follows from this labelling system that the face maps of the complex $\widetilde{D(n)_*}$ are given by
\[d_m ((\lambda,i,j),b) = \begin{cases}
\displaystyle \sum_{\gamma_m}(-1)^{|\gamma_m|}((\rho^m,i-1,j),\eta_\iota(\widetilde{\gamma_m}) b) \hspace{3cm} \hfill m < i-1\\
\displaystyle \sum_{h=0}^{\lambda_{i-1}}\sum_{\gamma_{i-1,h}}(-1)^{|\gamma_{i-1,h}|}((\rho^{i-1},i-1,j+h),\eta_\iota(\widetilde{\gamma_{i-1,h}}) b) \hfill m = i-1\\
\displaystyle \sum_{h=0}^{\lambda_{i+1}}\sum_{\gamma_{i,h}}(-1)^{|\gamma_{i,h}|}((\rho^i,i,j+h),\eta_\iota(\widetilde{\gamma_{i,h}}) b) \hfill m = i\\
\displaystyle \sum_{\gamma_m}(-1)^{|\gamma_m|}((\rho^m,i,j),\eta_\iota(\widetilde{\gamma_m}) b) \hfill m > i.
\end{cases} \]
Notice that the sums over $h$ in the face maps $d_{i-1}$ and $d_i$ result from the fact that the number $j$ of points below the fixed point is a meaningful index of these cells. The signs in the formula come from the orientations of the cells. The differential $d : \widetilde{D(n)_q} \to \widetilde{D(n)_{q-1}}$ is given by the alternating sum of the face maps: $d = \sum_{m=1}^{q-n} (-1)^{m-1} d_m$.

Given a $B_n$-representation $L$, we want to give a description of the chain complex $\widetilde{D(n)_*} \otimes_{\mathbb{Z}B_n} L$ and its differential. Observe that we may identify $((\lambda,i,j),b) \otimes \ell$ with $((\lambda,i,j),1) \otimes b(\ell)$, hence there is a natural isomorphism of $k$-modules $\widetilde{D(n)_*} \otimes_{\mathbb{Z}B_n} L \cong D(n)_* \otimes L$. Furthermore, this identification when applied to the face maps and the differential results in the desired formula for the differential of $D(n)_* \otimes L$; for instance, when $m <i-1$,
\[\begin{split}
d_m(((\lambda,i,j),b) \otimes \ell) & = \displaystyle \sum_{\gamma_m}(-1)^{|\gamma_m|}((\rho_m,i-1,j),\eta_\iota(\widetilde{\gamma_m}) b) \otimes \ell\\
& = \displaystyle \sum_{\gamma_m}(-1)^{|\gamma_m|}((\rho_m,i-1,j),1) \otimes \eta_\iota(\widetilde{\gamma_m})b(\ell).
\end{split}\]
This concludes our proof of the theorem.
\end{proof}

This result gives another tool to compute the homology of the Artin groups of type \textit{B} with twisted coefficients in a similar manner as Corollary~\ref{cor:etw_rewrite}; however, the geometric construction of the chain complex $D(n)_* \otimes L$ offers a more intuitive view of the computation than the induced representation $\mathrm{Ind}^{A_{n+1}}_{B_n}(L)$ does.

\begin{cor}
There is an isomorphism
\[H_*(B_n;L) \cong H_{2n-*}(D(n)_* \otimes L^*)^*. \]
\end{cor}

\begin{proof}
Dualize over $k$ both sides of Theorem~\ref{thm:hom_FNF_Cx} with coefficients in $\mathcal{L}^*$, apply the universal coefficient theorem and Poincar\'e duality to the dual of the left side, and invoke the fact that $\pi_1 (\mathrm{Conf}_n (\mathbb{C}^\times)) = B_n$ to complete our proof.
\end{proof}

Interestingly, this result combined with Corollary~\ref{cor:etw_rewrite} shows that there is a quasi-isomorphism between the complexes $D(n)_* \otimes L^*$ and $C(n+1)_* \otimes \mathrm{Ind}^{A_{n+1}}_{B_n} (L^*)$. It turns out that these chain complexes are in fact isomorphic, which evidences a close relationship between the cellular stratification we constructed for $\mathrm{Conf}_n (\mathbb{C}^\times)$ and the induced representation.

\begin{prop}\label{prop:Dn=Cn}
For any $B_n$-representation $L$, there is an isomorphism of chain complexes
\[D(n)_* \otimes L \cong C(n+1)_{*+2} \otimes \mathrm{Ind}^{A_{n+1}}_{B_n} (L).\]
\end{prop}

\begin{proof}
Define a chain map $D(n)_* \otimes L \to C(n+1)_{*+2} \otimes \mathrm{Ind}^{A_{n+1}}_{B_n} (L)$ by sending $(\lambda,i,j) \otimes \ell$ to $\lambda \otimes \alpha_\iota (\ell)$ where $\iota = j+1+ \sum_{m=1}^{i-1}\lambda_m$. Conversely, given an element $\lambda \otimes t \in C(n+1)_{*+2} \otimes \mathrm{Ind}^{A_{n+1}}_{B_n} (L)$, observe that $t$ must be an element of $\alpha_\iota L$ for some $ 1\le \iota \le n+1$. Let $i$ be the largest integer such that $j' := \iota - \sum_{m=1}^{i-1} \lambda_m > 0$, then there is a chain map $C(n+1)_{*+2} \otimes \mathrm{Ind}^{A_{n+1}}_{B_n} (L) \to D(n)_* \otimes L$ that sends $\lambda \otimes t$ to $(\lambda,i,j'-1) \otimes \alpha^{-1}_\iota(t)$. These maps are evidently inverses as maps of graded $k$-modules, hence the chain complexes are isomorphic.
\end{proof}

There are two interesting consequences of this proposition. First, at this point, we have introduced two different methods to compute the homology of the Artin group $B_n$ with twisted coefficients in a $B_n$-representation $L$, by computing the homology of either complexes in Proposition~\ref{prop:Dn=Cn}. It may depend on the nature of the representation $L$ that working with the induced representation is more convenient than working with the more complicated complex $D(n)_*$ (e.g., when $L = V^{\otimes n} \otimes W$), and vice versa. Hence this result provides a flexible approach to the computation of the homology of the Artin groups of type \textit{B}. Secondly, Proposition~\ref{prop:Dn=Cn} demonstrates an example of the connection between the induced representation of braid groups and the topology of configuration spaces. In more generality, given a representation $L$ of a subgroup of the braid group, we may transform the information about the induced representation of $L$ into information about a cellular structure of a configuration space, whose fundamental group is the aforementioned subgroup.

\subsection{Geometric interpretation of the complex $F_*(\mathfrak{M},\mathfrak{I})$}

The purpose of this final subsection is to relate the homological algebra objects developed in Section~\ref{ssec:alg_req} with the Fox-Neuwirth cellular chain complex elaborated in the previous section. Recall that we constructed two algebraic objects to express the homology of type-\textit{B} Artin groups: the free resolution $F_*(M,A)$ of an $A$-bimodule $M$ and the $\mathfrak{A}$-bimodule $\mathfrak{M}$ over a quantum shuffle algebra $\mathfrak{A}$. These objects capture two key aspects of the cellular stratification of $\mathrm{Conf}_*(\mathbb{C}^\times)$: the configuration of the vertical columns and the configuration of points on the imaginary axis.

The $\mathfrak{A}$-bimodule $\mathfrak{M}$ provides an algebraic analogue of the imaginary axis in a configuration in Conf$_n(\mathbb{C}^\times)$. The single copy of $W$ in every summand $V^{\otimes j} \otimes W \otimes V^{\otimes q-j-1}$ of $\mathfrak{M}$ decorates the fixed point at the origin, while all other points in the configuration are decorated by $V$. Meanwhile, the chain complex $F_*(M,A)$ mirrors the structure of the vertical columns in the Fox-Neuwirth cellular stratification of Conf$_*(\mathbb{C}^\times)$. Given the choices of the algebra $A = \mathfrak{A}(V^*_\epsilon)$ and the bimodule $M = \mathfrak{M}(V^*_\epsilon, W^*)$, the single copy of $M$ in each summand of the graded module $F_*(M,A)$ represents the imaginary axis in the configuration, while column-combining operations are encoded in the multiplication of both the algebra and the bimodule. These observations offer a more intuitive perspective on the construction of our algebraic objects, which was not obvious from the first approach involving the induced representation.


\subsection*{Acknowledgments}

The author would like to acknowledge Craig Westerland for initiating the subject of study of this paper and helpful discussions. We also appreciate Calista Bernard for useful suggestions, and Alexander Voronov for providing insights about the Hochschild chain complex and Hochschild homology. Finally, we thank anonymous referees for helpful comments.

\bibliographystyle{alpha}
\bibliography{BraidRef}

\end{document}